\documentclass[12pt]{amsart}
\UseRawInputEncoding
\pdfoutput=1
\usepackage{fullpage}
\usepackage{graphicx}
\usepackage{todonotes}
\usepackage{tkz-euclide}

\usepackage{tikz}
\usetikzlibrary{decorations.markings}
\usepackage{tkz-euclide}

\newtheorem{theorem}{Theorem}[section]
\newtheorem{corollary}[theorem]{Corollary}
\newtheorem{lemma}[theorem]{Lemma}
\newtheorem{proposition}[theorem]{Proposition}

\newtheorem{remark}[theorem]{Remark}
\newtheorem{assumption}[theorem]{Assumption}

\numberwithin{equation}{section}

\newcommand\CC {{\mathbb C}}
\newcommand\DD {{\mathbb D}}
\newcommand\EE {{\mathbb E}}

\newcommand\HH {{\mathbb H}}

\newcommand\NN {{\mathbb N}}

\newcommand\QQ {{\mathbb Q}}
\newcommand\RR {{\mathbb R}}

\newcommand\ZZ {{\mathbb Z}}

\newcommand\sltwor{{\rm SL(2,\RR)}}
\newcommand\sltwoz{{\rm SL(2,\ZZ)}}
\newcommand\sltwoc{{\rm SL(2,\CC)}}

\newcommand\pslgroup{{\rm PSL}}
\newcommand\glgroup{{\rm GL}}
\newcommand\sugroup{{\rm SU}}
\newcommand\psugroup{{\rm PSU}}

\newcommand\bad{{\rm Bad}}

\newcommand\diameter{{\rm Diam }}
\newcommand\distance{{\rm Dist }}
\newcommand\domain{{\rm Dom }}

\newcommand\id{{\rm Id}}
\newcommand\im{{\rm Im }}
\newcommand\interior{{\rm Int}}

\newcommand\lipschitz{{\rm Lip}}

\newcommand\re{{\rm Re }}

\newcommand\spectrum{{\rm sp}}
\newcommand\support{{\rm Supp}}
\newcommand\variation{{\rm Var}}

\newcommand\cA{{\mathcal{A}  }}
\newcommand\cB{{\mathcal{B}  }}
\newcommand\cC{{\mathcal{C}  }}

\newcommand\cF{{\mathcal{F}  }}

\newcommand\cH{{\mathcal{H}  }}
\newcommand\cI{{\mathcal{I}  }}
\newcommand\cJ{{\mathcal{J}  }}

\newcommand\cL{{\mathcal{L}  }}
\newcommand\cM{{\mathcal{M}  }}
\newcommand\cN{{\mathcal{N}  }}

\newcommand\cP{{\mathcal{P}  }}

\newcommand\cR{{\mathcal{R}  }}
\newcommand\cS{{\mathcal{S}  }}

\newcommand\cU{{\mathcal{U}  }}
\newcommand\cV{{\mathcal{V}  }}
\newcommand\cW{{\mathcal{W}  }}

\begin{document}

\title{Transfer operators and dimension of bad sets for non-uniform Fuchsian lattices}

\author[L. Marchese]{Luca Marchese}

\address{Dipartimento di Matematica, Universit\`a di Bologna, Piazza di Porta San Donato 5, 40126, Bologna, Italia}

\email{luca.marchese4@unibo.it}





\begin{abstract}
The set of real numbers which are badly approximable by rationals admits an exhaustion by sets 
$\bad(\epsilon)$, whose dimension converges to $1$ as $\epsilon$ goes to zero. D. Hensley computed the  asymptotic for the dimension up to the first order in $\epsilon$, via an analogous estimate for the set of real numbers whose continued fraction has all entries uniformly bounded. We consider diophantine approximations by parabolic fixed points of any non-uniform lattice in $\pslgroup(2,\RR)$ and a geometric notion of $\epsilon$-badly approximable points. We compute the dimension of the set of such points up to the first order in $\epsilon$, via the thermodynamic method of Ruelle and Bowen. Geometric good approximations are related to a notion of bounded partial quotients for the Bowen-Series expansion. This gives a family of Cantor sets and associated quasi-compact transfer operators, with simple and positive maximal eigenvalue. Perturbative analysis of spectra applies. Our techniques only apply to non-uniform lattices admitting a finite index free subgroup satisfying a specific property. 
\end{abstract}

\maketitle


\section{Introduction and main statement}

From the theory of continued fractions we know that if $\alpha\in\RR\setminus\QQ$ is any irrational number then there exist infinitely many $p\in\ZZ$ and $q\in\NN^\ast$ with 
$$
\bigg|\alpha-\frac{p}{q}\bigg|
<
\frac{\epsilon_0}{q^2}
\quad
\textrm{ where }
\quad
\epsilon_0:=\frac{1}{\sqrt{5}}.
$$
The inequality above is sharp, indeed replacing $\epsilon_0$ by any $\epsilon$ with 
$
0<\epsilon<(\sqrt{5})^{-1}
$ 
and considering $\alpha:=(\sqrt{5}-1)/2$ one gets only finitely many solutions 
$p/q\in\QQ$. Thus for any $\epsilon>0$ small enough we have a non-empty set
$$
\bad(\epsilon):=
\left\{
\alpha\in\RR:\bigg|\alpha-\frac{p}{q}\bigg|
\geq
\frac{\epsilon}{q^2}
\quad\forall\quad\frac{p}{q}\in\QQ
\right\},
$$
where 
$
\bad(\epsilon')\subset\bad(\epsilon)
$ 
for $\epsilon'\geq\epsilon$. The union 
$
\bad:=\bigcup_{\epsilon>0}\bad(\epsilon)
$ 
is known as the set of \emph{badly approximable numbers}, and has full dimension in the real line. Denoting by $\dim_H(E)$ the \emph{Hausdorff dimension} of a set $E\subset\RR$ (see \S~\ref{SectionDimensionCantorSet}), we can measure how the size of $\bad(\epsilon)$ increases when $\epsilon\to0$. According to \cite{Hensley} we have
$$
\dim_H\big(\bad(\epsilon)\big)
=
1-\frac{6}{\pi^2}\cdot\epsilon+o(\epsilon).
$$
Moreover in \cite{Hensley} it is obtained the finer asymptotic up to order $O(n^{-2})$ of 
$\dim_H(E_n)$, where $E_n$ denotes the set of $\alpha\in\RR$ whose continued fraction 
$
\alpha=a_0+[a_1,a_2,\dots]
$ 
have entries $a_k\leq n$ for any $k\in\NN^\ast$. 

\smallskip

There are several natural generalizations of the set of badly approximable numbers and its exhaustion into sets $\bad(\epsilon)$, in particular for the euclidean space of any dimension and for systems of linear or affine forms. Other natural generalizations arise considering Kleinian groups acting on the boundary of the hyperbolic space, or the set of directions on a given \emph{translation surface}. 
In such cases, the terminology of \emph{bad sets} is somehow standard, as the notation $\bad(\epsilon)$ for the sets in the exhaustion. Generally, bad sets have full dimension, and more precisely they are \emph{thick}. For systems of affine forms this was proved in \cite{Kleinbock}. In \cite{Schmidt} it was proved that the set of badly approximable systems of linear forms is \emph{winning} for the so-called \emph{Schmidt's game}, a property which implies thickness. The full dimension result for real numbers was established by Jarn\'ik in 1929. In \cite{BeresnevichGhoshSimmonsVelani}, it is proved that for non-elementary geometrically finite Kleinian groups the set of badly approximable points has full dimension in the limit set of the group. This generalizes a previous result in \cite{Patterson} for Fuchsian group of the first kind. 
Further generalizations appear in ~\cite{DasFishmanSimmonsUrbanski}. 
Thickness of bad sets for directions on a given translation surface has been proved in \cite{KleinbockWeiss}. In~\cite{ChaikaCheungMasur} it was proved that the same sets of directions are  \emph{absolute winning}, a notion that was introduced previously in~\cite{McMullenAbsolute}. On the other hand, for $\epsilon>0$ sets 
$\bad(\epsilon)$ do not have full dimension, and the main problem appearing in the literature is to compute asymptotic formulas for their dimension.  The prototypes for such estimates are Hensley's Theorem mentioned above, and a previous result of Kurzweil (see \cite{Kurzweil}), which establishes the linear bound 
$$
0.25\cdot\epsilon\leq 1-\dim_H\big(\bad(\epsilon)\big)\leq 0.99\cdot\epsilon
$$ 
for badly approximable real numbers. In general, we can call \emph{Kurzweil's bound} a pair of upper and lower bounds for the difference between the dimension of the ambient space and 
$
\dim_H\big(\bad(\epsilon)\big)
$. 
In \cite{Weil} Kurzweil's bounds have been obtained for points in the euclidian space, and also in the hyperbolic space, for the action of some class of geometrically finite Kleinian groups. Similar bounds are obtained in \cite{BroderickKleinbock} for systems of linear forms and in \cite{MarcheseTrevinoWeil} for the set of directions on a given translation surface. The common aspect of all these results is that the Kurzweil's bound that they provide is not linear, but one inequality can only be obtained for some power $\epsilon^\beta$ with $0<\beta<1$. Finally, in \cite{Simmons}, the Kurzweil's bound for systems of linear forms is improved up to Hensley's first order estimate. Results and techniques related to those of this paper appear also in \cite{Kessebohmer}, \cite{Stadlbauer} and \cite{Sullivan}. This paper is devoted to the proof of Theorem~\ref{TheoremMainTheorem} below, which establishes Hensley's result for the action of non-uniform lattices in $\sltwor$ on the boundary of the upper half-plane. Theorem~\ref{TheoremMainTheorem} also gives positive answer to a question in \cite{MarcheseTrevinoWeil} concerning badly approximable directions on a \emph{Veech surface} (see Theorem~1.1 in \cite{MarcheseTrevinoWeil} and the comments after the statement).

\subsection{Notation}

Let $\sltwoc$ be the group of matrices 
\begin{equation}
\label{EquationCoefficientsSL(2,C)}
G=
\begin{pmatrix}
a & b \\
c & d  
\end{pmatrix}
\end{equation}
with $a,b,c,d\in\CC$ and $ad-bc=1$, where any such $G$ acts on points $z\in\CC\cup\{\infty\}$ by 
\begin{equation}
\label{EquationActionSL(2,C)}
G\cdot z:=\frac{az+b}{cz+d}.
\end{equation}
Denote $a=a(G)$, $b=b(G)$, $c=c(G)$ and $d=d(G)$ the coefficients of $G$ in  Equation~\eqref{EquationCoefficientsSL(2,C)}. The group $\sltwor$ of $G$ with coefficients 
$a,b,c,d$ in $\RR$ acts by isometries on the upper half-plane $\HH:=\{z\in\CC:\im(z)>0\}$, and inherits  a topology from the identification with the set of $(a,b,c,d)\in\RR^4$ with $ad-bc=1$. 
A \emph{Fuchsian group} is a discrete subgroup $\Gamma<\sltwor$. Referring to \S~\ref{SectionBackground}, we say that $\Gamma$ is a \emph{lattice} if it has a \emph{Dirichlet region} 
$\Omega\subset\HH$ with $\mu(\Omega)<+\infty$, where $\mu$ denotes the hyperbolic area. If $\Omega$ is not compact, then the lattice $\Gamma$ is said \emph{non-uniform}. In this case the intersection 
$
\overline{\Omega}\cap\partial\HH
$ 
is a finite non-empty set (recall that $\partial\HH=\RR\cup\{\infty\}$), whose elements are called the vertices \emph{at infinity} of $\Omega$. A point $z\in\RR\cup\{\infty\}$ is a parabolic fixed point for $\Gamma$ if there exists $P\in\Gamma$ parabolic with $P(z)=z$, where we recall that, in the notation of Equation~\eqref{EquationCoefficientsSL(2,C)}, $G\in\sltwoc$ is parabolic if $a+d=2$. 
Let $\cP_\Gamma$ be the set of parabolic fixed points of $\Gamma$, which is equal to the orbit under $\Gamma$ of the vertices at infinity of 
$\Omega$. The set $\cP_\Gamma$ is dense in $\RR$. Two points $z_1$ and $z_2$ in $\cP_\Gamma$ are \emph{equivalent} if $z_2=G\cdot z_1$ for some $G\in\Gamma$. Any non-uniform lattice $\Gamma$ has a finite number $p\geq1$ of equivalence classes 
$[z_1],\dots,[z_p]$ of parabolic fixed points, which are called the \emph{cusps} of $\Gamma$, and 
correspond to the punctures of the quotient surface $\Gamma\backslash\HH$.

\subsection{Main statement}

Classical diophantine approximations are related to the action of the \emph{modular group} $\sltwoz$, that is the subgroup of $G\in\sltwor$ with coefficients $a,b,c,d$ in $\ZZ$, referring to the notation of 
Equation~\eqref{EquationCoefficientsSL(2,C)}. The orbit $\sltwoz\cdot\infty$ is the entire set $\QQ$, which is the unique class of parabolic fixed points. For $\epsilon>0$ we have
$$
\alpha\in\bad(\epsilon)
\quad
\Leftrightarrow
\quad
\left|\alpha-G\cdot\infty\right|
\geq 
\frac{\epsilon}{c^{2}(G)}\quad\forall\quad G\in\sltwoz
\textrm{ with }
c(G)\not=0.
$$

In order to generalize this notion, let $\Gamma$ be a non-uniform lattice and $p\geq1$ be the number of its cusps. Fix a family 
$\cS=(A_1,\dots,A_p)$ of elements $A_k\in\sltwor$ such that the set of points 
$
\{z_1,\dots,z_p\}\subset\cP_\Gamma
$ 
defined by 
\begin{equation}
\label{EquationRepresentativesOfCusps}
z_k=A_k\cdot\infty
\quad
\textrm{ for }
\quad
k=1,\dots,p
\end{equation}
is a complete set of inequivalent parabolic fixed points for $\Gamma$. See 
Figure~\ref{FigureFundamentalDomainsUpperHalfPlane}. Any parabolic fixed point has the from 
$G\cdot z_k$ for some $G\in\Gamma$ and $k=1,\dots,p$. In the notation of 
Equation~\eqref{EquationCoefficientsSL(2,C)}, define the \emph{denominator} of $G\cdot z_k$ as 
\begin{equation}
\label{EquationDefinitionDenominator}
D(G\cdot z_k):=|c(GA_k)|.
\end{equation}
We have $G\cdot z_k=\infty$ if and only if $GA_k$ is an element of $\sltwor$ fixing $\infty$, that is 
$c(GA_k)=0$. The next equivalence follows
\begin{equation}
\label{EquationDenominatorsAreWellDefined(1)}
D(G\cdot z_k)=0
\quad
\Leftrightarrow
\quad
G\cdot z_k=\infty.
\end{equation}
Moreover, for any $G_1,G_2$ in $\Gamma$ and any $k,j$ in $\{1,\dots,p\}$, if 
$G_1\cdot z_k=G_2\cdot z_j$, then 
$$
k=j
\quad\text{ and }\quad
A_k^{-1}G_2^{-1}G_1A_k\cdot\infty=\infty,
$$
that is   
$
P:=A_k^{-1}G_2^{-1}G_1A_k
$ 
is an element of $\Gamma_k:=A_k^{-1}\Gamma A_k$ with $P\cdot\infty=\infty$. 
Corollary~\ref{CorollaryParabolicPointsOnlyFixedByParabolicElement} implies that $P$ is parabolic fixing $\infty$. Thus 
$
c(G_1A_k)=c(G_2A_kP)=c(G_2A_k)
$. 
It follows that the denominators are well defined, that is
\begin{equation}
\label{EquationDenominatorsAreWellDefined(2)}
G_1\cdot z_k=G_2\cdot z_j
\quad
\Rightarrow
\quad
D(G_1\cdot z_k)=D(G_2\cdot z_j).
\end{equation}

For $k=1,\dots,p$ the \emph{horoball} 
$
B_k:=A_k\big(\{z\in\HH:\im(z)>1\}\big)
$ 
is tangent to $\RR\cup\{\infty\}$ at $z_k$. Thus $G(B_k)$ is a ball tangent to the real line at 
$G\cdot z_k$ for any $G\in\Gamma$ with $D(G\cdot z_k)\not=0$. The denominators measure how the diameter of these balls shrinks to zero as $G$ varies in $\Gamma$. 
Equation~\eqref{EquationDiameterHoroball} gives 
$
D(G\cdot z_k)=1/\sqrt{\diameter\big(G(B_k)\big)}
$. 

\begin{remark}
By Equation~\eqref{EquationDenominatorsAreWellDefined(2)}, denominators do not change replacing the set $\cS$ by 
\begin{equation}
\label{EquationLeftMultiplicationDenominators}
\cS':=(G_1A_1,\dots,G_PA_p)
\quad
\textrm{ where }
\quad
G_k\in\Gamma
\quad
\textrm{ for }
\quad
k=1,\dots,p.
\end{equation}
The only freedom left in the choice of the set $\cS$ is to replace it by 
\begin{equation}
\label{EquationRightMultiplicationDenominators}
\cS'':=(A_1U_1,\dots,A_pU_p)
\quad
\textrm{ with }
\quad
U_k\in\cU 
\quad
\textrm{ for }
\quad
k=1,\dots,p,
\end{equation}
where $\cU$ denotes the \emph{upper triangular subgroup} of those $U\in\sltwor$ with $c(U)=0$, referring to the notation in Equation~\eqref{EquationCoefficientsSL(2,C)}. In this case the new denominators change by
$$
D''(G\cdot z_k)=\big|a(U_k)\big|\cdot D(G\cdot z_k)
\quad
\textrm{ for }
\quad
k=1,\dots,p.
$$
\end{remark}

\begin{figure}
\begin{center}
\begin{tikzpicture}[scale=2]

\begin{scope}

\clip(-1,-0.3) rectangle (3,2);

\draw (-1,0) -- (3,0);

\filldraw[fill=black!15!white, draw=black,thick] 
(1.5,2.5) -- (1.5,{sin(60)}) 
arc (120:180:1) 
arc (0:120:0.333) 
arc (60:180:0.333) 
arc (0:60:1) -- (-0.5,2.5) -- (1.5,2.5);

\filldraw[fill=black!30!white] 
(-1.1,1.5) -- (3.1,1.5) -- (3.1,3) -- (-1.1,3);

\filldraw[fill=black!30!white] (1,0.25) circle (0.25);

\filldraw[fill=black!30!white] (0,0.25) circle (0.25);

\draw[-,thick] 
(1.5,2.5) -- (1.5,{sin(60)}) 
arc (120:180:1) 
arc (0:120:0.333) 
arc (60:180:0.333) 
arc (0:60:1) -- (-0.5,2.5) -- (1.5,2.5);

\filldraw[fill=black!30!white] (2.5,0.15) circle (0.15);
\draw[-] (2.5,0.04) -- (2.5,-0.04);

\node[] at (0.5,1.9) [below] {$z_1:=\infty$};
\node[] at (2,1.9) [below] {$B_1$};
\node[] at (-0.2,0) [below] {$z_2:=A_2\cdot\infty$};
\node[] at (-0.4,0.4) {$B_2$};
\node[] at (1.2,0) [below] {$z_3:=A_3\cdot\infty$};
\node[] at (1.4,0.4) {$B_3$};
\node[] at (2.5,0) [below] {$G\cdot\infty$};
\node[] at (2.5,0.45) {$G(B_1)$};

\node[] at (0.5,1.2) {$\Omega$};

\end{scope}

\end{tikzpicture}
\end{center}
\caption{A fundamental domain $\Omega$ for a lattice $\Gamma$ and horoballs $B_1,B_2,B_3$ tangent to a complete set of inequivalent parabolic fixed points $z_1,z_2,z_3$.}
\label{FigureFundamentalDomainsUpperHalfPlane}
\end{figure}

According to Patterson (Theorem~1 in \S~7 of \cite{Patterson}), there exists $M=M(\Gamma,\cS)>0$, depending only on $\Gamma$ and $\cS$, such that for $Q>0$ big enough and any $\alpha\in\RR$ there exists a parabolic fixed point $G\cdot z_k\in\cP_\Gamma$ with
\begin{equation}
\label{EquationPattersonTheorem}
|\alpha-G\cdot z_k|
\leq
\frac{M}{D(G\cdot z_k)Q}
\quad
\textrm{ and }
\quad
0<D(G\cdot z_k)\leq Q.
\end{equation}
For $\Gamma=\sltwoz$, $\cS=\{\id\}$ and $M=1$, 
Equation~\eqref{EquationPattersonTheorem} gives the classical Dirichlet Theorem. 
In the general case, for any $\alpha\in\RR$ there exist infinitely many $G\cdot z_k\in\cP_\Gamma$ with  
$$
|\alpha-G\cdot z_k|
\leq
\frac{M}{D^2(G\cdot z_k)}.
$$
Moreover the constant $M(\Gamma,\cS)$ in the condition above cannot be replaced by an arbitrarily small constant. For any fixed $\epsilon>0$ small enough define $\bad(\Gamma,\cS,\epsilon)$ as the set of those $\alpha\in\RR$ such that 
\begin{equation}
\label{EquationDefinitionBadGammaEpsilon}
\left|\alpha-G\cdot z_k\right|
\geq
\frac{\epsilon}{D^2(G\cdot z_k)}
\quad\text{ for any }\quad
G\cdot z_k\in\cP_\Gamma
\quad\text{ with }\quad
D(G\cdot z_k)\not=0.
\end{equation}

The techniques in this paper require the existence of a subgroup $\Gamma_0<\Gamma$ with a specific property which is  stated in Assumption~\ref{AssumptionIdealFordDomain} below, namely a finite index free subgroup $\Gamma_0$ whose \emph{Ford domain} is an \emph{ideal polygon}, where the terminology is introduced in the next \S~\ref{SectionBackground}. We cannot establish if such $\Gamma_0$ always exists, nor we are able to provide counterexamples. In Appendix~\ref{AppendixExample} we introduce a class of Fuchsian groups $\Gamma$ satisfying Assumption~\ref{AssumptionIdealFordDomain} with $\Gamma_0=\Gamma$, where for any such group the quotient surface is a punctured sphere.

\begin{theorem}
\label{TheoremMainTheorem}
Let $\Gamma<\sltwor$ be a non-uniform lattice satisfying Assumption~\ref{AssumptionIdealFordDomain} below and let $\cS$ be a set as in Equation~\eqref{EquationRepresentativesOfCusps}. 
Then there exists a positive constant $\Theta=\Theta(\Gamma,\cS)>0$ such that for any $\epsilon>0$ small enough we have 
$$
\dim_H\big(\bad(\Gamma,\cS,\epsilon)\big)
=
1-\Theta\cdot\epsilon+o(\epsilon).
$$
\end{theorem}

\begin{remark}
The constant $\Theta(\Gamma,\cS)$ is given by Equation~\eqref{EquationFormulaConstantTheta} at the end of this paper, in terms of integrals with respect to a measure $\mu_{(1,\infty)}$, whose density is not computed explicitly. Nevertheless $\Theta(\Gamma,\cS)$ satisfies several invariance properties, derived a posteriori from Theorem~\ref{TheoremMainTheorem}. First of all, 
in \S~\ref{SectionBowenSeriesExpansion} we make some choice of a finite index free subgroup 
$\Gamma_0<\Gamma$, but $\Theta(\Gamma,\cS)$ does not depend on such choice. 
Furthermore the set $\bad(\Gamma,\cS,\epsilon)$ does not change replacing $\cS$ by 
$\cS'$ in 
Equation~\eqref{EquationLeftMultiplicationDenominators}, and thus 
$
\Theta(\Gamma,\cS')=\Theta(\Gamma,\cS)
$. 
On the other hand, if $\cS$ is replaced by $\cS''$ in 
Equation~\eqref{EquationRightMultiplicationDenominators} then $\Theta$ varies accordingly. For example, fixing $U\in\cU$ and setting $U_k=U$ for any $k$, we get 
$
\Theta(\Gamma,\cS'')=a(U)^{-2}\cdot\Theta(\Gamma,\cS)
$. 
The set  
$
\bigcup_{\epsilon>0}\bad(\Gamma,\cS,\epsilon)
$ 
of all badly approximable points does not depend on the choice of $\cS$ and has full dimension (see page 558 in~\cite{Patterson}). 
Once $\cS$ is fixed, we write simply 
$\bad(\Gamma,\epsilon)$ instead of $\bad(\Gamma,\cS,\epsilon)$.
\end{remark}

\subsection{Continued fractions and strategy of the proof}
\label{SectionStrategyOfTheProof}

This paper is devoted to the proof of Theorem~\ref{TheoremMainTheorem}. We sketch the strategy via an analogy with the classical case.

If $\alpha=a_0+[a_1,a_2,\dots]$ is the continued fraction of $\alpha$, where $a_0\in\ZZ$, $a_k\geq1$ for $k\geq1$, then the \emph{convergents} 
$
p_n/q_n:=a_0+[a_1,\dots,a_n]
$ 
satisfy the properties below.
\begin{enumerate}
\item
If $|\alpha-p/q|<(1/2)q^{-2}$ then $p/q=p_n/q_n$ for some $n\geq1$.
\item
For any $n\geq1$ we have 
$
\displaystyle{\frac{1}{2+a_{n+1}}\leq q_n^2\cdot|\alpha-p_n/q_n|\leq \frac{1}{a_{n+1}}.}
$
\end{enumerate}

It follows that $E_{N-2}\subset \bad(\epsilon)\subset E_N$ for any $\epsilon$, where 
$
N\leq\epsilon^{-1}<N+1
$ 
and where 
$
E_N=\{\alpha:a_k\leq N\,\forall\,k\geq1\}
$. 
Since $(N-2)^{-1}-N^{-1}=O(N^{-2})$, then the asymptotic 
$
\dim_H(E_N)=1-(6/\pi^2)\cdot N^{-1}+o( N^{-1})
$, 
which is proven in~\cite{Hensley}, gives the dimension of $\bad(\epsilon)$ up to order $o(\epsilon)$. By a nice formula of R. Bowen we have  $s_N=\dim_H(E_N)$ if and only if $\lambda(s_N,N)=1$, where 
$\lambda(s,N)$ is the simple and positive maximal eigenvalue of the transfer operators (acting on a proper space)
$$
f(x)\mapsto \sum_{k=1}^N\frac{1}{(k+x)^{2s}}f\bigg(\frac{1}{k+x}\bigg).
$$
Perturbative analysis of spectra gives $s_N$ up to the first order in $N^{-1}$. 
Roughly speaking, set $\epsilon:=N^{-1}$ and consider it as a continuous variable, so that the implicit function theorem for 
$
(s,\epsilon)\mapsto\lambda(s,\epsilon)
$ 
gives 
$$
\frac{ds(\epsilon)}{d\epsilon}\bigg|_{\epsilon=0}=
-\frac
{\partial_\epsilon\lambda(s=1,\epsilon=0)}
{\partial_s\lambda(s=1,\epsilon=0)}=
-\frac{6}{\pi^2}.
$$

The properties in Points (1) and (2) above are extended to the general case by 
Theorem~\ref{TheoremGoodApproximationsAndBoundaryExpansion} below (which also appears in~\cite{MarcheseExpansion}). 
Equation~\eqref{EquationDefinitionCuspidalAccelerationBowenSeriesMap} defines the acceleration $\cF$ of the \emph{Bowen-Series map}, which plays the role of the Gauss map. Smooth branches of 
$\cF$ are elements $F_W\in\Gamma$, labelled by \emph{cuspidal words} $W$, which belong to a countable alphabet $\cW$. 
Equation~\eqref{EquationDefinitionGeometricLenght} defines the \emph{geometric length} $|W|$ of a cuspidal word. Theorem~\ref{TheoremGoodApproximationsAndBoundaryExpansion} implies  
$
\EE_{T-\text{Const.}}\subset\bad(\Gamma,\epsilon)\subset\EE_T
$, 
where $T:=\epsilon^{-1}$ and $\EE_T$ is the Cantor set of points encoded by the finite subalphabet of those $W\in\cW$ with $|W|\leq T$. See Lemma~\ref{LemmaBadAndBoundedContinuedFraction}. 
Such sharp inclusions rely strongly on the definition of geometric length. This gives an approximation of 
$
\dim_H\big(\bad(\Gamma,\epsilon)\big)
$ 
by $\dim_H(\EE_T)$ up to order $o(\epsilon)$.

Theorem~\ref{TheoremTransferOperatorCircle} below concerns the transfer operators $L_{(s,T)}$ for the map $\cF$ and its restrictions to $\EE_T$, where we introduce an extra parameter $s>0$. These operators are quasi-compact with simple maximal eigenvalue $\lambda(s,T)>0$. Moreover Bowen formula holds, that is we have $s_T=\dim_H(\EE_T)$ if and only if $\lambda(s_T,T)=1$. 
Perturbative methods in \S~\ref{SectionPerturbativeEstimateMaximalEigenvalue} give $s_T$ up to order $o(T^{-1})$, via a sort of implicit function theorem for the function 
$(s,T)\mapsto\lambda(s,T)-1$.

\subsection*{Contents of this paper}

In \S~\ref{SectionBackground} we consider a non-uniform lattice $\Gamma$ satisfying Assumption~\ref{AssumptionIdealFordDomain}, that is admitting a finite index free subgroup $\Gamma_0$ whose \emph{Ford  domain} is an \emph{ideal polygon}, then we establish a convenient labelling of the sides of such domain. 
In \S~\ref{SectionBowenSeriesExpansion} we use the \emph{pairings} between sides of the Ford domain of $\Gamma_0$ to define the Bowen-Series expansion and its cuspidal acceleration. In \S~\ref{SectionGoodApproximationsAndBowenSeries} we prove Theorem~\ref{TheoremGoodApproximationsAndBoundaryExpansion}, which generalises Properties (1) and (2) stated in \S~\ref{SectionStrategyOfTheProof}. In \S~\ref{SectionProofMainTheorem} we state Theorem~\ref{TheoremDimensionShiftSpace} on the dimension of the Cantor sets $\EE_T$ of points encoded by cuspidal words with $|W|\leq T$, whose proof takes the most of this paper. Then we show that Theorem~\ref{TheoremMainTheorem} follows easily from Theorem~\ref{TheoremGoodApproximationsAndBoundaryExpansion} and Theorem~\ref{TheoremDimensionShiftSpace}.

In \S~\ref{SectionSubshiftFiniteType} we describe combinatorially the Cantor set $\EE_T$, proving in particular Proposition~\ref{PropositionAperiodicityTransitionMatrix} on \emph{aperiodicity} for the transition matrix of the underlying subshift. 
In \S~\ref{SectionEstimatesContractionDistortion} we prove estimates on contraction and distortion for the cuspidal acceleration map $\cF$. In \S~\ref{SectionTransferOperatorAndDimension} we recall 
Theorem~\ref{TheoremRuellePerronFrobenius} on transfer operators $\cL_{(s,T)}$ for aperiodic sub-shifts of finite type $\Sigma$, which provides \emph{Gibbs measures} $m_{(s,T)}$ on $\Sigma$. Bowen formula holds because the projection of $m_{(s,T)}$ to $\EE_T$ is \emph{Ahlfors regular} when $\cL_{(s,T)}$ has spectral radius $1$ (Proposition~\ref{PropositionDimensionAndSpectralRadius}). Estimates in 
\S~\ref{SectionEstimatesContractionDistortion} ensure that the logarithm of the derivative of $\cF$ is a good potential for $\cL_{(s,T)}$.

In \S~\ref{SectionTransferOperatorOnTheCircle} we state and prove 
Theorem~\ref{TheoremTransferOperatorCircle}. We define a proper Banach space $\cB$ of functions on the circle and consider the transfer operator $L_{(s,T)}$, which is quasi-compact and has the same simple maximal eigenvalue $\lambda(s,T)>0$ as $\cL_{(s,T)}$. Quasi-compactness follows from estimates in 
\S~\ref{SectionEstimatesContractionDistortion}. A main goal is to have all operators $L_{(s,T)}$ acting on the same space $\cB$, in particular $L_{(1,\infty)}$, which arises as a limit and does not have a correspondent $\cL_{(s,T)}$ at the level of the sub-shift. 
In \S~\ref{SectionPerturbativeEstimateMaximalEigenvalue} we prove 
Theorem~\ref{TheoremDimensionShiftSpace} via perturbative analysis of the spectra.

In \S~\ref{AppendixExample} we introduce a class of Fuchsian groups satisfying Assumption~\ref{AssumptionIdealFordDomain}. 
In \S~\ref{AppendixSeparation} we recall elementary properties of horoballs at parabolic fixed points. 
In \S~\ref{AppendixSpectralProjectors} we resume some basic properties of spectra and spectral projectors of bounded linear operators, in particular their stability.

\subsection*{Acknowledgements}

The author is grateful to R. Trevi\~{n}o and S. Weil for discussions on dimension of bad sets in the setting of translation surfaces, to M. Artigiani and C. Ulcigrai for discussions on the boundary expansion, and to S. Gou{\"e}zel for discussions on spectral properties of transfer operators. The author could extend his background on dimension of bad sets at the \emph{Arbeitsgemeinschaft} on Diophantine Approximation, Fractal Geometry and Dynamics at Oberwolfach, and he would like to thank V. Beresnevich and S. Velani for the invitation.

\section{Background}
\label{SectionBackground}

Consider also the unit disc $\DD:=\{z\in\CC;|z|<1\}$ model for the hyperbolic space, which is related to 
$\HH$ by the map $\varphi:\HH\to\DD$ defined by 
\begin{equation}
\label{EquationConjugationUpperHalfPlaneDisc}
\varphi(z):=\frac{z-i}{z+i}.
\end{equation}
The conjugate of $\sltwor$ under $\varphi$ is the group $\sugroup(1,1)$ of matrices 
$
F\in\glgroup(2,\CC)
$ 
with 
\begin{equation}
\label{EquationCoefficientsSU(1,1)}
F=
\begin{pmatrix}
\alpha & \overline{\beta} \\
\beta & \overline{\alpha}  
\end{pmatrix}
\quad
\textrm{ with }
\quad
|\alpha|^2-|\beta|^2=1.
\end{equation}
Denote $\alpha=\alpha(F)$ and $\beta=\beta(F)$ the coefficients of $F\in\sugroup(1,1)$ as in Equation~\eqref{EquationCoefficientsSU(1,1)}. In this section we consider Fuchsian groups $\Gamma\subset\sugroup(1,1)$ satisfying Assumption \ref{AssumptionIdealFordDomain} below, which provides the existence of a finite index free subgroup $\Gamma_0$ of $\Gamma$ whose \emph{Ford domain} $\Omega_\DD$ is an \emph{ideal polygon}. Such choice of $\Gamma_0$ and $\Omega_\DD$ is fixed and used in the rest of the paper.

\subsection{Isometric circles}
\label{SectionIsometricCircles}

Consider $F\in\sugroup(1,1)$ and $\alpha=\alpha(F)$, $\beta=\beta(F)$ as in 
Equation~\eqref{EquationCoefficientsSU(1,1)}. Assume $\beta\not=0$ and let 
$
\omega_F:=-\overline{\alpha}/\beta
$ 
be the pole of $F$. The \emph{isometric circle} $I_F$ of $F$ is the euclidian circle centered at 
$\omega_F$ with euclidian radius $\rho(F):=|\beta|^{-1}$, that is 
$$
I_F:=\{\xi\in\CC:|\xi-\omega_F|=|\beta|^{-1}\}=\{\xi\in\CC:|D_zF|=1\}.
$$
We have $F(I_F)=I_{F^{-1}}$ with $\rho(F)=\rho(F^{-1})$ and $|\omega_{F^{-1}}|=|\omega_F|$, by  Theorem~3.3.2 in \cite{KatokFuchsian}. 
Moreover $I_F\cap\DD$ is a geodesic of $\DD$ for any $F\in\sugroup(1,1)$, by Theorem~3.3.3 in \cite{KatokFuchsian}. For any $F\in\sugroup(1,1)$ denote $U_F$ the disc in $\CC$ with 
$\partial U_F=I_F$, that is the interior of $I_F$. We have $|D_\xi F|<1$ for any 
$\xi\in \CC\setminus \overline{U_F}$, that is $F$ contracts (in weak sense) the euclidian metric. 
Fix $F,G$ in $\sugroup(1,1)$ and $z\in\DD$. Observe that 
$
z\in G^{-1}(U_F)\Leftrightarrow G(z)\in U_F
$ 
and 
$
|D_zFG|=|D_{G(z)}F|\cdot|D_zG|
$, 
therefore 
$$
G^{-1}(U_F)\cap U_{G}
\subset
U_{FG}
\subset
U_G\cup G^{-1}(U_F).
$$
If $U_F\cap U_{G^{-1}}=\emptyset$, which is equivalent to $G^{-1}(U_F)\subset U_G$, then we have
\begin{equation}
\label{EquationIsometricCircleComposition}
G^{-1}(U_F)\subset U_{FG}\subset U_G.
\end{equation}

\subsection{Dirichlet regions and Ford domain}
\label{SectionPropertiesDirichletRegion}

Let $\Gamma\subset\sugroup(1,1)$ be a Fuchsian group. Recall that a \emph{Dirichlet region} 
$\Omega$ for $\Gamma$ is a closed subset of $\DD$ with non-empty interior $\interior(\Omega)$, whose boundary $\partial\Omega$ is the countable union of segments $s$ of hyperbolic geodesic. The vertices of $\Omega$ are the endpoints of segments $s$ in $\partial\Omega$. By Theorem~3.2.2 and Theorem~3.5.1 in \cite{KatokFuchsian} we have the properties below.
\begin{enumerate}
\item
We have 
$
\DD=\bigcup_{F\in\Gamma}F(\Omega)
$ 
with 
$
\interior(\Omega)\cap F(\interior(\Omega))=\emptyset
$ 
for any $F\in\Gamma\setminus\{\id\}$.
\item
For any compact $K\subset\DD$ we have $K\cap F(\Omega)\not=\emptyset$ only for finitely many 
$F\in\Gamma$.
\item
According to \S~3.5 in \cite{KatokFuchsian}, for any $z\in\partial\Omega$ there is $F\in\Gamma\setminus\{\id\}$ with 
$F\cdot z\in\partial\Omega$. If $z$ is a vertex of $\Omega$ then also $F\cdot z$ is. The set of vertices of $\Omega$ decomposes into equivalence classes under the action of $\Gamma$. Any equivalence class is finite.
\item
If $s$ is a side of $\Omega$ then there exists an unique $F\in\Gamma\setminus\{\id\}$ such that 
$\widehat{s}:=F(s)$ is a side of $\Omega$. We say that $s$ and $\widehat{s}$ are paired sides. 
\item
The set of pairings generates $\Gamma$ (Theorem~3.5.4 in \cite{KatokFuchsian}).
\item
If $\mu(\Omega)<+\infty$ then $\Omega$ has finitely many sides 
(Theorem~4.1.1 in \cite{KatokFuchsian}), so that $\Gamma$ is finitely generated.
\item
If $\Omega$ is non-compact with $\mu(\Omega)<+\infty$ then the set of parabolic fixed points of $\Gamma$ is the orbit 
$
\Gamma\big(\Omega\cap\partial\DD\big)
$ 
of vertices at infinity of $\Omega$ (Theorem~4.2.5 in \cite{KatokFuchsian}).
\end{enumerate}

Finally assume that the origin $0\in\DD$ is not fixed by any $F\in\Gamma\setminus\{\id\}$, which in the notation of  
Equation~\eqref{EquationCoefficientsSU(1,1)} is equivalent to  
$\beta(F)\not=0$ for any $F\in\Gamma\setminus\{\id\}$, so that the isometric circle $I_F$ and the disc $U_F$ are defined. 
\begin{enumerate}
\setcounter{enumi}{7}  
\item
Under the last condition we can define the \emph{Ford domain} of $\Gamma$ as
\begin{equation}
\label{EquationDefinitionFordDomain}
\Omega:=\overline{\DD\setminus\bigcup_{F\in\Gamma\setminus\{\id\}}U_F}
\end{equation}
and Theorem 3.3.5 in \cite{KatokFuchsian} implies that $\Omega$ is a Dirichlet region for $\Gamma$. 
\end{enumerate}

\subsection{Labelled ideal polygon}
\label{SectionLabelledIdealPolygon}

Let $\Gamma$ be a Fuchsian group with $\mu(\Omega)<+\infty$. A Dirichlet region $\Omega$ for $\Gamma$ is an \emph{ideal polygon} if, as an hyperbolic polygon, it has all its vertices on the boundary $\partial\DD$. According to~\cite{Tukia} (the Proposition at page 15), there exist a finite index free subgroup $\Gamma_0<\Gamma$ whose fundamental domain is an ideal polygon, but such domain does not necessarily coincide with the ford domain of $\Gamma_0$. Therefore we need to introduce the following assumption.

\begin{assumption}
\label{AssumptionIdealFordDomain}
Assume that there exists a finite index free subgroup $\Gamma_0$ of $\Gamma$ whose \emph{Ford domain} $\Omega_\DD$, which is defined by Equation \eqref{EquationDefinitionFordDomain}, is an \emph{ideal polygon}.
\end{assumption}

\begin{remark}
Appendix~\ref{AppendixExample} introduces a class of groups $\Gamma$ satisfying Assumption~\ref{AssumptionIdealFordDomain}, and for all such group the quotient surface is a punctured sphere. We cannot establish if Assumption~\ref{AssumptionIdealFordDomain} is satisfied by any non uniform Fuchsian lattice.
\end{remark}

Let $\Gamma\subset\sugroup(1,1)$ be a non-uniform lattice, fix a finite index free subgroup 
$\Gamma_0<\Gamma$ as in Assumption \ref{AssumptionIdealFordDomain} and let $\Omega_\DD$ be the ford domain of $\Gamma_0$, which is an ideal polygon. Point (8) in \S~\ref{SectionPropertiesDirichletRegion} implies that $\Omega_\DD$ is a Dirichlet region for $\Gamma_0$. The discussion in \S~\ref{SectionPropertiesDirichletRegion} implies that 
$\Omega_\DD$ has finitely many sides. Moreover there is an even number $2d$ of sides, by Point (4) in \S~\ref{SectionPropertiesDirichletRegion}. All the $2d$ vertices of $\Omega_\DD$ belong to $\partial\DD$ by Assumption \ref{AssumptionIdealFordDomain}. Any side $s$ is a complete geodesic in $\DD$ and for any such $s$ there exists an unique $F\in\Gamma_0$ such that $F(s)$ is another side of $\Omega_0$ with $F(s)\not=s$. The sides $s$ and $F(s)$ are thus paired. By Point (5) in \S~\ref{SectionPropertiesDirichletRegion} the set of $d$ pairings arising in this way generates $\Gamma_0$. See Figure~\ref{FigureLabelledIdealPolygon}. 
Consider two finite alphabets $\cA_0$ and $\widehat{\cA_0}$, both with $d$ elements, and a map 
$$
\iota:\cA_0\cup\widehat{\cA_0}\to\cA_0\cup\widehat{\cA_0}
\quad
\textrm{ with }
\quad
\iota^2=\id
\quad
\textrm{ and }
\quad
\iota(\cA_0)=\widehat{\cA_0},
$$
that is an involution of $\cA_0\cup\widehat{\cA_0}$ which exchanges $\cA_0$ with $\widehat{\cA_0}$. 
Set $\cA:=\cA_0\cup\widehat{\cA_0}$ and for any $a\in\cA$, denote $\widehat{a}:=\iota(a)$.

\begin{figure}
\begin{center}
\begin{tikzpicture}[scale=2.3]

\begin{scope}

\tkzDefPoint(0,0){O}
\tkzDefPoint(1,0){A}
\tkzDrawCircle[thick,fill=black!15](O,A)
\tkzClipCircle(O,A)



\tkzDefPoint({cos(0)},{sin(0)}){z0}
\tkzDefPoint({cos(0.2*pi)},{sin(0.2*pi)}){z1}
\tkzDefPoint({cos(0.5*pi)},{sin(0.5*pi)}){z2}
\tkzDefPoint({cos(0.7*pi)},{sin(0.7*pi)}){z3}
\tkzDefPoint({cos(1.0*pi)},{sin(1.0*pi)}){z4}
\tkzDefPoint({cos(1.3*pi)},{sin(1.3*pi)}){z5}
\tkzDefPoint({cos(1.5*pi)},{sin(1.5*pi)}){z6}
\tkzDefPoint({cos(1.7*pi)},{sin(1.7*pi)}){z7}

\tkzDefCircle[orthogonal through=z0 and z1](O,A)
\tkzGetPoint{tmp}
\tkzDrawCircle[fill=white](tmp,z1)
\node[] at ({0.85*cos(17)},{0.85*sin(17)}) {$s_a$};
\tkzDefCircle[orthogonal through=z1 and z2](O,A)
\tkzGetPoint{tmp}
\tkzDrawCircle[fill=white](tmp,z2)
\node[] at ({0.8*cos(65)},{0.8*sin(65)}) {$s_b$};
\tkzDefCircle[orthogonal through=z2 and z3](O,A)
\tkzGetPoint{tmp}
\tkzDrawCircle[fill=white](tmp,z3)
\node[] at ({0.85*cos(110)},{0.85*sin(110)}) {$s_{\widehat{a}}$};
\tkzDefCircle[orthogonal through=z3 and z4](O,A)
\tkzGetPoint{tmp}
\tkzDrawCircle[fill=white](tmp,z4)
\node[] at ({0.8*cos(150)},{0.8*sin(150)}) {$s_{\widehat{b}}$};
\tkzDefCircle[orthogonal through=z4 and z5](O,A)
\tkzGetPoint{tmp}
\tkzDrawCircle[fill=white](tmp,z5)
\node[] at ({0.8*cos(210)},{0.8*sin(210)}) {$s_c$};
\tkzDefCircle[orthogonal through=z5 and z6](O,A)
\tkzGetPoint{tmp}
\tkzDrawCircle[fill=white](tmp,z6)
\node[] at ({0.85*cos(255)},{0.85*sin(255)}) {$s_d$};
\tkzDefCircle[orthogonal through=z6 and z7](O,A)
\tkzGetPoint{tmp}
\tkzDrawCircle[fill=white](tmp,z7)
\node[] at ({0.88*cos(288)},{0.88*sin(288)}) {$s_{\widehat{d}}$};
\tkzDefCircle[orthogonal through=z7 and z0](O,A)
\tkzGetPoint{tmp}
\tkzDrawCircle[fill=white](tmp,z0)
\node[] at ({0.75*cos(330)},{0.75*sin(330)}) {$s_{\widehat{c}}$};

\draw[->,thick] ({0.7*cos(255)},{0.7*sin(255)}) .. controls 
({0.6*cos(265)},{0.6*sin(265)}) and ({0.6*cos(275)},{0.6*sin(275)}) 
.. ({0.7*cos(285)},{0.7*sin(285)}) {};

\node[] at ({0.65*cos(270)},{0.65*sin(270)}) [above] {$F_{\widehat{d}}$};

\draw[->,thick] ({0.5*cos(145)},{0.5*sin(145)}) .. controls 
({0.3*cos(120)},{0.3*sin(120)}) and ({0.3*cos(90)},{0.3*sin(90)}) 
.. ({0.5*cos(65)},{0.5*sin(65)}) {};

\node[] at ({0.55*cos(100)},{0.55*sin(100)}) [right,below] {$F_{b}$};

\end{scope}

\begin{scope}

\tkzDefPoint(0,0){O}
\tkzDefPoint(1,0){A}
\tkzDrawCircle(O,A)

\node[] at ({cos(0)},{sin(0)}) [right] {$\xi^L_{\widehat{c}}=\xi^R_a$};
\node[] at ({cos(40)},{sin(40)}) [right] {$\xi^L_a=\xi^R_b$};
\node[] at ({cos(90)},{sin(90)}) [above] {$\xi^L_b$};
\node[] at ({cos(130)},{sin(130)}) [above] {$\xi^R_{\widehat{b}}$};
\node[] at ({cos(180)},{sin(180)}) [left] {$\xi^L_{\widehat{b}}$};
\node[] at ({cos(310)},{sin(310)}) [right,below] {$\xi^R_{\widehat{c}}$};

\node[] at ({1.15*cos(17)},{1.15*sin(17)}) {$[a]_\DD$};
\node[] at ({1.15*cos(60)},{1.15*sin(60)}) {$[b]_\DD$};
\node[] at ({1.15*cos(150)},{1.15*sin(150)}) {$[\widehat{b}]_\DD$};

\end{scope}
\end{tikzpicture}

\end{center}
\caption{Ideal polygon labelled by 
$
\cA=\{a,b,c,d,\widehat{a},\widehat{b},\widehat{c},\widehat{d}\}
$, 
with some examples of the notation in \S~\ref{SectionTheBoundaryMap}.}
\label{FigureLabelledIdealPolygon}
\end{figure}

\begin{itemize}
\item
Label the sides of $\Omega_\DD$ by the letters in $\cA$, so that for any 
$
a\in\cA
$ 
the sides $s_a$ and $s_{\widehat{a}}$ are those which are paired by the action of $\Gamma_0$.  
\item
For any pair of sides $s_a$ and $s_{\widehat{a}}$ as above, let $F_a$ be the unique element of 
$\Gamma_0$ such that
\begin{equation}
\label{EquationPairingPolygonSides}
F_a(s_{\widehat{a}})=s_a.
\end{equation}
\item
For any $a\in\cA$ we have $F_{\widehat{a}}=F_a^{-1}$, and the latter form a set of generators for $\Gamma_0$. 
\end{itemize}

Denote  
$
\Omega_\HH:=\varphi^{-1}(\Omega_\DD)\subset\HH
$ 
the preimage of $\Omega_\DD$ under the map in 
Equation~\eqref{EquationConjugationUpperHalfPlaneDisc}.

\section{The Bowen-Series expansion}
\label{SectionBowenSeriesExpansion}

We follow \cite{ArtigianiMarcheseUlcigrai(Veech)} and \cite{ArtigianiMarcheseUlcigrai(Fuchsian)}. The original construction is in \cite{BowenSeries}. 
Let $\Gamma<\psugroup(1,1)$ be a non-uniform lattice and $\Gamma_0<\Gamma$ be a finite index free subgroup. Let $\Omega_\DD$ be an ideal polygon labelled by letters in $\cA$, as in \S~\ref{SectionLabelledIdealPolygon}.

\subsection{The boundary map}
\label{SectionTheBoundaryMap}

Consider arcs $J\subset\partial\DD$ with $J\not=\partial\DD$ and parametrize them by $t\mapsto e^{-it}$ with $t\in(x,y)$. Set $\inf J:=e^{-ix}$ and $\sup J:=e^{-iy}$. We say that $J$ is \emph{right open} if $\inf J\in J$ and $\sup J\not\in J$. 

\smallskip

For $a\in\cA$ let $F_a$ be the map in Equation~\eqref{EquationPairingPolygonSides}. 
Let $I_{F_a}$ be the isometric circle of $F_a$ and $U_{F_a}$ be its interior, as in \S~\ref{SectionIsometricCircles}. 
Recall that  
$
s_{\widehat{a}}=I_{F_a}\cap\DD
$ 
and 
$
s_{a}=I_{F_{\widehat{a}}}\cap\DD
$. 
Let $[a]_\DD$ be the right open arc of $\partial\DD$ cut by the side $s_a$, that is
$$
[a]_\DD:=U_{F_{\widehat{a}}}\cap\partial\DD.
$$
Set $\xi_a^L:=\inf [a]_\DD$ and $\xi_a^R:=\sup [a]_\DD$. 
Figure~\ref{FigureLabelledIdealPolygon} shows examples of such notation. In order to take account of the cyclic order in $\partial\DD$ of the arcs $[a]_\DD$ above, fix $a_0\in\cA$ and define a map 
$o:\cA\to\ZZ/2d\ZZ$ setting $o(a_0):=0$ and
\begin{equation}
\label{EquationCyclicOrderLetters}
o(b)=o(a)+1 \mod2d
\quad
\textrm{ for }
\quad
a,b\in\cA
\quad
\textrm{ with }
\quad
\xi^R_a=\xi^L_b.
\end{equation}

We have $F_a(I_{F_a})=I_{F_{\widehat{a}}}$ for any $a\in\cA$, 
thus $F_a$ sends the complement of $[\widehat{a}]_\DD$ to $[a]_\DD$, that is 
\begin{equation}
\label{EquationActionGeneratorsBoundary}
F_a\big(\partial\DD\setminus [\widehat{a}]_\DD\big)=[a]_\DD.
\end{equation}
The \emph{Bowen-Series map} is the map $\cB\cS:\partial\DD\to\partial\DD$ defined by
\begin{equation}
\label{EquationBowenSeriesMap}
\cB\cS(\xi):=F_a^{-1}\cdot \xi
\quad
\textrm{iff}
\quad
\xi\in [a]_\DD. 
\end{equation}
Since the arcs $[a]_\DD$ are right open, then $\cB\cS$ is left continuous. 
The \emph{boundary expansion} of  a point $\xi\in\partial\DD$ is the sequence $(a_k)_{k\in\NN}$ of letters $a_k\in\cA$ with 
\begin{equation}
\label{EquationDefinitionBoundaryExpansion}
\cB\cS^k(\xi)\in [a_k]_\DD
\quad
\textrm{ for any }
k\in \NN.
\end{equation}
By Equation~\eqref{EquationActionGeneratorsBoundary}, any such sequence satisfies the so-called \emph{no backtracking Condition}:
\begin{equation}
\label{EqNobacktrack}
a_{k+1}\not=\widehat{a_k}
\textrm{ for any }
k\in\NN.
\end{equation}

Any finite word $(a_0,\dots,a_n)$ satisfying Condition~\eqref{EqNobacktrack} corresponds to a \emph{factor} of the map $\cB\cS:\partial\DD\to\partial\DD$, that is a finite concatenation  
$
F_{a_n}^{-1}\circ\dots\circ F_{a_0}^{-1}
$ 
arising from iterations of $\cB\cS$. We call \emph{admissible word}, or simply \emph{word}, any finite or infinite word in the letters of $\cA$ satisfying 
Condition~\eqref{EqNobacktrack}. For any finite word 
$(a_0,\dots,a_n)$ we use the notation
$$
F_{a_0,\dots,a_n}:=F_{a_0}\circ\dots\circ F_{a_n}\in \Gamma_0.
$$

Define the right open arc 
$
[a_0,\dots,a_n]_\DD
$ 
as the set of $\xi\in\partial\DD$ such that $\cB\cS^k(\xi)\in[a_k]_\DD$ for any $k=0,\dots,n$, that is
\begin{equation}
\label{EquationDefinitionCylinderBowenSeries}
[a_0,\dots,a_n]_\DD
:=
F_{a_0,\dots,a_{n-1}}[a_n]_\DD
=
F_{a_0,\dots,a_n}\big(\partial\DD\setminus[\widehat{a_n}]_\DD\big).
\end{equation}
Two such arcs satisfy    
$
[a_0,\dots,a_n]_\DD\subset [b_0,\dots,b_m]_\DD
$ 
if and only if $m\geq n$ and $a_k=b_k$ for $k=0,\dots,n$. 

\begin{lemma}
The arcs $[a_0,\dots,a_n]_\DD$ shrink to a point as $n\to\infty$.
\end{lemma}

\begin{proof}
Let
$
U_{a_0,\dots,a_n}:=U_{F_{a_0,\dots,a_n}}
$ 
be the interior of the isometric circle of $F_{a_0,\dots,a_n}$. 
Equation~\eqref{EquationIsometricCircleComposition} and an easy inductive argument imply that for any $n$ we have $U_{a_0,\dots,a_n}\subset U_{a_n}$ and thus 
$
U_{a_0,a_1,\dots,a_n}\subset U_{a_1,\dots,a_n}
$. 
The last inclusion and Equation~\eqref{EquationIsometricCircleComposition} imply
$$
F_{a_0,\dots,a_{n-1}}[a_n]_\DD
\subset
F_{a_0,\dots,a_{n-1}}(U_{\widehat{a_n}})
=
F_{\widehat{a_{n-1}},\dots,\widehat{a_0}}^{-1}(U_{\widehat{a_n}})
\subset
U_{\widehat{a_n},\widehat{a_{n-1}},\dots,\widehat{a_0}}.
$$
The Lemma follows because the radii of isometric circles of an infinite sequence of elements in a Fuchsian group shrink to zero. See Theorem~3.3.7 in~\cite{KatokFuchsian}.
\end{proof}

Any sequence $(a_k)_{k\in\NN}$ satisfying Condition~\eqref{EqNobacktrack} corresponds to a point 
$
\xi=[a_0,a_1,\dots]_\DD
$ 
in $\partial\DD$, where we use the notation   
$$
[a_0,a_1,\dots]_\DD:=\bigcap_{n\in\NN}[a_0\dots,a_n]_\DD.
$$
Conversely, if $(a_k)_{k\in\NN}$ is the boundary expansion of $\xi\in\partial\DD$, then 
$\xi=[a_0,a_1,\dots]_\DD$. Since
$
[a_0,a_1,\dots]_\DD\in [a_0]_\DD
$, 
where $\cB\cS$ acts by $F_{a_0}^{-1}$, then
$$
\cB\cS\big([a_0,a_1,\dots]_\DD\big)
=
F_{a_0}^{-1}\big([a_0,a_1,\dots]_\DD\big)
=
[a_1,a_2,\dots]_\DD,
$$
that is the map $\cB\cS$ is the shift on the space of admissible infinite words.

\subsection{Cuspidal words}
\label{SectionCuspidalWords}

Consider the map $o:\cA\to\ZZ/2d\ZZ$ in Equation~\eqref{EquationCyclicOrderLetters}.

\begin{lemma}
\label{LemmaCombinatorialPropertiesCuspidal}
Fix $n$ and a word $(a_0,\dots,a_n)$ satisfying Condition~\eqref{EqNobacktrack} with $n\geq1$ and $a_0=a_n$. The map $F_{a_0,\dots,a_{n-1}}$ is a parabolic element of $\Gamma_0$ fixing $\xi^R_{a_0}$ if and only if 
\begin{equation}
\label{EquationRightCuspidalWord(-)}
o(a_{k+1})=o(\widehat{a_k})-1
\quad
\textrm{ for any }
\quad
k=0,\dots,n-1.
\end{equation}
The map $F_{a_0,\dots,a_{n-1}}$ is a parabolic element of $\Gamma_0$ fixing $\xi^L_{a_0}$ if and only if 
\begin{equation}
\label{EquationLeftCuspidalWord(+)}
o(a_{k+1})=o(\widehat{a_k})+1
\quad
\textrm{ for any }
\quad
k=0,\dots,n-1.
\end{equation}
\end{lemma}

\begin{proof}
Observe that $F_a\cdot \xi^R_{\widehat{a}}=\xi^L_a$ and 
$F_a\cdot \xi^L_{\widehat{a}}=\xi^R_a$ for any $a\in\cA$. 
Equation~\eqref{EquationRightCuspidalWord(-)} implies $\xi^R_{a_{k+1}}=\xi^L_{\widehat{a_k}}$ and therefore 
$$
\xi^R_{a_0}
=
F_{a_0}\cdot \xi^L_{\widehat{a_0}}
=
F_{a_0}\cdot \xi^R_{a_1}
=
\dots
=
F_{a_0,\dots,a_{n-1}}\cdot \xi^R_{a_n}
=
F_{a_0,\dots,a_{n-1}}\cdot \xi^R_{a_0}.
$$
Corollary~\ref{CorollaryParabolicPointsOnlyFixedByParabolicElement} implies that 
$F_{a_0,\dots,a_{n-1}}$ is parabolic fixing $\xi^R_{a_0}$. 
Conversely, fix $k=0,\dots,n-1$ and observe that 
$
F_{a_k}(\text{vertices of }\Omega_\DD)\subset\overline{[a_k]_\DD}
$, 
because the vertices of $\Omega_\DD$ belongs to the closure of 
$\partial\DD\setminus[\widehat{a_k}]$. Condition~\eqref{EqNobacktrack} gives 
$
\xi^R_{a_{k+1}}\not=\xi^R_{\widehat{a_k}}=\inf\big(\partial\DD\setminus[\widehat{a_k}]\big)
$, 
thus we have 
$$
\text{either }
\quad
\xi^R_{a_{k+1}}=\xi^L_{\widehat{a_k}}=\sup\big(\partial\DD\setminus[\widehat{a_k}]\big)
\quad\text{ or }\quad
F_{a_k}\cdot \xi^R_{a_{k+1}}\in\text{ interior}([a_k]_\DD).
$$
Since 
$
F_{a_0,\dots,a_{n-1}}\cdot \xi^R_{a_n}=F_{a_0,\dots,a_{n-1}}\cdot\xi^R_{a_0}=\xi^R_{a_0}
$ 
is not an interior point of $[a_0]_\DD$, we must have $\xi^R_{a_{k+1}}=\xi^L_{\widehat{a_k}}$ for any $k=0,\dots,n-1$, that is $o(a_{k+1})=o(\widehat{a_k})-1$. The first part of the Lemma is proved. The second part follows with the same argument.
\end{proof}

Let $W=(a_0,\dots,a_n)$ be an admissible word. We say that $W$ is a \emph{cuspidal word} if it is the initial factor of an admissible word $(a_0,\dots,a_m)$ with $m\geq n$ such that 
$F_{a_0,\dots,a_m}$ is a parabolic element of $\Gamma_0$ fixing a vertex of $\Omega_\DD$. 

\begin{itemize}
\item
If $n\geq1$ and Equation~\eqref{EquationRightCuspidalWord(-)} is satisfied, we say that $W$ is a \emph{right cuspidal word}. In this case we define its type by $\varepsilon(W):=R$ and we set 
$\xi_W:=\xi^R_{a_0}$.
\item
If $n\geq1$ and Equation~\eqref{EquationLeftCuspidalWord(+)} is satisfied, we say that $W$ is a \emph{left cuspidal word}. In this case we define its type by $\varepsilon(W):=L$ and we set 
$\xi_W:=\xi^L_{a_0}$.
\item
If $n=0$, that is $W=(a_0)$ has just one letter, the type $\varepsilon(W)$ is not defined. We set by convention $\xi_W:=\xi^R_{a_0}$.
\end{itemize}

If $W=(a_0,\dots,a_n)$ is cuspidal with $n\geq1$, Lemma~\ref{LemmaCombinatorialPropertiesCuspidal} implies 
$
\xi^{\varepsilon(W)}_{a_k}=F_{a_k}\cdot\xi^{\varepsilon(W)}_{a_{k+1}}
$ 
for any $k=0,\dots,n-1$ and it follows 
\begin{equation}
\label{EquationCommonEndpointsCuspidalWords}
\xi_W=\partial[a_0]_\DD\cap\partial[a_0,a_1]_\DD\cap\dots\cap\partial[a_0,\dots,a_n]_\DD,
\end{equation}
that is the $n+1$ arcs above share $\xi_W$ as common endpoint (see also \S~2.4 in \cite{ArtigianiMarcheseUlcigrai(Fuchsian)} and \S~4.3 in \cite{ArtigianiMarcheseUlcigrai(Veech)}). Moreover, for any fixed $m\geq n$, there exists an unique cuspidal word $W'=(a_0,\dots,a_m)$ such that $W$ is an initial factor of $W'$. On the other hand any letter $a_0\in\cA$ generates exactly two cuspidal words of fixed length $m\geq1$. By Lemma~\ref{LemmaCombinatorialPropertiesCuspidal}, the word $(a_0,\dots,a_n)$ is left cuspidal if and only if $(\widehat{a_n},\dots,\widehat{a_0})$ is right cuspidal. Cusps of $\DD/\Gamma_0$ are in bijection with the set of words as in 
Lemma~\ref{LemmaCombinatorialPropertiesCuspidal} with minimal length, modulo inversion operation and cyclic permutation of the entries. A sequence $(a_n)_{n \in \NN}$ is said \emph{cuspidal} if any initial factor $(a_0,\dots,a_n)$ with $n\in\NN$ is a cuspidal word, and \emph{eventually cuspidal} if there exists 
$k \in \NN$ such that $(a_{n+k})_{n \in \NN}$ is a cuspidal sequence.

\subsection{The cuspidal acceleration}
\label{SectionCuspidalAcceleration}

If $W=(b_0,\dots,b_m)$ and $W'=(a_0,\dots,a_n)$ are words with 
$a_0\not=\widehat{b_m}$, define the word 
$
W\ast W':=(b_0,\dots,b_m,a_0,\dots,a_n)
$. 
Let $(a_n)_{n\in\NN}$ be a sequence satisfying Condition \eqref{EqNobacktrack} and not eventually cuspidal.

\begin{description}
\item[Initial step]
Set $n(0):=0$. Let $n(1)\in\NN$ be the maximal integer $n(1)\geq 1$ such that 
$
(a_0,\dots,a_{n(1)-1})
$ 
is cuspidal, then set
$$
W_0:=(a_0,\dots,a_{n(1)-1}).
$$
\item[Recursive step]
Fix $r\geq1$ and assume that the instants $n(0)<\dots<n(r)$ and the cuspidal words $W_0,\dots,W_{r-1}$ are defined. Define $n(r+1)\geq n(r)+1$ as the maximal integer such that 
$
[a_{n(r)},\dots,a_{n(r+1)-1}]
$ 
is cuspidal, then set 
\begin{equation}
\label{EquationDependenceOfCuspidalWordFromPoint}
W_r:=(a_{n(r)},\dots,a_{n(r+1)-1}).
\end{equation}
\end{description}

The sequence of words $(W_r)_{r\in\NN}$ is called the \emph{cuspidal decomposition} of 
$(a_n)_{n\in\NN}$. We have of course 
$
a_0,a_1,a_2\dots=W_0\ast W_1\ast\dots
$.

\begin{remark}
\label{RemarkConcatenationCuspidalWords}
If 
$
W_{r-1}:=(a_{n(r-1)},\dots,a_{n(r)-1})
$ 
and 
$
W_{r}:=(a_{n(r)},\dots,a_{n(r+1)-1})
$ 
are two consecutive cuspidal words in the cuspidal decomposition of a sequence $(a_n)_{n\in\NN}$ satisfying Condition \eqref{EqNobacktrack}, then the word 
$
(a_{n(r)-1},a_{n(r)},\dots,a_{n(r+1)-1})
$ 
can be cuspidal.
\end{remark}

Let $\cW$ be the set of all cuspidal words. Define the \emph{transition matrix} 
$M\in\{0,1\}^{\cW\times\cW}$ as the infinite matrix such that for any pair 
$
W=(b_0,\dots,b_m),W'=(a_0,\dots,a_n)
$ 
of cuspidal words the coefficient $M_{W,W'}\in\{0,1\}$ is defined by
\begin{equation}
\label{EquationDefinitionTransitionMatrix}
M_{W,W'}=
\left\{
\begin{array}{l}
0 
\quad\textrm{ if } 
a_0=\widehat{b_m}
\\
0 
\quad\textrm{ if } 
m=0
\quad
\textrm{ and }
\quad
o(a_0)=o(\widehat{b_0})\pm1
\\
0
\quad\textrm{ if } 
m\geq 1 
\quad
\textrm{ , }
\quad
\varepsilon(W')=L
\quad
\textrm{ and }
\quad
o(a_0)=o(\widehat{b_m})+1
\\
0
\quad\textrm{ if } 
m\geq 1
\quad
\textrm{ , }
\quad
\varepsilon(W')=R
\quad
\textrm{ and }
\quad
o(a_0)=o(\widehat{b_m})-1
\\
1 \quad\textrm{ otherwise. } 
\end{array}
\right.
\end{equation}

We have $M_{W,W'}=1$ if and only if the concatenated word $W\ast W'$ is admissible and $W\ast(a_0)$ is not cuspidal, where we observe that $(b_m)\ast W'$ can be cuspidal, this asymmetry being pointed out in 
Remark~\ref{RemarkConcatenationCuspidalWords}.

\smallskip

For $W=(a_0,\dots,a_n)$ in $\cW$ denote $F_W:=F_{a_0,\dots,a_n}$ the corresponding element of 
$\Gamma_0$. Set also 
$
F_{W_0,\dots,W_r}:=F_{a_0,\dots,a_n}
$ 
for 
$
(a_0,\dots,a_n)=W_0\ast\dots\ast W_r
$. 
Consider $\xi\in\partial\DD$ and let $(a_n)_{n\in\NN}$ be its boundary expansion as in 
Equation~\eqref{EquationDefinitionBoundaryExpansion}. Let $(W_r)_{r\in\NN}$ be the cuspidal decomposition of 
$(a_n)_{n\in\NN}$, where the $r$-th term $W_r=W_r(\xi)$ is given by 
Equation~\eqref{EquationDependenceOfCuspidalWordFromPoint}. We write
\begin{equation}
\label{EquationCuspidalCuttingSequence(UnitDisc)}
\xi=[a_0,a_1,\dots]_\DD=[W_0,W_1,\dots]_\DD.
\end{equation}
The cuspidal acceleration of the Bowen-Series map $\cB\cS:\partial\DD\to\partial\DD$ is the map 
$\cF:\partial\DD\to\partial\DD$ defined by
\begin{equation}
\label{EquationDefinitionCuspidalAccelerationBowenSeriesMap}
\cF(\xi):=F_W^{-1}\cdot \xi
\quad
\textrm{ if }
\quad
W_0(\xi)=W.
\end{equation}

\section{Good approximations and Bowen-Series expansion}
\label{SectionGoodApproximationsAndBowenSeries}

In this section we transfer to $\HH$ the tools of \S~\ref{SectionBowenSeriesExpansion}, providing a notion of boundary expansion on $\RR$. The main result is Theorem~\ref{TheoremGoodApproximationsAndBoundaryExpansion} below (appearing also in~\cite{MarcheseExpansion}), which establishes that the \emph{best approximations} of a given $\alpha\in\RR$ are exactly the \emph{convergents} provided by the Bowen-Series expansion. The central notion is that of \emph{geometric length} $|W|$ of a cuspidal word $W\in\cW$, defined by Equation~\eqref{EquationDefinitionGeometricLenght} below. 
Given $T>0$, Equation~\eqref{EquationDefinitionContinuedFractionCantor} in \S~\ref{SectionProofMainTheorem} defines the cantor set $\EE_T$ of those boundary points whose expansion contains only cuspidal words with geometric length bounded by $T$, and such sets are relevant in this context because their dimension can be studied via thermodynamic formalism. Similar Cantor sets are used in \cite{ArtigianiMarcheseUlcigrai(Fuchsian)}, where the length of allowed cuspidal words is measured differently. The notion of geometric length used here is adapted in order to have the sharp inclusions in the next Lemma~\ref{LemmaBadAndBoundedContinuedFraction}.

\subsection{Convergents and geometric length of cuspidal words}

Consider the finite index free subgroup $\Gamma_0<\Gamma$ and the ideal polygon 
$\Omega_\DD$ that we fixed in \S~\ref{SectionLabelledIdealPolygon}. 
Since $\Gamma_0$ has finite index in $\Gamma$ then the two groups have the same set of parabolic fixed points. Recalling Point (7) in \S~\ref{SectionPropertiesDirichletRegion} we write
\begin{equation}
\label{EquationSameParabolicFixedPoints}
\cP_\Gamma
=
\Gamma_0(\Omega_\DD\cap\partial\DD).
\end{equation}

\begin{figure}
\begin{center}
\begin{tikzpicture}[scale=2.55]

\begin{scope}

\tkzDefPoint(0,0){O}
\tkzDefPoint(1,0){A}
\tkzDrawCircle(O,A)

\node[] at (1,0) [right]{$\xi_W$};

\end{scope}

\begin{scope}

\tkzDefPoint(0,0){O}
\tkzDefPoint(1,0){A}
\tkzDrawCircle[thick,fill=black!15](O,A)

\tkzClipCircle(O,A)


\node[] at (-0.2,-0.5) {$\Omega_\DD$};

\tkzDefPoint({cos(0)},{sin(0)}){z0}
\tkzDefPoint({cos(pi/2)},{sin(pi/2)}){z1}
\tkzDefPoint({cos(5*pi/6)},{sin(5*pi/6)}){z2}
\tkzDefPoint({cos(8*pi/6)},{sin(8*pi/6)}){z3}
\tkzDefPoint({cos(3*pi/2)},{sin(3*pi/2)}){z4}
\tkzDefPoint({cos(10*pi/6)},{sin(10*pi/6)}){z5}

\tkzDefCircle[orthogonal through=z0 and z1](O,A)
\tkzGetPoint{tmp}
\tkzDrawCircle[fill=white](tmp,z1)
\node[] at ({0.5*cos(27.5)},{0.5*sin(27.5)}) {$s_0$};
\node[] at ({0.35*cos(70)},{0.35*sin(70)}) {$a_0$};
\tkzDefCircle[orthogonal through=z1 and z2](O,A)
\tkzGetPoint{tmp}
\tkzDrawCircle[fill=white](tmp,z2)
\tkzDefCircle[orthogonal through=z2 and z3](O,A)
\tkzGetPoint{tmp}
\tkzDrawCircle[fill=white](tmp,z3)
\node[] at ({0.55*cos(195)},{0.55*sin(195)}) {$a_3$};
\tkzDefCircle[orthogonal through=z3 and z4](O,A)
\tkzGetPoint{tmp}
\tkzDrawCircle[fill=white](tmp,z4)
\node[] at ({0.85*cos(255)},{0.85*sin(255)}) {$a_1$};
\tkzDefCircle[orthogonal through=z4 and z5](O,A)
\tkzGetPoint{tmp}
\tkzDrawCircle[fill=white](tmp,z5)
\tkzDefCircle[orthogonal through=z5 and z0](O,A)
\tkzGetPoint{tmp}
\tkzDrawCircle[fill=white](tmp,z0)
\node[right] at ({0.65*cos(330)},{0.65*sin(330)}) {$a_2$};

\draw[->,thick] ({0.75*cos(280)},{0.75*sin(280)}) .. controls 
({0.7*cos(275)},{0.7*sin(275)}) and ({0.7*cos(265)},{0.7*sin(265)}) 
.. ({0.75*cos(260)},{0.75*sin(260)}) {};

\draw[->,thick] ({0.4*cos(125)},{0.4*sin(125)}) .. controls 
({0.2*cos(180)},{0.2*sin(180)}) and ({0.2*cos(240)},{0.2*sin(240)}) 
.. ({0.4*cos(310)},{0.4*sin(310)}) {};

\draw[->,thick] ({0.3*cos(210)},{0.3*sin(210)}) .. controls 
({0.1*cos(240)},{0.1*sin(240)}) and ({0.1*cos(300)},{0.1*sin(300)}) 
.. ({0.3*cos(35)},{0.3*sin(35)}) {};


\tkzDefPoint({cos(pi/4)},{sin(pi/4)}){z6}
\tkzDefPoint({cos(pi/8)},{sin(pi/6)}){z7}
\tkzDefPoint({cos(pi/10)},{sin(pi/10)}){z8}

\tkzDefCircle[orthogonal through=z1 and z8](O,A)
\tkzGetPoint{tmp}
\tkzDrawCircle[thick,dashed](tmp,z1)
\tkzDefCircle[orthogonal through=z0 and z6](O,A)
\tkzGetPoint{tmp}
\tkzDrawCircle(tmp,z6)
\node[] at ({0.8*cos(35)},{0.8*sin(35)}) {$s_1$};
\tkzDefCircle[orthogonal through=z0 and z7](O,A)
\tkzGetPoint{tmp}
\tkzDrawCircle(tmp,z7)
\node[] at ({0.925*cos(23)},{0.925*sin(23)}) {$s_2$};
\tkzDefCircle[orthogonal through=z0 and z8](O,A)
\tkzGetPoint{tmp}
\tkzDrawCircle(tmp,z8)
\node[] at ({0.925*cos(11)},{0.925*sin(11)}) {$s_3$};

\end{scope}

\end{tikzpicture}

\begin{tikzpicture}[scale=0.39]

\clip[shift={(0,2)}](-1,-4.5) rectangle (16,8.5);

\draw[thin] (-1,0) -- (16,0);

\draw (0,0) -- (0,10);
\node[] at (-0.5,7) {$e_3$};
\draw (2,0) -- (2,10);
\node[] at (1.5,7) {$e_2$};
\draw (5,0) -- (5,10);
\node[] at (4.5,7) {$e_1$};
\node[] at (8.5,7) {$e_0$};

\draw[<->,dashed] (0,-1) -- (9,-1);
\node[] at (4.5,-2) [] {$|W|=|\re(e_3)-\re(e_0)|$};
\draw[thin,dotted] (0,-1.2) -- (0,0);
\draw[thin,dotted] (9,-1.2) -- (9,0);

\filldraw[fill=black!15!white,draw=black] 
(15,11) -- (15,0)
arc (0:180:0.5) 
arc  (0:180:1)
arc (0:180:0.5) 
arc (0:180:1) -- (9,11) -- (15,11);

\node[] at (12,9.5) {$\infty$};

\draw[thick,dashed] (0,0) arc (180:0:4.5);

\node[] at (12,4.5) {$A_k^{-1}B^{-1}\Omega_\HH$};

\end{tikzpicture}
\end{center}
\caption{Geometric length $|W|$ of a right cuspidal word 
$
W=(a_0,a_1,a_2,a_3)
$. 
The arrows inside $\Omega_\DD$ represent the action of $F_{a_0},F_{a_1},F_{a_2}$. The arcs $s_0:=s_{a_0}$, 
$
s_1:=F_{a_0}(s_{a_1})
$,
$
s_2:=F_{a_0,a_1}(s_{a_2})
$ 
and 
$
s_3:=F_{a_0,a_1,a_2}(s_{a_3})
$ 
share the common vertex $\xi_W$, which is sent to 
$\infty$ under the map $A_k^{-1}B^{-1}\varphi^{-1}$. Thus the arcs $s_0,s_1,s_2,s_3$ in $\DD$ are sent to parallel vertical arcs $e_i:=\varphi^{-1}(s_i)$ in $\HH$.}
\label{FigureGeometricLengthCuspidalWord}
\end{figure}

Let $\cS=(A_1,\dots,A_p)$ be as in Equation~\eqref{EquationRepresentativesOfCusps}. 
Let 
$
\Omega_\HH:=\varphi^{-1}(\Omega_\DD)
$ 
be the pre-image of $\Omega_\DD$ in $\HH$ under the map in  
Equation~\eqref{EquationConjugationUpperHalfPlaneDisc}. 
Any vertex $\xi$ of $\Omega_\DD$ corresponds to an unique vertex $\zeta=\varphi^{-1}(\xi)$ of 
$\Omega_\HH$. For any such vertex $\zeta$ consider $B\in\Gamma$ and $k\in\{1,\dots,p\}$ with 
\begin{equation}
\label{EquationVerticesDomainInfinity}
\zeta=BA_k\cdot \infty
\end{equation}

Any side $s_a$ of $\Omega_\DD$ corresponds to an unique side $e_{a}:=\varphi^{-1}(s_a)$ of $\Omega_\HH$, where $a\in\cA$. 
If $BA_k\cdot \infty=B'A_j\cdot \infty$, then $j=k$ and 
Corollary~\ref{CorollaryParabolicPointsOnlyFixedByParabolicElement} implies $B'=BP$, where $P\in\Gamma$ is parabolic fixing $A_k\cdot\infty$. Hence the map $z\mapsto A_k^{-1}PA_k(z)$ is an horizontal translation in $\HH$. 
If $s$ and $s'$ are geodesics in $\DD$ having $\xi$ as common endpoint, their pre-images in $\HH$ under 
$\varphi\circ B\circ A_k$ are parallel vertical half lines whose distance does not depend on the choice of $B$ in Equation~\eqref{EquationVerticesDomainInfinity}. We have a well defined positive real number 
$$
\Delta(s,s',\xi):=
\left|
\re\big(A_k^{-1}B^{-1}\varphi^{-1}(s)\big)
-
\re\big(A_k^{-1}B^{-1}\varphi^{-1}(s')\big)
\right|.
$$

Fix a cuspidal word  $W=(a_0,\dots,a_{n})$ and the vertex $\xi_W$ of $\Omega_\DD$ associated to $W$ in \S~\ref{SectionCuspidalWords}. For $n\geq1$, 
Equation~\eqref{EquationCommonEndpointsCuspidalWords} implies that the geodesics 
$
s_{a_0},F_{a_0}(s_{a_1}),\dots,F_{a_0,\dots,a_{n-1}}(s_{a_n})
$ 
all have $\xi_W$ as common endpoint. See Figure~\ref{FigureGeometricLengthCuspidalWord}. Define the \emph{geometric length} $|W|\geq0$ of $W$ as
\begin{equation}
\label{EquationDefinitionGeometricLenght}
|W|:=
\left\{
\begin{array}{ll}
\Delta\big(s_{a_0},F_{a_0,\dots,a_{n-1}}(s_{a_n}),\xi_W\big)
&\text{ if }n\geq1
\\
0
&\text{ if }n=0.
\end{array}
\right.
\end{equation}

For $a\in\cA$ set 
$
G_a=\varphi^{-1}\circ F_a\circ\varphi
$.
Set 
$
G_{a_0,\dots,a_n}:=G_{a_0}\circ\dots\circ G_{a_n}
$ 
for any word $(a_0,\dots,a_n)$ and 
$
G_{W_0,\dots,W_r}=G_{a_0,\dots,a_n}
$ 
if $(a_0,\dots,a_n)=W_0\ast\dots\ast W_r$. Define the interval
$$
[a_0,\dots,a_n]_\HH:=
\varphi^{-1}\big([a_0,\dots,a_n]_\DD\big)=
G_{a_0,\dots,a_n}\big(\partial\HH\setminus[\widehat{a_n}]_\HH\big).
$$

If the concatenation $W_0\ast\dots\ast W_r$ is admissible, set 
$$
G_{W_0,\dots,W_r}:=\varphi^{-1}\circ F_{W_0,\dots,W_r} \circ\varphi,
$$
where $F_{W_0,\dots,W_r}$ is defined in \S~\ref{SectionCuspidalAcceleration}. 

We encode any $\alpha\in\RR$ by the same cutting sequence as $\varphi(\alpha)\in\DD$, that is we define 
$$
[a_0,a_1,\dots]_\HH:=\varphi^{-1}\big([a_0,a_1,\dots]_\DD\big).
$$ 
 
Let $(W_r)_{r\in\NN}$ be the cuspidal decomposition of $(a_n)_{n\in\NN}$. Equation~\eqref{EquationCuspidalCuttingSequence(UnitDisc)} becomes 
\begin{equation}
\label{EquationCuspidalCuttingSequence(UpperHalfPlane)}
\alpha=[W_0,W_1,\dots]_\HH:=[a_0,a_1,\dots]_\HH.
\end{equation}

For $r\in\NN$ let $W_r$ be the $r$-th cuspidal word. Set 
$
\zeta_{W_r}:=\varphi^{-1}(\xi_{W_r})
$. 
The convergents of $\alpha$ are  
\begin{equation}
\label{EquationConvergentBowenSeries}
\zeta_r:=G_{W_0,\dots,W_{r-1}}\cdot\zeta_{W_r}
\quad\text{ ; }\quad
r\in\NN.
\end{equation}

For $k=1,\dots,p$ let $\mu_k>0$ be such that the primitive parabolic element 
$P_k\in A_k \Gamma A_k^{-1}$ fixing $\infty$ acts by $P_k(z)=z+\mu_k$. Set 
$\mu:=\max\{\mu_1,\dots,\mu_p\}$.

\begin{theorem}
\label{TheoremGoodApproximationsAndBoundaryExpansion}
For any $r\in\NN$ with $|W_r|>0$ we have 
\begin{equation}
\label{EquationGoodApproximationsAndBoundaryExpansion(1)}
\frac{1}{|W_r|+2\mu}
\leq
D(G_{W_0,\dots,W_{r-1}}\cdot\zeta_{W_r})^2
\cdot
|\alpha-G_{W_0,\dots,W_{r-1}}\cdot \zeta_{W_r}|
\leq 
\frac{1}{|W_r|}.
\end{equation}
Moreover there exists $\epsilon_0>0$ depending only on $\Omega_\DD$ and on $\cS$, such that for any $G\in\Gamma$ and $k=1,\dots,p$ with $D(G\cdot z_k)\not=0$ the condition 
$$
D(G\cdot z_k)^2\cdot|\alpha-G\cdot z_k|<\epsilon_0
$$
implies that there exists some $r\in\NN$ such that  
\begin{equation}
\label{EquationGoodApproximationsAndBoundaryExpansion(2)}
G\cdot z_k=G_{W_0,\dots,W_{r-1}}\cdot\zeta_{W_r}
\quad
\textrm{ where }
\quad
|W_r|>0.
\end{equation}
\end{theorem}

\begin{remark}
\label{RemarkCovarianceIndependence}
Equation~\eqref{EquationGoodApproximationsAndBoundaryExpansion(1)} holds for any choice of $\cS$ as in Equation~\eqref{EquationRepresentativesOfCusps}, and this follows because geometric length and denominators satisfy a form of equivariance under the choice of $\cS$. 
\end{remark}

\subsection{Reduced form of parabolic fixed points}
\label{SectionParabolicFixedPointsAndFiniteIndexFreeSubgroup}

Fix $G\cdot z_k\in\cP_\Gamma$. Recall Equation~\eqref{EquationSameParabolicFixedPoints} and write 
elements of $\Gamma_0$ in the generators $\{G_a:a\in\cA\}$. There exists an unique admissible word $b_0,\dots,b_m$ and a vertex $\zeta$ of 
$\Omega_\HH$ which is not an endpoint of $e_{\widehat{b_m}}$ such that 
$$
G\cdot z_k=G_{b_0,\dots,b_m}\cdot\zeta.
$$
The representation above is called the \emph{reduced form} of the parabolic fixed point $G\cdot z_k$. 
In the next Lemmas~\ref{LemmaDistancePoleReducedForm} and~\ref{LemmaInequalityDenominatorsReducedForm}, let $(b_0,\dots,b_m)$ be a non-trivial admissible word and let $\zeta_0$ be a vertex of $\Omega_\HH$ which is not an endpoint of $e_{\widehat{b_m}}$, so that 
$
G_{b_0,\dots,b_m}\cdot\zeta_0
$ 
is a parabolic fixed point written in its reduced form and different from $\infty$.

\begin{lemma}
\label{LemmaDistancePoleReducedForm}
There exists a constant $\kappa_1>0$, depending only on $\Omega_\HH$, such that 
$$
\left|
\zeta_0-G_{b_0,\dots,b_m}^{-1}\cdot\infty
\right|
\geq \kappa_1,
$$
that is the vertex $\zeta_0$ and the pole of $G_{b_0,\dots,b_m}$ stay at distance uniformly bounded from below.
\end{lemma}

\begin{proof}
We have 
$
G_{b_0,\dots,b_m}\big(\RR\setminus[\widehat{b_m}]_\HH\big)=[b_0,\dots,b_m]_\HH
$ 
By Equation~\eqref{EquationDefinitionCylinderBowenSeries}. 
Since $\infty$ does not belong to the interior of $[b_0,\dots,b_m]_\HH$ then the pole of $G_{b_0,\dots,b_m}$ belongs to the closure of 
$[\widehat{b_m}]_\HH$. The Lemma follows because $\zeta_0$ is a vertex of $\Omega_\HH$ different from the endpoints of $e_{\widehat{b_m}}$.
\end{proof}

\begin{lemma}
\label{LemmaInequalityDenominatorsReducedForm}
There exists a constant $\kappa_2>0$, depending only on $\Omega_\HH$ and on $\cS$, such that the following holds. 
\begin{enumerate}
\item
If $\zeta_1$ is a vertex of $\Omega_\HH$ different from $\zeta_0$, then  
$$
D(G_{b_0,\dots,b_m}\cdot\zeta_0)
\geq 
\kappa_2\cdot D(G_{b_0,\dots,b_m}\cdot\zeta_1).
$$
\item
If $b_{m+1}$ satisfies $b_{m+1}\not=\widehat{b_m}$ and $\zeta_2$ is a vertex of $\Omega_\HH$ with 
$
G_{b_{m+1}}\cdot\zeta_2\not=\zeta_0
$,  
then   
$$
D(G_{b_0,\dots,b_m}\cdot\zeta_0)
\geq 
\kappa_2\cdot D(G_{b_0,\dots,b_m,b_{m+1}}\cdot\zeta_2).
$$
\end{enumerate}
\end{lemma}

\begin{proof}
We prove Part (1). Set $G:=G_{b_0,\dots,b_m}$, 
$
\zeta:=G\cdot\zeta_0
$, 
and 
$
\zeta':=G\cdot\zeta_1
$. 
If $\zeta'=\infty$ then the statement is trivially true. If $D(G\cdot\zeta_1)\not=0$, let 
$\zeta_0=B_0A_k\cdot\infty$ and $\zeta_1=B_1A_j\cdot\infty$ as in 
Equation~\eqref{EquationVerticesDomainInfinity}. Referring to  
Equation~\eqref{EquationCoefficientsSL(2,C)}, let $c,d$ be the entries of $G$. Let $a_0,c_0$ and $a_1,c_1$ be the entries of $B_0A_k$ and $B_1A_j$ respectively. We prove an upper bound for
$$
\frac{D(G_{b_0,\dots,b_m}\cdot\zeta_1)}{D(G_{b_0,\dots,b_m}\cdot\zeta_0)}
=
\left|
\frac{ca_1+dc_1}{ca_0+dc_0}
\right|.
$$
We cannot have $c_0=c_1=0$, because $\zeta_0\not=\zeta_1$ and in particular $\zeta_0$, $\zeta_1$ cannot be both equal to $\infty$. Moreover $G\cdot\zeta_0$, $G\cdot\zeta_1$ are both different from $\infty$, thus condition $c=0$ implies $c_0,c_1\not=0$. Hence for $c=0$ Part (1) follows because the ratio above equals $|c_1/c_0|$, which varies in a finite set of values and is therefore bounded from above. If $c,c_0,c_1\not=0$ then 
$$
\left|
\frac{ca_1+dc_1}{ca_0+dc_0}
\right|
=
\left|
\frac{c_1}{c_0}
\right|
\cdot
\left|
\frac{(a_1/c_1)-(-d/c)}{(a_0/c_0)-(-d/c)}
\right|
=
\left|
\frac{c_1}{c_0}
\right|
\cdot
\left|
\frac{\zeta_1-(G^{-1}\cdot\infty)}{\zeta_0-(G^{-1}\cdot\infty)}
\right|.
$$
In this case Part (1) follows because $|c_1/c_0|$ is bounded from above, and 
Lemma~\ref{LemmaDistancePoleReducedForm} gives a lower bound for the denominator of the second factor (the numerator is not bounded, but as it increases the ratio converges to $1$). If $c,c_0\not=0$ and $c_1=0$ then Lemma~\ref{LemmaDistancePoleReducedForm} gives
$$
\left|
\frac{ca_1+dc_1}{ca_0+dc_0}
\right|
=
\left|
\frac{a_1}{c_0}
\right|
\cdot
\left|
\frac{1}{(a_0/c_0)-(-d/c)}
\right|
=
\left|
\frac{a_1}{c_0}
\right|
\cdot
\left|
\frac{1}{\zeta_0-(G^{-1}\cdot\infty)}
\right|
\leq
\left|
\frac{a_1}{c_0\cdot\kappa_1}
\right|,
$$
and Part (1) follows observing that $a_1/c_0$ varies in a finite set of values. Finally, if $c,c_1\not=0$ and $c_0=0$ then
$$
\left|
\frac{ca_1+dc_1}{ca_0+dc_0}
\right|
=
\left|
\frac{a_1}{a_0}-(-d/c)\frac{c_1}{a_0}
\right|
\leq
\left|\frac{a_1}{a_0}\right|+|G^{-1}\cdot\infty|\left|\frac{c_1}{a_0}\right|.
$$
In this case $\zeta_0=\infty$, which is not an endpoint of $[\widehat{b_m}]$. Thus $[\widehat{b_m}]$ is contained in the compact interval of $\RR$ delimited by the two parallel vertical segments of $\Omega_\HH$. Hence $|G^{-1}\cdot\infty|$ is uniformly bounded, because the pole $G^{-1}\cdot\infty$ belongs to the closure of $[\widehat{b_m}]$ (see proof of Lemma~\ref{LemmaDistancePoleReducedForm}). Part (1) follows in this case too, and the proof is complete. Part (2) follows similarly, replacing 
$\zeta_1$ by 
$
\zeta_\ast:=G_{b_{m+1}}\cdot\zeta_2
$ 
and observing that, since $G_{b_{m+1}}$ varies in the finite set 
$\{G_a:a\in\cA\}$ then also the entries of $X\in\sltwor$ with
$
G_{b_{m+1}}\cdot\zeta_2=X\cdot \infty
$ 
vary in a finite set. Moreover $\zeta_0\not=\zeta_\ast$, and thus 
$
G\cdot\zeta_0\not=G\cdot\zeta_\ast
$. 
\end{proof}

\subsection{Proof of Theorem~\ref{TheoremGoodApproximationsAndBoundaryExpansion}}

Let 
$
\alpha=[a_0,a_1,\dots]_\HH=[W_0,W_1,\dots]_\HH
$ 
be the expansion of $\alpha\in\RR$ as in 
Equation~\eqref{EquationCuspidalCuttingSequence(UpperHalfPlane)}.

\begin{figure}
\begin{center}
\begin{tikzpicture}[scale=0.31]

\clip(-1,-5) rectangle (16,11.5);

\draw[thin] (-1,0) -- (16,0);

\filldraw[fill=black!15!white, draw=black] 
(15,0) arc (0:180:1) 
arc  (0:180:1)
arc (0:180:3.5) 
arc (0:180:1) 
arc (0:180:1) 
arc (180:0:7.5);

\node[] at (12.5,2) {$G(\Omega_\HH)$};

\draw  (4,0) arc (180:0:0.5);
\node[] at (7.2,0.7) {$e'_2$};
\draw  (4,0) arc (180:0:1.25);
\node[] at (8.2,2.2) {$e'_1$};
\draw  (4,0) arc (180:0:2.25);
\node[] at (9.2,3.7) {$e'_0$};

\draw[thick,dashed]  (6,0) -- (6,12);
\node[circle,fill,inner sep=1pt] at (6,0) {};
\node[] at (6.5,0) [below] {$\alpha$};

\node[circle,fill,inner sep=1pt] at (4,0) {};
\node[] at (3.5,0) [below] {$G\zeta_{W_r}$};

\draw[thick] (4,2) circle (2);

\end{tikzpicture}
\begin{tikzpicture}[scale=0.31]

\clip(-1,-5) rectangle (16,11.5);

\draw[thin] (-1,0) -- (16,0);


\filldraw[fill=black!15!white, draw=black] 
(12,0) arc (0:180:3.5) 
arc  (0:180:0.5)
arc (0:180:1) 
arc (0:180:1) -- (0,16.5) -- (12,16.5) -- (12,0);
\node[] at (10,6.5) {$\Omega_\HH$};

\draw  (5,0) arc (180:0:0.5);
\node[] at (8,0.5) {$e_2$};
\draw  (5,0) arc (180:0:1.25);
\node[] at (9,2) {$e_1$};
\draw  (5,0) arc (180:0:2.25);
\node[] at (10,3.5) {$e_0$};

\draw[thick,dashed]  (7,0) arc (0:180:3);
\node[circle,fill,inner sep=1pt] at (7,0) {};
\node[] at (8.5,0) [below] {$G^{-1}\alpha$};

\node[circle,fill,inner sep=1pt] at (5,0) {};
\node[] at (4.5,0) [below] {$\zeta_{W_r}$};

\draw[thick] (5,1.33) circle (1.33);

\end{tikzpicture}
\begin{tikzpicture}[scale=0.31]

\clip(-1,-5) rectangle (16.5,11.5);

\draw[thin] (-1,0) -- (17,0);

\draw[<->,dashed] (1.5,-2) -- (9,-2);
\node[] at (4.5,-2.25) [above] {$|W_r|$};
\draw[thin,dotted] (1.5,-2.2) -- (1.5,0);
\draw[thin,dotted] (9,-2.2) -- (9,0);

\draw[<->,dashed] (9,-2) -- (16,-2);
\node[] at (12.5,-2.25) [above] {$\mu$};
\draw[thin,dotted] (16,-2.2) -- (16,0);

\filldraw[fill=black!15!white,draw=black] 
(16,12) -- (16,0)
arc (0:180:1) 
arc  (0:180:1)
arc (0:180:0.5) 
arc (0:180:1) -- (9,12) -- (16,12);


\draw (0,0) -- (0,12);
\node[] at (2,8) {$e''_2$};
\draw (1.5,0) -- (1.5,12);
\node[] at (5.5,8) {$e''_1$};
\draw (5,0) -- (5,12);
\node[] at (9.5,8) {$e''_0$};

\draw[thick,dashed] (1,0) arc (180:0:6.25);

\draw[thick] (-1,6.25) -- (16,6.25);

\draw[<->,dashed] (1,-4) -- (13.5,-4);
\node[] at (7,-4.25) [above] {$2T$};
\draw[thin,dotted] (1,-4.2) -- (1,0);
\draw[thin,dotted] (13.5,-4.2) -- (13.5,0);

\node[] at (12.5,10) {$A_k^{-1}B^{-1}\Omega_\HH$};

\end{tikzpicture}
\end{center}
\caption{The $r$-th cuspidal word $W_r=(a_0,a_1,a_2)$ of $\alpha$ is the first cuspidal word of $G^{-1}\cdot\alpha$, where $G=G_{W_0,\dots,W_{r-1}}$. The vertex 
$\zeta_{W_r}$ of $\Omega_\HH$ is common to the arcs $e_0=e_{a_0}$, $e_1:=G_{a_0}e_{a_1}$ and $e_2:=G_{a_0a_1}e_{a_2}$. The arcs $e'_i=Ge_i$ share the vertex $G\zeta_{W_r}$. The point $\zeta_{W_r}$ is sent to 
$\infty$, and the arcs $e_0,e_1,e_2$ are sent to the parallel vertical arcs $e''_0,e''_1,e''_2$. 
We have $|W_r|=\big|\re(e''_2)-\re(e''_0)\big|$.}
\label{FigureProofLemmaConvergentsAndGeometricLength}
\end{figure}

\subsubsection{Proof of Equation~\eqref{EquationGoodApproximationsAndBoundaryExpansion(1)}}

Fix $r\in\NN$ with $|W_r|>0$. Take $k\in\{1,\dots,p\}$ and $B\in\Gamma$ as in 
Equation~\eqref{EquationVerticesDomainInfinity}, that is $\zeta_{W_r}=BA_k\cdot \infty$. 
As in Figure~\ref{FigureProofLemmaConvergentsAndGeometricLength}, let $T>0$ be such that the horoball
$$
B_T:=G_{W_0,\dots,W_{r-1}}BA_k\big(\{z\in\HH:\im(z)>T\}\big)
$$
is tangent at $G_{W_0,\dots,W_{r-1}}\cdot \zeta_{W_r}$ with radius 
$
\rho(B_T)=|\alpha-G_{W_0,\dots,W_{r-1}}\cdot \zeta_{W_r}|
$. 
Equation~\eqref{EquationDiameterHoroball} and 
Equation~\eqref{EquationDefinitionDenominator} give
$$
D(G_{W_0,\dots,W_{r-1}}\cdot\zeta_{W_r})^2
\cdot
|\alpha-G_{W_0,\dots,W_{r-1}}\cdot\zeta_{W_r}|
=
c^2(G_{W_0,\dots,W_{r-1}}BA_k)\cdot\frac{\diameter(B_T)}{2}
=
\frac{1}{2T}.
$$
The geodesic in $\HH$ with endpoints 
$(G_{W_0,\dots,W_{r-1}}BA_k)^{-1}\cdot \infty$ and 
$(G_{W_0,\dots,W_{r-1}}BA_k)^{-1}\cdot \alpha$ is tangent to $\{z\in\HH:\im(z)>T\}$. 
Equation~\eqref{EquationGoodApproximationsAndBoundaryExpansion(1)} follows because 
Equation~\eqref{EquationDefinitionGeometricLenght} gives
$$
|W_r|\leq 2T\leq |W_r|+2\mu.
$$

\subsubsection{Proof of Equation~\eqref{EquationGoodApproximationsAndBoundaryExpansion(2)}}

Referring to \S~\ref{SectionParabolicFixedPointsAndFiniteIndexFreeSubgroup}, let $\zeta_0$ be the vertex of $\Omega_\HH$ and $(b_0,\dots,b_m)$ be the admissible word such that the reduced form of the parabolic fixed point $G\cdot z_k$ is
$$
G\cdot z_k=G_{b_0,\dots,b_m}\cdot\zeta_0,
$$
where $\zeta_0$ is not an endpoint of $e_{\widehat{b_m}}$ whenever $(b_0,\dots,b_m)$ is not the empty word. Assume  
$
D(G\cdot z_k)^2|\alpha-G\cdot z_k|<\epsilon_0
$, 
where the constant $\epsilon_0>0$ will be determined later.

\smallskip

\emph{Step $(0)$.} Assume that $(b_0,\dots,b_m)$ is the empty word, so that  
$
\zeta_0=G\cdot z_k\not=\infty
$. 
Consider the extra assumption $|W_0|>0$ and $\zeta_0=\zeta_{W_0}$ on pairs $(\alpha,\zeta_0)$, where 
$
\zeta_{W_0}=\varphi^{-1}(\xi_{W_0})
$ 
and $\xi_{W_0}$ is the vertex of $\Omega_\DD$ associated to $W_0$ as in \S~\ref{SectionCuspidalWords}. 
Define $\epsilon_0>0$ by 
$$
\epsilon_0:=\inf_{(\alpha,\zeta_0)}D(\zeta_0)^2\cdot|\alpha-\zeta_0|,
$$
where the infimum is taken over all pairs $(\alpha,\zeta_0)$ not satisfying the extra assumption. With such $\epsilon_0$, the statement follows whenever $(b_0,\dots,b_m)$ is the empty word.

\smallskip

\emph{Step $(1)$.} Now assume that $(b_0,\dots,b_m)$ is not the empty word. Then $G\cdot z_k$ is an interior point of  
$
[b_0,\dots,b_m]_\HH
$. 
Let $\zeta_1,\zeta_2$ be the endpoints of $[\widehat{b_m}]$, which are vertices of $\Omega_\HH$ different from $\zeta_0$. The endpoints of 
$
[b_0,\dots,b_m]_\HH
$ 
are 
$
\zeta'_i:=G_{b_0,\dots,b_m}\cdot\zeta_i
$ 
for $i=1,2$, according to Equation~\eqref{EquationDefinitionCylinderBowenSeries}. 
Let $N\geq-1$ be maximal with $a_n=b_n$ for any $n=0,\dots,N$, where the last condition is empty for $N=-1$, and where 
$N\leq m$. Observe that condition $N\leq m-1$ implies  
$
\alpha\not\in [b_0,\dots,b_m]_\HH
$, 
and therefore 
\begin{align*}
|\alpha-G\cdot z_k|
&
\geq
\min_{i=1,2}
|\zeta'_i-G\cdot z_k|
=
\min_{i=1,2}
|G_{b_0,\dots,b_m}\cdot\zeta_i-G_{b_0,\dots,b_m}\cdot\zeta_0|
\\
&
\geq
\frac{S_0}{D(G_{b_0,\dots,b_m}\cdot\zeta_0)}\cdot
\min_{i=1,2}
\frac{1}{D(G_{b_0,\dots,b_m}\cdot\zeta_i)}
\geq
\frac{S_0\kappa_2}{D(G_{b_0,\dots,b_m}\cdot\zeta_0)^2},
\end{align*}
where the third inequality follows from 
Part (1) of Lemma~\ref{LemmaInequalityDenominatorsReducedForm} and the second from 
Equation~\eqref{EquationDistanceDisjointHoroballs}. Therefore $N=m$, provided that 
$\epsilon_0<\kappa_2S_0$. 

We proved 
$
[a_0,\dots,a_m]_\HH=[b_0,\dots,b_m]_\HH
$. 
Moreover $G\cdot z_k$ does not belong to the interior of $[a_0,\dots,a_m,a_{m+1}]_\HH$, since the latter is a subinterval of $[b_0,\dots,b_m]_\HH$ delimited by the image under $G_{b_0,\dots,b_m}$ of two consecutive vertices of $\Omega_\HH$. The same argument as in the first part of Step (1), which is left to the reader, shows that $G\cdot z_k$ is an endpoint of $[a_0,\dots,a_m,a_{m+1}]_\HH$. 

\smallskip

\emph{Step $(2)$.} We show that 
$
G\cdot z_k=G_{b_0,\dots,b_m}\cdot \zeta_0
$ 
is an endpoint of $[a_0,\dots,a_{m+2}]_\HH$. Otherwise $G\cdot z_k$ doesn't belong to the closure of 
$[a_0,\dots,a_{m+2}]_\HH$. Since   
$
\alpha\in[a_0,\dots,a_{m+2}]_\HH
$  
then 
\begin{align*}
|\alpha-G\cdot z_k|
&
\geq
|G_{b_0,\dots,b_m,a_{m+1}}\cdot\zeta_3-G_{b_0,\dots,b_m}\cdot \zeta_0|
\\
&
\geq
\frac
{S_0}
{D(G_{b_0,\dots,b_m}\cdot\zeta_0)D(G_{b_0,\dots,b_m,a_{m+1}}\cdot\zeta_3)}
\geq
\frac
{S_0\kappa_2}
{D(G_{b_0,\dots,b_m}\cdot\zeta_0)^2},
\end{align*}
where 
$
G_{b_0,\dots,b_m,a_{m+1}}\cdot\zeta_3
$ 
is the endpoint of $[a_0,\dots,a_{m+2}]_\HH$ which is closest to $G\cdot z_k$ and where $\zeta_3$ is a vertex of $\Omega_\HH$ which is not an endpoint of $e_{\widehat{a_{m+1}}}$. We use Equation~\eqref{EquationDistanceDisjointHoroballs} and Part (2) of Lemma~\ref{LemmaInequalityDenominatorsReducedForm}. The inequality is absurd by condition 
$
\epsilon_0<\kappa_2S_0
$. 

\smallskip

\emph{Step (3).} Let $r$ be minimal such that $(a_0,\dots,a_m)$ is an initial factor of $W_0\ast\dots\ast W_{r-1}$. If $(a_0,\dots,a_{m+2})$ is also an initial factor of $W_0\ast\dots\ast W_{r-1}$, then
$
G_{W_0,\dots,W_{r-1}}\cdot \xi_{W_{r-1}}
$ 
is a common endpoint of the intervals $[a_0,\dots,a_m]_\HH$, $[a_0,\dots,a_{m+1}]_\HH$ and 
$[a_0,\dots,a_{m+2}]_\HH$, according to Equation~\eqref{EquationCommonEndpointsCuspidalWords}. 
Without loss of generality we have 
$$
G_{W_0,\dots,W_{r-1}}\cdot \xi_{W_{r-1}}
=
\inf[a_0,\dots,a_m]_\HH
=
\inf[a_0,\dots,a_{m+1}]_\HH
=
\inf[a_0,\dots,a_{m+2}]_\HH.
$$
The common endpoint is not $G\cdot z_k$, which belongs to the interior of 
$[a_0,\dots,a_m]_\HH$. Thus Step $(1)$ implies   
$
G\cdot z_k=\sup[a_0,\dots,a_{m+1}]_\HH
$, 
which is absurd because $G\cdot z_k$ is an endpoint of $[a_0,\dots,a_{m+2}]_\HH$ by Step $(2)$. Hence 
$
W_0\ast\dots\ast W_{r-1}
$ 
is either equal to $(a_0,\dots,a_m)$ or to $(a_0,\dots,a_{m+1})$. Moreover $(a_{m+1},a_{m+2})$ is a cuspidal word, because $[a_0,\dots,a_{m+1}]_\HH$ and $[a_0,\dots,a_{m+2}]_\HH$ share the endpoint 
$G\cdot z_k$.

\smallskip

$\bullet$ In case  
$
W_0\ast\dots\ast W_{r-1}=(a_0,\dots,a_m)
$ 
the word $(a_{m+1},a_{m+2})$ is an initial factor of $W_r$, that is $|W_r|>0$ and $\zeta_0=\zeta_{W_r}$. 

\smallskip

$\bullet$ In case    
$
W_0\ast\dots\ast W_{r-1}=(a_0,\dots,a_{m+1})
$ 
the word $W':=(a_{m+1})\ast W_r$ is also cuspidal (this is allowed by 
Remark~\ref{RemarkConcatenationCuspidalWords}). 
If $|W_r|=0$, that is $W_r=(a_{m+2})$, then $G\cdot z_k$ does not belong to the closure of 
$[a_0,\dots,a_{m+3}]_\HH$ and we get an absurd by 
$$
|\alpha-G\cdot z_k|
\geq
|G_{b_0,\dots,b_m}\cdot \zeta_0-G_{b_0,\dots,b_m,a_{m+1},a_{m+2}}\cdot\zeta_3|
\geq
\frac
{S_0\kappa_2}
{D(G_{b_0,\dots,b_m}\cdot\zeta_0)^2},
$$
where $\zeta_3$ is a vertex of $\Omega_\HH$ and 
$
G_{b_0,\dots,b_m,a_{m+1},a_{m+2}}\cdot\zeta_3
$ 
is the endpoint of $[a_0,\dots,a_{m+3}]_\HH$ which is closest to $G\cdot z_k$. In the last inequality we reason as in Step $(2)$, replacing $\kappa_2$ by a smaller constant and extending Part (2) of 
Lemma~\ref{LemmaInequalityDenominatorsReducedForm} one more step, in order to compare 
$D(G_{b_0,\dots,b_m}\cdot\zeta_0)$ and 
$D(G_{b_0,\dots,b_m,a_{m+1},a_{m+2}}\cdot\zeta_3)$. Since $W'$ is cuspidal with $|W'|>0$ we have 
$\zeta_0=\zeta_{W'}$. But we have also 
$
\zeta_{W'}=G_{a_{m+1}}\cdot\zeta_{W_r}
$, 
which implies 
$$
G_{b_0,\dots,b_m}\cdot\zeta_0=
G_{a_0,\dots,a_m}\cdot G_{a_{m+1}}\cdot\zeta_{W_r}=
G_{W_0,\dots,W_{r-1}}\cdot\zeta_{W_r}.
$$
In both cases Equation~\eqref{EquationGoodApproximationsAndBoundaryExpansion(2)} follows. The proof of Theorem~\ref{TheoremGoodApproximationsAndBoundaryExpansion} is complete. $\qed$

\section{Proof of Main Theorem~\ref{TheoremMainTheorem}}
\label{SectionProofMainTheorem}

In this paragraph we prove Theorem~\ref{TheoremMainTheorem} using Theorem~\ref{TheoremDimensionShiftSpace} below, whose proof is the subject of the rest of the paper, starting from \S~\ref{SectionSubshiftFiniteType}. 
Recall from Equation~\eqref{EquationDefinitionGeometricLenght} the definition of geometric length $|W|$ of a cuspidal word $W\in\cW$. Fix $T>0$ and let $\EE_T$ be the set of points 
$
\xi=[a_0,a_1,\dots]_\DD=[W_0,W_1,\dots]_\DD 
$ 
in $\partial\DD$ whose expansion in 
Equation~\eqref{EquationCuspidalCuttingSequence(UnitDisc)} satisfies 
\begin{equation}
\label{EquationDefinitionContinuedFractionCantor}
|W_r|\leq T \quad\text{ for any }\quad r\in\NN,
\end{equation}
that is all cuspidal words $W_r$ in the expansion of $\xi$ have geometric length bounded by $T$. 

\begin{theorem}
\label{TheoremDimensionShiftSpace}
Then exists a constant $\Theta>0$ such that for any $T>0$ big enough have
$$
\dim_H(\EE_T)=1-\Theta\cdot T^{-1}+o(T^{-1}).
$$
\end{theorem}

Let $\Theta>0$ be as in Theorem~\ref{TheoremDimensionShiftSpace}
and $\mu>0$ be as in Theorem~\ref{TheoremGoodApproximationsAndBoundaryExpansion}. For any 
$\epsilon>0$ set 
$
\displaystyle{\epsilon':=\frac{\epsilon}{1-2\mu\epsilon}}
$. 
It is clear that we have
$$
\big|
\big(1-\Theta\cdot\epsilon+o(\epsilon)\big)
-
\big(1-\Theta\cdot\epsilon'+o(\epsilon')\big)
\big|=o(\epsilon).
$$
The map $\varphi:\HH\to\DD$ in Equation~\eqref{EquationConjugationUpperHalfPlaneDisc} is smooth, thus it does not change Hausdorff dimension. Theorem~\ref{TheoremMainTheorem} follows from 
Theorem~\ref{TheoremDimensionShiftSpace} and from 
Lemma~\ref{LemmaBadAndBoundedContinuedFraction} below.

\begin{lemma}
\label{LemmaBadAndBoundedContinuedFraction}
For any $\epsilon<\epsilon_0$ we have
$
\EE_{\epsilon^{-1}-2\mu}
\subset
\varphi\left(\bad(\Gamma,\epsilon)\right)
\subset
\EE_{\epsilon^{-1}}.
$ 
\end{lemma}

\begin{proof}
Fix $\epsilon<\epsilon_0$. For any $\alpha\in\RR$ consider the cuspidal acceleration of its boundary expansion, that is write $\alpha=[W_0,W_1,\dots]_\HH$. Consider any $\alpha\in\bad(\Gamma,\epsilon)$. 
Equation~\eqref{EquationGoodApproximationsAndBoundaryExpansion(1)} implies that for any $r\in\NN$ we have
$$
\epsilon\leq 
D(G_{W_0,\dots,W_{r-1}},\zeta_{W_r})^2\cdot|\alpha-G_{W_0,\dots,W_{r-1}}\cdot\zeta_{W_r}|
\leq
\frac{1}{|W_r|}.
$$
Thus it follows  
$
\alpha\in\varphi^{-1}\big(\EE_{\epsilon^{-1}}\big)
$. 
On the other hand, consider  
$
\alpha\not\in\bad(\Gamma,\epsilon)
$ 
and let $G\in\Gamma$ and 
$z_k=A_k\cdot\infty$ such that 
$
D(G,z_k)^2\cdot|\alpha-G\cdot z_k|<\epsilon
$. 
Equation~\eqref{EquationGoodApproximationsAndBoundaryExpansion(2)} implies that there exists some $r\in\NN$ such that 
$
G\cdot z_k=G_{W_0,\dots,W_{r-1}}\cdot\zeta_{W_r}
$. 
For such $r$ Equation~\eqref{EquationGoodApproximationsAndBoundaryExpansion(1)} implies 
$$
\frac{1}{|W_r|+2\mu}
\leq
D(G_{W_0,\dots,W_{r-1}},\zeta_{W_r})^2\cdot|\alpha-G_{W_0,\dots,W_{r-1}}\cdot\zeta_{W_r}|
<
\epsilon.
$$
Thus it follows that 
$
\alpha\not\in\varphi^{-1}\big(\EE_{\epsilon^{-1}-2\mu}\big)
$. 
\end{proof}

\section{Cantor sets and subshifts of finite type}
\label{SectionSubshiftFiniteType}

\subsection{Aperiodicity of transition matrix}

Let $M_{W',W}\in\{0,1\}^{\cW\times\cW}$ be the transition matrix in    
Equation~\eqref{EquationDefinitionTransitionMatrix}. Fix $0<T\leq+\infty$. Let $\cW_T$ be the set of cuspidal words $W\in\cW$ with geometric length $|W|\leq T$, see 
Equation~\eqref{EquationDefinitionGeometricLenght}. 
The matrix $M_{W',W}$ describes also the allowed transitions between the elements of $\cW_T$. According to Lemma~1.3 in~\cite{BowenEquilibriumStates}, for any $m\geq1$ the entry $M^m_{W',W}$ of the $m$-th power $M^m$ of $M$ is the number of different words $W_0,W_1,\dots,W_m$ in the letters of $\cW_T$ with length $m+1$ with  
\begin{enumerate}
\item
$M_{W_i,W_{i+1}}=1$ for any $i=0,\dots,m-1$
\item
$W_0=W'$ and $W_m=W$.
\end{enumerate}

Following \S~1 in \cite{BowenEquilibriumStates} and \S~1 in \cite{ParryPollicott} we say that the matrix $M$ is \emph{aperiodic} if there exists some $m\in\NN$ such that $M^m_{W',W}\geq1$ for any $W,W'\in\cW$.

\begin{proposition}
\label{PropositionAperiodicityTransitionMatrix}
The matrix $M$ is aperiodic for any $T>0$ big enough. More precisely, if $T$ is big enough, we have 
$M^2_{W',W}\geq 1$ for any $W,W'$ in $\cW_T$.
\end{proposition}

\begin{proof}
Recall Lemma~\ref{LemmaCombinatorialPropertiesCuspidal} and fix $r\geq1$. Consider the map 
$
\chi\mapsto\phi_{(r)}(\chi)
$ 
from $\cA$ to itself, where $\chi'=\phi_{(r)}(\chi)$ is the unique letter such that the word 
$
(\chi=\chi_0,\dots,\chi_{r}=\chi')
$ 
is right cuspidal. If $(a_0,\dots,a_r)$ and $(b_0,\dots,b_r)$ are two right cuspidal words with 
$a_r=b_r$, then 
$
(\widehat{a_r},\dots,\widehat{a_0})
$ 
and 
$
(\widehat{b_r},\dots,\widehat{b_0})
$ 
are left cuspidal with 
$
\widehat{a_r}=\widehat{b_r}
$, 
by Lemma~\ref{LemmaCombinatorialPropertiesCuspidal}. Thus the two words are equal, because the first letter determines left cuspidal words of fixed length $r\geq1$. Hence $b_0=a_0$. 
The map $\phi_{(r)}$ is a bijection, because it is injective. 

Fix any pair of elements 
$
W'=(a_0,\dots,a_{n})
$ 
and 
$
W=(b_0,\dots,b_{m})
$ 
in $\cW$. Assume first that the alphabet $\cA$ has at least 6 letters. We show that $M^2_{W',W}\geq 1$ showing that there exists a right cuspidal word $X=[\chi_0,\chi_1]\in\cW$ such that $M_{W',X}=1$ and $M_{X,W}=1$, where such $X$ is determined by the choice of its first letter $\chi_0$. According to Equation~\eqref{EquationDefinitionTransitionMatrix}, a sufficient (but in general not necessary) condition on $\chi_0$ in order to have $M_{W',X}=1$ is 
$$
o(\chi_0)-o(\widehat{a_n})\not=-1,0,1,
$$ 
which corresponds to 3 forbidden values for $\chi_0$. Moreover, since $(\chi_0,\chi_1)$ is right cuspidal, condition $M_{W',X}=1$ is satisfied if and only if
$$
o(b_0)-o(\widehat{\chi_1})\not=-1,0,
$$
which corresponds to 2 forbidden values for $\chi_1$, and thus 2 forbidden values for $\chi_0$, because the map $\phi_{(1)}$ is a bijection. Therefore there are at most 5 forbidden values for $\chi_0$, and thus at least 1 possible choice, which determines a word $X\in\cW$ as required.

Assume now that $\cA$ has only 4 letters, denoted   
$
a,b,\widehat{a},\widehat{b}
$. 
The two corresponding generators $F_a,F_b$ of $\Gamma_0$ are either both parabolic (with different fixed point in $\partial\DD$) or both hyperbolic (with different axis which intersect transversally, where the axis of an hyperbolic element $G$ is the geodesics whose endpoints are the fixed points of $G$). In both cases we find a cuspidal word $X\in\cW$ as required considering all possible values for the pair 
$(a_n,b_0)$, that is all possible values for the last letter of $W'$ and for the first of $W$. More precisely, for any value of $(a_n,b_0)$ we exhibit a concatenation of 3 cuspidal words of the form 
$$
W'\ast X\ast W=(\dots,a_n)\ast X\ast(b_0,\dots).
$$

If $F_a,F_b$ are both parabolic, then there are in total 12 different type of cuspidal words: 3 of them start with the letter $a$, that is  
$$
(\underbrace{a,\dots,a}_{k\geq1\textrm{ times }})
\quad
\textrm{ , }
\quad
(\underbrace{a,\widehat{b},\dots,a,\widehat{b}}_{k\geq1\textrm{ times }})
\quad
\textrm{ , }
\quad
(\underbrace{a,\widehat{b},\dots,a,\widehat{b}}_{k\geq1\textrm{ times }},a),
$$
while the other 9 are obtained in the obvious way replacing the role of letters. For 
$a_n=a$ there are 4 possible values of $b_0$, and the list below gives, for each case, a $X\in\cW$ such that the concatenation $W\ast X\ast W'$ satisfies Condition~\eqref{EquationDefinitionTransitionMatrix}:
$$
\begin{array}{ccccc}
\,(\dots,a) & \ast & (b) & \ast & (a,\dots)
\\
\,(\dots,a)& \ast & (b,\widehat{a},b) & \ast & (b,\dots)
\\
\,(\dots,a) & \ast & (b,b) & \ast & (\widehat{a},\dots)
\\
\,(\dots,a) & \ast & (b,\widehat{a}) & \ast & (\widehat{b},\dots).
\end{array}
$$
For the other 3 values of $a_n$ the analogous list is obtained by obvious substituions and details are left to the reader. Finally, if $F_a,F_b$ are both hyperbolic, the are only two cuspidal sequences, which are
\begin{align*}
&
a,b,\widehat{a},\widehat{b},a,b,\widehat{a},\widehat{b},\dots
\\
&
b,a,\widehat{b},\widehat{a},b,a,\widehat{b},\widehat{a},\dots
\end{align*}
and cuspidal words are any finite subword of the sequences above. For $a_n=a$ there are 4 possible values of $b_0$, and the list below gives, for each case, a $X\in\cW$ such that the concatenation 
$W\ast X\ast W'$ satisfies Condition~\eqref{EquationDefinitionTransitionMatrix}:
$$
\begin{array}{ccccc}
\,(\dots,a) & \ast & (a) & \ast & (a,\dots)
\\
\,(\dots,a)& \ast & (a,b,\widehat{a}) & \ast & (b,\dots)
\\
\,(\dots,a) & \ast & (a,b,\widehat{a}) & \ast & (\widehat{a},\dots)
\\
\,(\dots,a) & \ast & (a,\widehat{b},\widehat{a}) & \ast & (\widehat{b},\dots).
\end{array}
$$
For the other 3 values of $a_n$ the analogous list is obtained by obvious substituions and details are left to the reader. The Proposition is proved.
\end{proof}

\subsection{The space of the sub-shift}
\label{SectionShiftSpace}

Denote by $w=(W_r)_{r\in\NN}$ the elements in $\cW_T^\NN$, that is half-infinite sequences in the letters 
$W\in\cW_T$. The \emph{shift space} is 
$$
\Sigma:=
\left\{
w=(W_r)_{r\in\NN}:M_{W_r,W_{r+1}}=1\quad\forall\quad r\in\NN
\right\}.
$$

Considering the discrete topology on $\cW_T$ and the product topology on $\cW_T^\NN$, we obtain a compact totally disconnected topological space. The shift space $\Sigma$ is a compact subset of 
$\cW_T^\NN$. Following \S~1 in \cite{BowenEquilibriumStates} and \S~1 in \cite{ParryPollicott} we fix 
$\theta$ with $0<\theta<1$ and define a distance on $\cW_T^\NN$, and thus on $\Sigma$, by setting 
$$
d_\theta(w,w'):=\theta^N
\quad
\textrm{ where }
\quad
N:=\max\{n\in\NN: W_r=W'_r\quad\forall\quad r=0,\dots,n\}
$$
and where $w=(W_r)_{r\in\NN}$ and $w'=(W'_r)_{r\in\NN}$ are any pair of sequences in $\cW_T^\NN$. A basis of open sets for the induced discrete topology is given by the set of \emph{cylinders}, where for any finite sequence $W_0,\dots,W_n$ of elements of $\cW_T$ such that $M_{W_i,W_{i+1}}=1$ for any $i=0,\dots,n-1$ the corresponding cylinder is 
$$
[W_0,\dots,W_n]_\Sigma:=
\{w=(W'_r)_{r\in\NN}\in\Sigma: W'_r=W_r\quad\forall\quad r=0,\dots,n\}.
$$

There is a natural dynamics on $\Sigma$ given by the \emph{shift map}
$$
\sigma:\Sigma\to\Sigma
\quad
\textrm{ ; }
\quad
w=(W_{r})_{r\in\NN}\mapsto \sigma(w):=(W_{r+1})_{r\in\NN}.
$$

\subsection{Subshift and Cantor sets in the boundary}
\label{SectionSubshiftAndCantorSetBoundary}

Fix $T$ with $0<T<\infty$ and let $\EE_T$ be the set in 
Equation~\eqref{EquationDefinitionContinuedFractionCantor}. Consider the map
$$
\Pi:\Sigma\to\partial\DD
\quad
\textrm{ ; }
\quad
w=(W_r)_{r\in\NN}\mapsto\Pi(w):=[W_0\ast W_1\ast\dots]_\DD.
$$

For the map $\cF:\partial\DD\to\partial\DD$ 
in Equation~\eqref{EquationDefinitionCuspidalAccelerationBowenSeriesMap} we have the commutative diagram 
\begin{equation}
\label{EquationCommutativeDiagramShiftExpansion}
\Pi\circ\sigma=\cF\circ\Pi.
\end{equation}

Lemma~\ref{LemmaUniformContractionBranches} below gives a uniform (i.e. depending only on the geometry of $\Omega_\DD$ and not on $T$) contraction factor $0<\theta<1$ for the inverse branches of $\cF$. For this specific parameter the distance $d_\theta(\cdot,\cdot)$ reflects some metric properties of $\EE_T$. 

\begin{lemma}
\label{LemmaCodingLipschitz}
The map $\Pi:\Sigma\to\partial\DD$ is bijective onto $\EE_T$. Moreover $\Pi$ is Lipschitz, but its inverse 
$\Pi^{-1}$ is just H\"older continuous.
\end{lemma}

\begin{proof}
Part (1) of Proposition~\ref{PropositionSizeGapsAndIntervals} implies directly that $\Pi$ is Lipschitz. H\"older continuity for $\Pi^{-1}:\EE_T\to\Sigma$ follows from the estimates in 
Proposition~\ref{PropositionSizeGapsAndIntervals}. Such property of the inverse will not be used in the following, and details are left to the interested reader.
\end{proof}

\subsection{Cylinders in the Cantor set}
\label{SectionCylindersCantorSets}

Any word $(a_0,\dots,a_n)$ in the letters of $\cA$ corresponds to an arc 
$
[a_0,\dots,a_n]_\DD
$ 
as in Equation~\eqref{EquationDefinitionCylinderBowenSeries}. Points sharing the same coding for the map $\cF$ correspond to arcs with a different form, that we describe here.

\smallskip

Consider any cuspidal word $W=(b_0,\dots,b_m)$ in $\cW$ and set 
$$
[W]_\EE:=
\{\xi\in\partial\DD:\cF(\xi)=F_W^{-1}(\xi)\}=
\{\xi\in\partial\DD:W_0(\xi)=W\},
$$
that is the arc in $\partial\DD$ of points $\xi$ whose first cuspidal word, given by 
Equation~\eqref{EquationDependenceOfCuspidalWordFromPoint}, satisfies $W_0(\xi)=W$. Observe that
$$
W_0(\xi)=W
\quad
\Leftrightarrow
\quad
M_{W,W_0(F_W^{-1}(\xi))}=1,
$$
that is $\xi\in[W]_\EE$ if and only if the first letter of the cuspidal word $W_0(F_W^{-1}(\xi))$ satisfies Equation~\eqref{EquationDefinitionTransitionMatrix}. Therefore 
$
[W]_\EE=F_W\big(\domain(W)\big)
$, 
where we define
\begin{equation}
\label{EquationDomainImageCylinderShift}
\domain(W):=
\left\{
\begin{array}{ll}
\bigcup_{o(\chi)-o(\widehat{b_0})\not=0,\pm1}[\chi]_\DD
&
\textrm{ if }
m=0
\\
\bigcup_{o(\chi)-o(\widehat{b_m})\not=0,1}[\chi]_\DD
&
\textrm{ if }
m\geq1,\quad\epsilon(W)=L
\\
\bigcup_{o(\chi)-o(\widehat{b_m})\not=0,-1}[\chi]_\DD
&
\textrm{ if }
m\geq1,\quad\epsilon(W)=R,
\end{array}
\right.
\end{equation}
Recall Equation~\eqref{EquationDefinitionCylinderBowenSeries} and observe that 
$
[W]_\EE\not=[b_0,\dots,b_m]_\DD
$. 
Observe that the set of letters $\chi$ in 
Equation~\eqref{EquationDomainImageCylinderShift} is the same as the set of letters $a_0$ in 
Equation~\eqref{EquationDefinitionTransitionMatrix}. Therefore if $W'=(a_0,\dots,a_n)$ is any other cuspidal word we have the equivalence
\begin{equation}
\label{EquationEquivalenceShiftTransitionCylinderInclusion}
M_{W,W'}=1
\quad
\Leftrightarrow
\quad
[W']_\EE\subset\domain(W).
\end{equation}
Let $w_n:=(W_1,\dots,W_n)$ be a finite block of cuspidal words $W_k\in\cW$ for $k=1,\dots,n$ such that 
\begin{equation}
\label{EquationFiniteWordSubshitfAlphabeth}
M_{W_k,W_{k+1}}=1
\quad
\textrm{ for }
\quad
k=0,\dots,n-1.
\end{equation}
Define 
$
[W_1,\dots,W_n]_\EE:=\Pi\big([W_1,\dots,W_n]_\Sigma\big)
$, 
that is 
$$
[W_1,\dots,W_n]_\EE=
\{\xi\in\partial\DD:W_0(\cF^{k-1}(\xi))=W_k\textrm{ for }k=1,\dots,n\}.
$$
Equation~\eqref{EquationEquivalenceShiftTransitionCylinderInclusion} and 
Equation~\eqref{EquationFiniteWordSubshitfAlphabeth} give $[W_{k+1}]_\EE\subset\domain(W_k)$ for any $k=0,\dots,n-1$, and therefore
$
F_{W_k}\big([W_{k+1}]_\EE\big)\subset[W_k]_\EE
$, 
that is 
$
F_{W_k}\big([W_{k+1}]_\EE\big)=[W_k,W_{k+1}]_\EE
$. 
Iterating we get
$$
[W_1,\dots,W_n]_\EE
=
F_{W_1,\dots,W_n}\big(\domain(W_n)\big).
$$
For $w_n:=(W_1,\dots,W_n)$ as above set  
$
[w_n]_\EE:=[W_1,\dots,W_n]_\EE
$ 
and 
$
F_{w_n}:=F_{W_1,\dots,W_n}
$.

\section{Estimates on contraction and distortion}
\label{SectionEstimatesContractionDistortion}

All constants in this section only depend on the geometry of $\Omega_\DD$ and not on the parameter $T>0$ defining $\cW_T$. We call them \emph{uniform constants}, and don't assign a specific name to each of them. The only exception is the contraction factor $\theta$ in 
Lemma~\ref{LemmaUniformContractionBranches}, which appears also in 
Theorem~\ref{TheoremTransferOperatorCircle} below. The constant $\theta$ provided by 
Lemma~\ref{LemmaUniformContractionBranches} is fixed once and for all and will be used the the constructions which follow in the rest of the paper. All estimates in this section follow from 
Lemma~\ref{LemmaSeparationIsometricCircles}.

\subsection{Distance from the poles}
\label{SectionDistanceFromPoles}

For any admissible word $(a_0,\dots,a_n)$ in the letters of 
$\cA$ set
$$
U_{a_0,\dots,a_n}:=U_{F_{a_0,\dots,a_n}},
$$ 
that is the interior of the isometric circle of $F_{a_0,\dots,a_n}$ (see \S~\ref{SectionIsometricCircles}). We have 
$
|D_z F_{a_0,\dots,a_n}|\leq1
$ 
if and only if $z\in\CC\setminus U_{a_0,\dots,a_n}$. 
We have $\partial U_a\cap \DD=s_{\widehat{a}}$ for any $a\in\cA$. 
Equation~\eqref{EquationIsometricCircleComposition} and a standard inductive argument imply that for any $n$ we have $U_{a_0,\dots,a_n}\subset U_{a_n}$ and thus 
$$
U_{a_0,a_1,\dots,a_n}\subset U_{a_1,\dots,a_n}.
$$
The last inclusion and Equation~\eqref{EquationIsometricCircleComposition} imply
\begin{equation}
\label{EquationInclusionsIsometricCirclesLetters}
F_{a_0,\dots,a_{n-1}}(U_{\widehat{a_n}})
=
F_{\widehat{a_{n-1}},\dots,\widehat{a_0}}^{-1}(U_{\widehat{a_n}})
\subset
U_{\widehat{a_n},\widehat{a_{n-1}},\dots,\widehat{a_0}}
\subset
U_{\widehat{a_0}}.
\end{equation}

\begin{lemma}
\label{LemmaSeparationIsometricCircles}
Consider $m\geq0$ and letters $b_0,\dots,b_m$ such that $(b_0,\dots,b_m)$ is cuspidal. Let $a_0$ be any letter such that $(b_0,\dots,b_m,a_0)$ is admissible but not cuspidal. Then
$$
\overline{U_{b_0,\dots,b_m}}\cap\overline{U_{\widehat{a_0}}}=\emptyset.
$$
\end{lemma}

\begin{proof}
Let $o:\cA\to\ZZ/2d\ZZ$ be the map in Equation~\eqref{EquationCyclicOrderLetters}. Assume first $m=0$. In this case $(b_0,a_0)$ is not cuspidal if and only if  
$
|o(a_0)-o(\widehat{b_0})|\geq 2
$. 
This last condition implies 
$$
\overline{U_{b_0}}\cap\overline{U_{\widehat{a_0}}}=\emptyset.
$$
If $m\geq 1$ assume without loss of generality that $(b_0,\dots,b_m)$ is right cuspidal, the other case being the same. In particular we have $o(b_m)=o(\widehat{b_{m-1}})-1$. Let $\chi$ be the letter with $o(\chi)=o(\widehat{b_m})-1$, that is the letter such that 
$(b_{m-1},b_m,\chi)$ is right cuspidal, then let $\xi\in\partial\DD$ be the tangency point between the discs $U_{\widehat{\chi}}$ and $U_{b_m}$ (that is 
$
\{\xi\}=[\chi]_\DD\cap[\widehat{b_m}]_\DD
$). 
We have 
$$
\overline{F^{-1}_{b_m}(U_{b_{m-1}})}\cap \overline{U_{\widehat{\chi}}}=\{\xi\}.
$$
Equation~\eqref{EquationIsometricCircleComposition} gives 
$
F^{-1}_{b_m}(U_{b_{m-1}})\subset U_{b_{m-1}b_m}\subset U_{b_m}
$, 
and therefore $U_{b_{m-1}b_m}$ and $U_{\widehat{\chi}}$ are also tangent at $\xi$. Finally the inclusion 
$
U_{b_{m-1}b_m}\subset U_{b_m}
$ 
is strict, thus we get
$$
\overline{U_{b_{m-1}b_m}}\subset \{\xi\}\cup U_{\widehat{\chi}}\cup U_{b_m}.
$$
On the other hand $(b_{m-1},b_m,a_0)$ is not right cuspidal, thus $a_0\not=\widehat{b_m},\chi$, that is the disc $U_{\widehat{a_0}}$ is different from $U_{\widehat{\chi}}$ and from $U_{b_m}$, so that we get 
$$
\overline{U_{b_{m-1},b_m}}\cap\overline{U_{\widehat{a_0}}}=\emptyset.
$$
Then the statement follows because 
$
U_{b_0,\dots,b_m}\subset U_{b_{m-1}b_m}
$ 
according to Equation~\eqref{EquationIsometricCircleComposition}.
\end{proof}

\begin{remark}
For $U\subset\CC$ open and $\epsilon>0$ set 
$
B(U,\epsilon):=\{z\in\CC:\distance(z,U)<\epsilon\}
$. 
Lemma~\ref{LemmaSeparationIsometricCircles} keeps true replacing $U_{\widehat{a_0}}$ by 
$
B(U_{\widehat{a_0}},\epsilon)
$ 
for some $\epsilon>0$. This implies that one can consider an holomorphic version of the transfer operator $L_{(s,T)}$ in Theorem~\ref{TheoremTransferOperatorCircle} below. For such holomorphic transfer operator it is possible to pursue an analysis as in \cite{MoellerPohl}.
\end{remark}

\begin{lemma}
\label{LemmaDistanceFromPoles}
There exists a uniform positive constant $C>1$ such that for any finite block 
$
w_k:=(W_1,\dots,W_k)
$ 
as in Equation~\eqref{EquationFiniteWordSubshitfAlphabeth}, any $a_0\in\cA$ with 
$
[a_0]_\DD\subset\domain(W)
$ 
and any $\xi\in U_{\widehat{a_0}}$ the distance from the pole $\omega=\omega(F_{w_k})$ of $F_{w_k}$ satisfies 
$$
C^{-1}\leq|\xi-\omega|\leq C.
$$ 
In particular the property above holds for any $\xi\in\domain(W)$.
\end{lemma}

\begin{proof}
We have 
$
\omega\in U_{W_1,\dots,W_k}\subset U_{W_k}
$. 
Set $W_k=(b_0,\dots,b_m)$ and consider any $a_0\in\cA$ such that 
$
[a_0]_\EE\subset\domain(W_k)
$. 
If $m\geq1$ we have 
$$
U_{W_k}=U_{b_0,\dots,b_m}\subset U_{b_{m-1},b_m}
\quad
\textrm{ and }
\quad
\overline{U_{b_{m-1},b_m}}\cap \overline{U_{\widehat{a_0}}}=\emptyset
$$
otherwise if $m=0$ we have 
$$
U_{W_k}=U_{b_0}
\quad
\textrm{ and }
\quad
\overline{U_{b_0}}\cap \overline{U_{\widehat{a_0}}}=\emptyset,
$$
where in both cases the intersection is empty according to Lemma~\ref{LemmaSeparationIsometricCircles}. Since the sequences $(b_0,b_1,a_0)$ and $(b_0,a_0)$ as in the two cases above vary in a finite set, then there exists an uniform constant $C>0$ such that, for any $w_k=(W_1,\dots,W_k)$ as above we have
$$
\distance
\big(\overline{U_{W_1,\dots,W_k}},\overline{U_{\widehat{a_0}}}\big)
\geq
C^{-1}.
$$
The lower bound for $|\xi-\omega|$ follows. The upper bound holds trivially because in any infinite discrete subgroup of $\sugroup(1,1)$ we have 
$|\omega_F|\to1$ as $\|F\|\to\infty$, that is the poles $\omega_F$ of the maps $z\mapsto F(z)$ accumulate to the unit circle.
\end{proof}

\subsection{Uniform contraction}

\begin{lemma}
\label{LemmaWeakContractionAdmissibleWords}
For any admissible word $(a_0,\dots,a_n)$ and any $z\in\CC\setminus U_{a_n}$ we have
$$
|D_zF_{a_0,\dots,a_n}|\leq 1.
$$
\end{lemma}

\begin{proof}
The Lemma holds for $n=0$ because $|D_zF_a|\leq 1$ for any $a\in\cA$ and $z\in\CC\setminus U_{a}$, by definition of isometric circles. For $n\geq1$ the inequality follows by induction observing that 
$$
D_zF_{a_0,\dots,a_n}=
D_{F_{a_n}(z)}F_{a_0,\dots,a_{n-1}}\cdot D_zF_{a_n}
$$
and that admissibility, that is 
$
a_n\not=\widehat{a_{n-1}}
$, 
implies 
$$
F_{a_n}(\CC\setminus U_{a_n})
=
\CC\cap U_{\widehat{a_n}}
\subset
\CC\setminus U_{a_{n-1}}.
$$
\end{proof}

\begin{lemma}
\label{LemmaUniformContractionBranches}
There exists an uniform constant $0<\theta<1$ such that for any $W\in\cW$, any $a_0\in\cA$ with 
$
[a_0]_\DD\subset\domain(W)
$ 
and any $\xi\in U_{\widehat{a_0}}$ we have 
$$
|D_\xi F_W|\leq \theta.
$$ 
In particular the property above holds for any $\xi\in\domain(W)$.
\end{lemma}

\begin{proof}
Let $\cV_1$ be the set of words $(b_0,a_0)$ which are admissible and not cuspidal, then set  
$$
\theta_1:=\sup \{|D_\xi F_{b_0}|:(b_0,a_0)\in\cV_1,\xi\in U_{\widehat{a_0}}\}.
$$
We have $0<\theta_1<1$ according to Lemma~\ref{LemmaSeparationIsometricCircles}, where we recall that 
$|D_\xi F_{b_0}|<1$ if and only if $\xi\in\CC\setminus \overline{U_{b_0}}$. Similarly, let $\cV_2$ be the set of admissible words $(b_0,b_1,a_0)$ such that $(b_0,b_1)$ is cuspidal but $(b_0,b_1,a_0)$ is not cuspidal. By a similar argument, we have $0<\theta_2<1$, where 
$$
\theta_2:=\sup \{|D_\xi F_{b_0,b_1}|:(b_0,b_1,a_0)\in\cV_2,\xi\in U_{\widehat{a_0}}\}.
$$
In the general case consider $W=(b_0,\dots,b_m)$. If $m=0$ we have  
$
|D_\xi F_W|=|D_\xi F_{b_0}|
$, 
which is bounded by $\theta_1$. Otherwise, if $m\geq1$, the chain rule applied to $|D_\xi F_W|$ gives the factor 
$
|D_\xi F_{b_{m-1},b_m}|
$, 
which is bounded by $\theta_2$, and an other factor which is $\leq1$ because of Lemma~\ref{LemmaWeakContractionAdmissibleWords}. Hence in general the statement holds with  
$
\theta:=\max\{\theta_1,\theta_2\}
$. 
\end{proof}

\begin{corollary}
\label{CorollaryExponentiallyContractingDistance}
Consider $k\geq1$ and a block of cuspidal words  
$
w_k=(W_1,\dots,W_k)
$ 
as in Equation~\eqref{EquationFiniteWordSubshitfAlphabeth}. Then for any 
$a\in\cA$ with 
$
[a]_\DD\subset\domain(W_k)
$ 
and any $\xi',\xi$ in $U_{\widehat{a}}$ we have 
$$
\big|F_{w_k}(\xi')-F_{w_k}(\xi)\big|\leq \theta^k|\xi'-\xi|.
$$
\end{corollary}

\begin{proof}
Define $\gamma:[0,1]\to\CC$ by $\gamma(t):=\xi+t(\xi'-\xi)$. 
In particular $\gamma(t)\in U_{\widehat{a}}$ for any $t\in[0,1]$, so that 
$|D_{\gamma(t)}F_{W_k}|\leq\theta$ by  Lemma~\ref{LemmaUniformContractionBranches}. 
If $k\geq2$, for $l=2,\dots,k$, let $a_l$ be the first letter of $W_l$. We have 
$
[W_l]_\DD\subset[a_l]_\DD
$, 
but on the other hand $[W_l]_\DD\subset\domain(W_{l-1})$, according to 
Equation~\eqref{EquationEquivalenceShiftTransitionCylinderInclusion}. Therefore 
$
[a_l]_\DD\subset\domain(W_{l-1})
$. 
Moreover Equation~\eqref{EquationInclusionsIsometricCirclesLetters} implies 
$$
F_{W_l,\dots,W_k}(\gamma(t))
\in 
F_{W_l,\dots,W_k}(U_{\widehat{a}})
\subset
U_{\widehat{a_l}}.
$$
We get 
$
|D_{\gamma(t)}F_{W_1,\dots,W_k}|\leq\theta^k
$ 
factorizing the derivative with the chain rule and applying Lemma~\ref{LemmaUniformContractionBranches} to each factor. Finally the statement follows observing that
\begin{align*}
\big|F_{w_k}(\xi')-F_{w_k}(\xi)\big|
\leq
\int_0^1\big|d/dt F_{w_k}\big(\gamma(t)\big)\big|dt
=
\int_0^1\big|D_{\gamma(t)}F_{w_k}\big|\cdot|d\gamma(t)/dt|dt
\leq
\theta^k\cdot|\xi'-\xi|.
\end{align*}
\end{proof}

\subsection{Distortion of the derivative}

Consider blocks $w_k=(W_1,\dots,W_k)$ of cuspidal words as in 
Equation~\eqref{EquationFiniteWordSubshitfAlphabeth}, and for any such $w_k$ set 
\begin{equation}
\label{EquationDefinitionSupremumDerivativeBlockWords}
\|DF_{w_k}\|_\infty:=\sup_{\xi\in\domain(W_k)}|D_\xi F_{w_k}|.
\end{equation}

\begin{lemma}
\label{LemmaLipschitzConstantLogDerivative}
There exists an uniform constant $C>0$ such that for any $0<T\leq +\infty$ and $0<s\leq1$, any $a\in\cA$, $k\in\NN$ and any finite block $w_k=(W_1,\dots,W_k)$ as in Equation~\eqref{EquationFiniteWordSubshitfAlphabeth}, the map 
$$
\phi:\domain(W_k)\to\RR_+
\quad;\quad
\phi(\xi):=s\ln|D_\xi F_{w_k}|
$$
is Lipschitz with $\lipschitz(\phi)<C$.
\end{lemma}

\begin{proof}
In the notation of Equation~\eqref{EquationCoefficientsSU(1,1)}, let $\alpha$ and $\beta$ with 
$
|\alpha|^2-|\beta|^2=1
$ 
such that
$$
F_{w_k}(\xi)=
\frac{\alpha\xi+\overline{\beta}}{\beta\xi+\overline{\alpha}}
\quad
\textrm{ and }
\quad
D_\xi F_{w_k}=
\frac{1}{\beta^2\cdot(\xi-\omega)^2},
$$
where $\omega:=-\overline{\alpha}/\beta$ is the pole of $F_{w_k}$. We have 
$$
\phi(\xi)
=
\log\big(|D_\xi F_{w_k}|^s\big)
=
-2s\big(\ln|\beta|+\ln|\xi-\omega|\big),
$$
where we recall that $\beta(F)\not=0$ for any $F\in\Gamma_0$, since the latter does not have elliptic fixed points. Thus for any  
$\xi,\xi'\in\domain(W_k)$ we have 
$$
|\phi(\xi')-\phi(\xi)|=
2s\big|\ln|\xi'-\omega|-\ln|\xi-\omega|\big|
\leq
2sC\big||\xi'-\omega|-|\xi-\omega|\big|
\leq
2sC\big|\xi'-\xi\big|.
$$
The last inequality follows by triangle inequality for $|\cdot|$. The first inequality follows from 
Lemma~\ref{LemmaDistanceFromPoles} observing that 
$
|\ln(x)-\ln(y)|\leq C|x-y|
$ 
for any $x,y$ in $[C^{-1},+\infty)$.
\end{proof}

\begin{corollary}
\label{CorollaryLipschitzConstantDerivative}
There exists an uniform constant $C>0$ such that for any $0<T\leq +\infty$, any $k\in\NN$ and any finite block $w_k=(W_1,\dots,W_k)$ as in Equation~\eqref{EquationFiniteWordSubshitfAlphabeth} the following holds.
\begin{enumerate}
\item
For any $\xi',\xi$ in $\domain(W_k)$ we have
$$
\frac{1}{C}
\leq
\frac
{|D_{\xi'} F_{w_k}|}
{|D_{\xi}F_{w_k}|}
\leq C.
$$
\item
For any $0<s\leq1$ and any $\xi',\xi$ in $\domain(W_k)$ we have
$$
\big||D_{\xi'}F_{w_k}|^s-|D_{\xi}F_{w_k}|^s\big|
\leq
C\cdot\|DF_{w_k}\|^s_\infty\cdot|\xi'-\xi|.
$$
\end{enumerate}
\end{corollary}

\begin{proof}
Point (1) follows considering the expression of $F_{w_k}$ given by    
Equation~\eqref{EquationCoefficientsSU(1,1)} and observing that the ratio of the derivative at different points is bounded by Lemma~\ref{LemmaDistanceFromPoles}. We prove Point (2) using the same notation as in the proof of Lemma~\ref{LemmaLipschitzConstantLogDerivative}. In particular let $C'>0$ be the constant in the Lemma. We have 
$
\phi(\xi')\leq\phi(\xi)+C'\cdot|\xi'-\xi|
$. 
It follows
\begin{align*}
e^{\phi(\xi')}-e^{\phi(\xi)}
&
\leq
e^{\phi(\xi)+C'\cdot|\xi'-\xi|}-e^{\phi(\xi)}
\leq
e^{\phi(\xi)}\big(e^{C'\cdot|\xi'-\xi|}-1\big)
\\
&
\leq
e^{\phi(\xi)}\cdot e^{C'\cdot|\xi'-\xi|}\cdot C'\cdot|\xi'-\xi|
=
|D_\xi F_{w_k}|^s\cdot e^{C'\cdot|\xi'-\xi|}\cdot C'\cdot|\xi'-\xi|
\\
&
\leq
\|D F_{w_k}\|_\infty^s\cdot e^{C'\cdot|\xi'-\xi|}\cdot C'\cdot|\xi'-\xi|
\leq
C'e^{2C'}\cdot \|D F_{w_k}\|_\infty^s\cdot |\xi'-\xi|,
\end{align*}
where the third equality holds because $e^R-1\leq Re^R$ for any $R\geq0$ and the last inequality uses that $|\xi'-\xi|\leq 2$ for any $\xi,\xi'$ in $\partial\DD$. The statement follows because by symmetry the same bound holds for $e^{\phi(\xi)}-e^{\phi(\xi')}$.
\end{proof}

\subsection{Sizes of cylinders and intervals in the Cantor set}

Fix $n\geq0$. Consider a block $w_n=(W_0,\dots,W_n)$ as in 
Equation~\eqref{EquationFiniteWordSubshitfAlphabeth} and an interval $I\subset\domain(W_n)$. Parametrizing it by arc length $[0,|I|]\to I$, $t\mapsto\xi(t)$ we get 
$$
\big|F_{W_0,\dots,W_n}(I)\big|
=
\int_0^{|I|}\big|d/dt \big(F_{W_0,\dots,W_n}\circ\xi(t)\big)\big|dt
=
\int_0^{|I|}|D_{\xi(t)}F_{W_0,\dots,W_n}|dt,
$$
and thus
\begin{equation}
\label{EquationArcLengthAndDerivative}
|I|\cdot\inf_{\xi\in I}|D_\xi F_{W_0,\dots,W_n}|
\leq 
\big|F_{W_0,\dots,W_n}(I)\big|
\leq
|I|\cdot\sup_{\xi\in I}|D_\xi F_{W_0,\dots,W_n}|.
\end{equation}

Equation~\eqref{EquationArcLengthAndDerivative} refers to arc length of segments in $\partial\DD$, which is comparable with the diameter, as subsets of $\CC$. 
Proposition~\ref{PropositionSizeGapsAndIntervals} below holds for both diameter and arc length. 

\begin{proposition}
\label{PropositionSizeGapsAndIntervals}
There is an uniform constant $C>0$ such that for any finite block of words $w_k=(W_1,\dots,W_k)$ as in Equation~\eqref{EquationFiniteWordSubshitfAlphabeth} the following holds. 
\begin{enumerate}
\item
The cylinder
$
[W_1,\dots,W_k]_\EE
$ 
has size
$$
\big|
[W_1,\dots,W_k]_\EE
\big|
<C\cdot\theta^k.
$$
\item
For any $0\leq m\leq k$ we have 
$$
\big|[W_1,\dots,W_k]_\EE\big|
\leq
C\cdot\theta^{k-m}\cdot \big|[W_1,\dots,W_m]_\EE\big|.
$$
\item
For any interval $I\subset\domain(W_k)$ we have
$$
\big|F_{W_1,\dots,W_k}(I)\big|\geq \frac{|I|}{C}\cdot\big|[W_1,\dots,W_k]_\EE\big|.
$$
\end{enumerate}
\end{proposition}

\begin{proof}
Part (1) follows directly from Corollary~\ref{CorollaryExponentiallyContractingDistance}. 
Let $C'$ be the constant in Corollary~\ref{CorollaryLipschitzConstantDerivative}. Set 
$
E:=[W_{m+1},\dots,W_k]_\EE\subset\domain(W_m)
$. 
Observe that there exists an uniform constant $c>0$ such that $c<|\domain(W)|<2\pi$ for any $W\in\cW$. Part (2) follows from 
\begin{align*}
\big|[W_1,\dots,W_k]_\EE\big|
&
=
\big|F_{W_0,\dots,W_m}(E)\big|
\leq
|E|\cdot\sup_{\xi\in E}|D_\xi F_{W_1,\dots,W_m}|
\\
&
\leq
C\cdot\theta^{k-m}
\cdot
\sup_{\xi\in E}|D_\xi F_{W_1,\dots,W_m}|
\\
&
\leq
C\cdot\theta^{k-m}\cdot C'\cdot 
\inf_{\xi\in\domain(W_m)}|D_\xi F_{W_1,\dots,W_m}|
\\
&
\leq
\frac{C\cdot\theta^{k-m}\cdot C'}{\big|\domain(W_m)\big|}
\cdot
|\domain(W_m)|\cdot\inf_{\xi\in\domain(W_m)}|D_\xi F_{W_1,\dots,W_m}|
\\
&
\leq
\frac{C\cdot\theta^{k-m}\cdot C'}{\big|\domain(W_m)\big|}
\cdot
\big|[W_1,\dots,W_m]_\EE\big|,
\end{align*}
where the first and the last inequalities follow from   
Equation~\eqref{EquationArcLengthAndDerivative}, the second from Part (1) of this Proposition, the third from Part(1) of Corollary~\ref{CorollaryLipschitzConstantDerivative}. Part (3) holds because 
Equation~\eqref{EquationArcLengthAndDerivative} and Part (1) of 
Corollary~\ref{CorollaryLipschitzConstantDerivative} give
\begin{align*}
\big|F_{W_1,\dots,W_k}(I)\big|
&
\geq
|I|\cdot\inf_{\xi\in I}|D_{\xi}F_{W_1,\dots,W_k}|
\geq
\frac{|I|}{C'}\cdot\sup_{\xi\in\domain(W_k)}|D_{\xi}F_{W_1,\dots,W_k}|
\\
&
=
\frac{|I|}{C'\cdot|\domain(W_k)|}
\cdot
|\domain(W_k)|
\cdot
\sup_{\xi\in\domain(W_k)}|D_{\xi}F_{W_1,\dots,W_k}|
\\
&
\geq
\frac{|I|}{C'\cdot|\domain(W_k)|}\cdot\big|[W_1,\dots,W_k]_\EE\big|.
\end{align*}
\end{proof}

\section{Transfer operator and dimension}
\label{SectionTransferOperatorAndDimension}

\subsection{The Theorem of Ruelle-Perron-Frobenius}
\label{SectionRuellePerronFrobenius}

Let $\sigma:\Sigma\to\Sigma$ be the shift map and $d_\theta(\cdot,\cdot)$ be the distance on $\Sigma$  introduced in \S~\ref{SectionShiftSpace}, where $0<\theta<1$ is the uniform constant provided by  
Lemma~\ref{LemmaUniformContractionBranches}. Let $C(\Sigma)$ be the Banach space of continuous functions $f:\Sigma\to\CC$ with norm 
$$
\|f\|_\infty:=\sup_{w\in\Sigma}|f(w)|.
$$ 
Following \S~1 in \cite{BowenEquilibriumStates} and \S~2 in \cite{ParryPollicott}, for a fixed 
$
\varphi\in C(\Sigma)
$ 
(also said \emph{potential}) we consider the \emph{transfer operator}
$
\cL_\varphi:C(\Sigma)\to C(\Sigma)
$, 
also known as \emph{Ruelle operator} (see \cite{Ruelle}), which takes any 
$f\in C(\Sigma)$ to the function $\cL_\varphi f\in C(\Sigma)$ given by 
\begin{equation}
\label{EquationTransferOperatorGeneral}
(\cL_\varphi f)(w):=
\sum_{\sigma(w')=w}e^{\varphi(w')}f(w').
\end{equation}
The space $\cM(\Sigma)$ of Borel probability measures on $\Sigma$ is identified with the set of $\mu$ in the dual space $C(\Sigma)^\ast$ such that $\mu(1)=1$ and $\mu(f)\geq0$ if $f$ is real and positive. The dual operator 
$
\cL^\ast_\varphi:C(\Sigma)^\ast\to C(\Sigma)^\ast
$ 
sends the Borel probability measure $\mu$ to the continuous functional 
$\cL_\varphi^\ast\mu$ which acts on any $f\in C(\Sigma)$ as
$$
(\cL_\varphi^\ast\mu)(f):=\int (\cL_\varphi f)d\mu.
$$

Following \S~1 in \cite{BowenEquilibriumStates} and \S~1 in \cite{ParryPollicott}, for any 
$
f\in C(\Sigma)
$ 
and for any $n\in\NN$ define
$$
\variation_n(f):=
\sup\{|f(w)-f(w')|:W_r=W'_r\quad\forall\quad r=0,\dots,n\}
$$
Let 
$
\lipschitz(\Sigma,\theta)\subset C(\Sigma)
$ 
be the subspace of Lipschitz functions $f:\Sigma\to\CC$ with respect to the distance $d_\theta$, which is a Banach space for the norm
$$
\|f\|_\theta:=\|f\|_\infty+\lipschitz(f),
$$
where $\lipschitz(f)$ is defined in Equation~\eqref{EquationLipschitzConstant}. In the notation above,
$
\lipschitz(\Sigma,\theta)
$ 
is the space of functions $f$ such that there exist $C>0$ with 
$$
\variation_n(f)\leq C\cdot\theta^n\quad\forall\quad n\in\NN.
$$
For a potential 
$
\varphi\in\lipschitz(\Sigma,\theta)
$, 
the operator $\cL_\varphi$ in 
Equation~\eqref{EquationTransferOperatorGeneral} acts as a bounded linear operator on  
$
\lipschitz(\Sigma,\theta)
$ 
(see \cite{ParryPollicott} at page 19, or \cite{Baladi} at page 30).

\smallskip

Fix a continuous real-valued potential $\varphi:\Sigma\to\RR$. A probability (non necessarily $\sigma$-invariant) measure $m$ on $\Sigma$ is a \emph{Gibbs measure} if there exist constants 
$P=P(\varphi)\in\RR$ and $C=C(\varphi)>1$ such that for any 
$
w=(W_r)_{r\in\NN}\in\Sigma
$ 
and any $n\in\NN$ we have 
$$
\frac{1}{C}
\leq
\frac
{m\big([W_0,\dots,W_{n-1}]_\Sigma\big)}
{\exp\big(-nP+\sum_{k=0}^{n-1}\varphi(\sigma^k(w))\big)}
\leq
C.
$$
The real constant $P=P(\varphi)$ is called the \emph{pressure} for the potential $\varphi$. 
Theorem~\ref{TheoremRuellePerronFrobenius} resumes Theorem~2.2 and Corollary~3.2.1 in \cite{ParryPollicott}. See also Theorem~1.7 and Theorem~1.16 in \cite{BowenEquilibriumStates}. Basic notions on spectra of bounded linear operators are recalled in  
\S~\ref{AppendixSpectralProjectors}.

\begin{theorem}
[Ruelle-Perron-Frobenius]
\label{TheoremRuellePerronFrobenius}
Let $\sigma:\Sigma\to\Sigma$ be a subshift of finite type determined by an aperiodic matrix $M$. For a positive real-valued potential
$
\varphi\in\lipschitz(\Sigma,\theta)
$ 
the following holds.
\begin{enumerate}
\item
There exists a simple real eigenvalue $\lambda>0$ for 
$
\cL_\varphi:C(\Sigma)\to C(\Sigma)
$, 
which corresponds to a real eigenfunction 
$
h\in\lipschitz(\Sigma,\theta)
$ 
with $h(w)>0$ strictly for any $w\in\Sigma$.
\item
The remainder of the spectrum of 
$
\cL_\varphi:\lipschitz(\Sigma,\theta)\to\lipschitz(\Sigma,\theta)
$ 
is contained in the ball $B(0,\rho)\subset\CC$ for some  $\rho<\lambda$.
\end{enumerate}
Moreover set 
$
\psi:=\varphi-\ln h\circ\sigma+\ln h-\ln\lambda
$, 
in terms of the eigenfunction and eigenvalue $h$ and $\lambda$ above, and let 
$\cL_\psi$ be the corresponding operator. 
\begin{enumerate}
\setcounter{enumi}{2}
\item
There exists an unique $\sigma$-invariant Borel probability measure $m$ on $\Sigma$ such that 
$
\cL_\psi^\ast(m)=m
$, 
moreover $m$ is a Gibbs measure for $\varphi$ with pressure 
$
P(\varphi)=\ln\lambda
$.
\end{enumerate}
Finally consider the measure $d\widehat{m}(w):=h^{-1}(w)dm(w)$ with normalization 
$
\int_\Sigma hd\widehat{m}=1
$.
\begin{enumerate}
\setcounter{enumi}{3}
\item
We have 
$
\lambda^{-n}\cdot \cL_\varphi^nf\to \big(\int fd\widehat{m}\big)\cdot h
$ 
uniformly for any $f\in C(\Sigma)$.
\end{enumerate}
\end{theorem}

\subsection{Potential from boundary expansion}

Let $\Pi:\Sigma\to\partial\DD$ be the map in 
Equation~\eqref{EquationCommutativeDiagramShiftExpansion}. For a sequence 
$w=(W_r)_{r\in\NN}\in\Sigma$ denote $W_0(w)$ the first letter of $w$. Consider the continuous function 
$\varphi_T:\Sigma\to\RR$ defined by 
\begin{equation}
\label{EquationPotentialBowenSeries}
\varphi_T(w):=
-\ln\left|D_{\Pi(w)}F^{-1}_{W_0(w)}\right|
=
\ln|D_{\Pi(\sigma w)}F_{W_0(w)}|,
\end{equation}
where $\varphi_T$ depends on $T$ because the alphabet of the sub-shift 
$
\sigma:\Sigma\to\Sigma
$ 
is $\cW_T$. The second equality in 
Equation~\eqref{EquationPotentialBowenSeries} follows observing that      
$
D_{\xi'}F_{W_0}=\big(D_\xi F_{W_0}^{-1}\big)^{-1}
$ 
for any $W_0\in\cW$ and $\xi$, where $\xi':=F_{W_0}^{-1}(\xi)$, and that 
$
F^{-1}_{W_0(w)}\big(\Pi(w)\big)=\Pi\big(\sigma(w)\big)
$ 
for any $w\in\Sigma$, according to 
Equation~\eqref{EquationCommutativeDiagramShiftExpansion} (and applying the resulting identity to 
$\xi:=\Pi(w)$). 

\begin{lemma}
\label{LemmaComputationExponentialBirkoffSum}
For any $w=(W_0,W_1,\dots)\in\Sigma$ and any $n\in\NN^\ast$ let $\xi_w:=\Pi(\sigma^n w)$. Then we have 
$
\xi_w\in\domain(W_{n-1})
$ 
and 
$$
\exp\bigg(\sum_{k=0}^{n-1}\varphi_T(\sigma^kw)\bigg)
=
|D_{\xi_w}F_{W_0,\dots,W_{n-1}}|.
$$
\end{lemma}

\begin{proof}
The first statement holds because for any $k=0,\dots,n-1$ we have obviously 
$$
\Pi\circ\sigma^{n-k}w=
[W_{n-k},W_{n-k+1},\dots]_\DD\in
[W_{n-k}]_\EE
\subset\domain(W_{n-k-1}).
$$
In particular 
$
W_0(\Pi\circ\sigma^{n-k}w)=W_{n-k}
$. 
Hence, for the map $\cF$ in 
Equation~\eqref{EquationDefinitionCuspidalAccelerationBowenSeriesMap}, we have 
$$
\xi_w=
\cF^k\big(\Pi\circ\sigma^{n-k}(w)\big)=
F^{-1}_{W_{n-1}}\circ\dots\circ F^{-1}_{W_{n-k}}
\big(\Pi\circ\sigma^{n-k}(w)\big).
$$
The Lemma follows because the chain rule gives 
\begin{align*}
\exp\bigg(\sum_{k=0}^{n-1}\varphi_T(\sigma^kw)\bigg)
&
=
\big|D_{\Pi\circ\sigma^n(w)}F_{W_{n-1}}\big|
\cdot\dots\cdot
\big|D_{\Pi\circ\sigma(w)}F_{W_0}\big|
\\
&
=
\big|D_{\xi_w}F_{W_{n-1}}\big|
\cdot\dots\cdot
\big|D_{F_{W_1}\circ\dots\circ F_{W_{n-1}}(\xi_w)}F_{W_0}\big|=
|D_{\xi_w}F_{W_0,\dots,W_{n-1}}|.
\end{align*}
\end{proof}

Theorem~\ref{TheoremRuellePerronFrobenius} can be applied to the function $\varphi_T$, according to  Lemma~\ref{LemmaDerivativeHasBoundedVariation} below, which corresponds to Lemma~4 in \cite{BowenQuasicircles}.

\begin{lemma}
\label{LemmaDerivativeHasBoundedVariation}
For any $T$ such that $0<T<+\infty$ strictly there exists a constant $C=C(T)>0$ such that for any 
$n\in\NN$ we have
$$
\variation_n(\varphi_T)\leq C\theta^n.
$$
\end{lemma}

\begin{proof}
Observe that if $w=(W_0,W_1,\dots)$ then $W_0=W_0(w)$ and 
$
\Pi(\sigma w)\in[W_1]_\EE\subset\domain(W_0)
$. 
The Lemma follows from Lemma~\ref{LemmaLipschitzConstantLogDerivative} (applied for blocks of length $k=1)$ and 
Lemma~\ref{LemmaCodingLipschitz}.
\end{proof}

\subsection{Dimension of Cantor sets}
\label{SectionDimensionCantorSet}

Fix a subset $E\subset\CC$. For $\rho>0$, a \emph{$\rho$-cover of $E$} is a countable collection 
$\{B_i:i=0,1,2\dots\}$ with $E \subset \bigcup_i B_i $, where any $B_i$ is a ball with diameter 
$|B_i|\leq\rho$. Such a cover exists for every $\rho>0$. Fix $s$ with $0\leq s\leq 1$ and define
$$
H^s_\rho(E):=
\inf\sum_{i\in\cI}|B_i|^s,
$$
where the infimum is taken over all $\rho$-covers of $E$. The \emph{Hausdorff $s$-measure} $H^s(E)$ of $E$ is defined by
$$
H^s(E) :=
\lim_{\rho\to0} H^s_\rho(E)
\; = \;
\sup_{\rho>0}H^s_\rho(E).
$$
The \emph{Hausdorff dimension} $\dim_H E$ of a set $E$ is defined by
$$
\dim_H(E):=
\inf\left\{s:H^{s} (E) =0 \right\} =
\sup\left\{s:H^{s} (E) = \infty \right\}.
$$
The following Lemma is a classical fact, for a proof see Theorem~5.7 in \cite{Mattila}.

\begin{lemma}
Let $\nu$ be a probability measure on $\CC$ and let $E:=\support(\nu)$ be its support. Assume that there exists $s>0$, $C>1$ and $r_0>0$ such that for any $r<r_0$ and any $x\in E$ we have
\begin{equation}
\label{EquationDimensionMeasure}
\frac{1}{C}\cdot r^s
\leq
\nu\big(B(x,r)\big)
\leq C\cdot r^s.
\end{equation}
Then we have 
$
C^{-1}\leq H^s(E) \leq 5^s\cdot C
$ 
and in particular $\dim_H(E)=s$.
\end{lemma}

Fix $T<\infty$ and consider $\varphi_T$ in Equation~\eqref{EquationPotentialBowenSeries}. 
For any $0\leq s\leq 1$ consider the transfer operator 
$
\cL_{(s,T)}:\lipschitz(\Sigma,\theta)\to\lipschitz(\Sigma,\theta)
$ 
in Equation~\eqref{EquationTransferOperatorGeneral} corresponding to the potential 
$s\cdot\varphi_T$, that is the operator acting on $f\in\lipschitz(\Sigma,\theta)$ by 
\begin{equation}
\label{EquationTransferOperatorMoebiusShift}
(\cL_{(s,T)} f)(w)
:=
\sum_{\sigma(w')=w}\frac{1}{|D_{\Pi(w')}F^{-1}_{W_0(w')}|^s}f(w')
=
\sum_{\sigma(w')=w}|D_{\Pi(w)}F_{W_0(w')}|^sf(w'),
\end{equation}
where the second inequality follows from 
Equation~\eqref{EquationPotentialBowenSeries}. 
According to Theorem~\ref{TheoremRuellePerronFrobenius}, let $\lambda(s,T)>0$ be the maximal eigenvalue of $\cL_{(s,T)}$ and $m_{(s,T)}$ be the Gibbs measure for $\cL_{(s,T)}$ as in Point (3). Consider the measure on $\partial\DD$ given by 
\begin{equation}
\label{EquationDefinitionGibbsMeasureCantor(s,T)}
\nu_{(s,T)}:=\Pi_\ast(m_{(s,T)}),
\end{equation}
that is
$
\nu_{(s,T)}(I):=m_{(s,T)}\big(\Pi^{-1}(I)\big)
$ 
for any Borel set $I\subset \partial\DD$. In particular for any cylinder 
$$
\nu_{(s,T)}\big([W_0,\dots,W_n]_\EE\big)
=
m_{(s,T)}\big([W_0,\dots,W_n]_\Sigma\big).
$$

Proposition~\ref{PropositionDimensionAndSpectralRadius} below is adapted from Lemma~10 in \cite{BowenQuasicircles}.

\begin{proposition}
\label{PropositionDimensionAndSpectralRadius}
Fix $T$ with $0<T<\infty$ strictly and let $\cL_{(s,T)}$ be the operator as above.
\begin{enumerate}
\item
The function $s\mapsto P(s):=\log\lambda(s,T)$ is continuous and strictly decreasing monotone with $P(0)>0$ and $P(1)\leq0$. 
\item
We have $\dim_H(\EE_T)=s_T$ where $s_T$ is the unique solution of 
\begin{equation}
\label{EquationBowenFormulaDimension}
P(s_T)=0.
\end{equation}
Moreover for $s=s_T$ the measure $\nu_{(s_T,T)}$ is equivalent to the restriction to $\EE_T$ of the Hausdorff measure $H^{s_T}$.
\end{enumerate}
\end{proposition}

\begin{proof}
Part (1) corresponds to standard a property of pressure. See pages 20-21 in \cite{BowenQuasicircles} (see also pages 15-16 in \cite{Bedford} and page 44 in \cite{ParryPollicott}). We prove Part (2) in 3 steps.

\emph{Step $(1)$.} Following Lemma~5 in \cite{BowenQuasicircles}, we show that there exists an uniform constant $C_1>1$ such that for any $w=(W_r)_{r\in\NN}\in\Sigma$ and any $n\in\NN$ we have
\begin{equation}
\label{EquationLengthAndBirkoffSums}
\frac{1}{C_1}
\leq
\frac
{\big|[W_0,\dots,W_{n-1}]_\EE\big|}
{\exp\left(\sum_{k=0}^{n-1}\varphi_T(\sigma^k(w))\right)}
\leq C_1.
\end{equation}
Indeed we have 
$
[W_0,\dots,W_{n-1}]_\EE=F_{W_0,\dots,W_{n-1}}(D)
$, 
where $D:=\domain(W_{n-1})$, and 
Equation~\eqref{EquationArcLengthAndDerivative} gives
$$
|D|\cdot\inf_{\xi\in D}|D_\xi F_{W_0,\dots,W_{n-1}}|
\leq
\big|[W_0,\dots,W_{n-1}]_\EE\big|
\leq 
|D|\cdot\sup_{\xi\in D}|D_\xi F_{W_0,\dots,W_{n-1}}|.
$$
Thus Equation~\eqref{EquationLengthAndBirkoffSums} follows observing that 
$
\xi_w:=\Pi(\sigma^nw)\in\domain(W_{n-1})
$ 
and combining Lemma~\ref{LemmaComputationExponentialBirkoffSum} with Part (1) of 
Corollary~\ref{CorollaryLipschitzConstantDerivative}. 

\emph{Step $(2)$.} 
For the specific value $s=s_T$ we have $P(s_T)=0$. 
Point (3) of Theorem~\ref{TheoremRuellePerronFrobenius} implies that for the measure 
$\nu_{(s_T,T)}:=\Pi_\ast(m_{(s_T,T)})$ there exists a constant $C_2=C_2(T)>1$ such that for any 
$w=(W_r)_{r\in\RR}\in\Sigma$ and any $n\in\NN$ we have 
\begin{equation}
\label{EquationMeasureAndBirkoffSums}
\frac{1}{C_2}
\leq
\frac
{\nu_{(s_T,T)}\big([W_0,\dots,W_{n-1}]_\EE)\big)}
{\exp\left(s_T\cdot\sum_{k=0}^{n-1}\varphi_T(\sigma^k(w))\right)}
\leq C_2.
\end{equation}

\emph{Step $(3)$:} finally we show that for $s=s_T$ the measure $\nu_{(s_T,T)}$ satisfies 
Equation~\eqref{EquationDimensionMeasure}. Fix $r>0$ and $\xi\in\EE_T$ and consider the euclidian ball  $B(\xi,r)\subset\CC$. Let $w=(W_r)_{r\in\NN}\in\Sigma$ be the sequence of cuspidal words arising from the boundary expansion of $\xi$, that is the sequence such that 
$
\xi=[W_0,W_1,W_2\dots]_\DD
$. 
Let $n\in\NN$ be such that 
$$
\big|[W_0,\dots,W_{n}]_\EE\big|
\leq r<
\big|[W_0,\dots,W_{n-1}]_\EE\big|.
$$
The first inequality above implies 
$
[W_0,\dots,W_{n}]_\EE\subset B(\xi,r)
$. 
Observe that Part (3) of 
Proposition~\ref{PropositionSizeGapsAndIntervals} gives 
$$
\big|[W_0,\dots,W_{n}]_\EE\big|=
\big|F_{W_0,\dots,W_{n-1}}([W_{n}]_\EE)\big|\geq
\frac{\big|[W_{n}]_\EE\big|}{C_3}\big|[W_0,\dots,W_{n-1}]_\EE\big|\geq
\kappa\big|[W_0,\dots,W_{n-1}]_\EE\big|,
$$
where here $C_3$ denotes the constant in Proposition~\ref{PropositionSizeGapsAndIntervals} and  
$\kappa=\kappa(T)>0$ is a positive constant with $\big|[W]_\EE\big|\geq\kappa C_3$ for any $W\in\cW_T$. 
The lower bound in Equation~\eqref{EquationDimensionMeasure} follows from  
\begin{align*}
\nu_{(s_T,T))}\big(B(\xi,r)\big)
&
\geq
\nu_{(s_T,T))}\big([W_0,\dots,W_{n}]_\EE\big)
\geq
\frac{1}{C_2}\cdot\bigg(\exp\big(\sum_{k=0}^{n}\varphi_T(\sigma^k(w))\big)\bigg)^{s_T}
\\
&
\geq
\frac{\big|[W_0,\dots,W_{n}]_\EE\big|^{s_T}}{C_2\cdot C_1^{s_T}}
\geq
\frac{\big|[W_0,\dots,W_{n-1}]_\EE\big|^{s_T}}{C_2\cdot(\kappa C_1)^{s_T}}
\geq
\frac{r^{s_T}}{C_2\cdot(\kappa C_1)^{s_T}},
\end{align*}
where the second inequality holds by Equation~\eqref{EquationMeasureAndBirkoffSums} and the third by Equation~\eqref{EquationLengthAndBirkoffSums}. On the other hand, let $m\leq n$ be maximal such that 
$$
\nu_{(s_T,T)}\big(B(\xi,r)\big)
=
\nu_{(s_T,T)}\big(B(\xi,r)\cap[W_0,\dots,W_{m}]_\EE\big).
$$
For $a\in\cA$ let $\cI(a)$ be the set of right cuspidal words $W=(b_0,\dots,b_l)$ with $b_0=a$ and 
$|W|>T$. Observe that $I_a:=\bigcup_{W\in\cI(a)} [W]_\EE$ is an arc in $\partial\DD$ and we can choose the above constant $\kappa=\kappa(T)>0$ so that $|I_a|\geq\kappa C_3$ for any $a\in\cA$. 
By maximality of $m$ above, there exists some $a\in\cA$ such that\footnote{The first case holds if there exists $W\in\cW_T$ having first letter different from the first letter of $W_{m+1}$ and such that 
$\nu_{(s_T,T)}$ gives positive measure both to 
$[W_0,\dots,W_m,W_{m+1}]_\EE$ and to $[W_0,\dots,W_m,W]_\EE$.}
$$
\text{either }\quad
F_{W_0,\dots,W_m}(I_a)\subset B(\xi,r)
\quad\text{ or }\quad
F_{W_0,\dots,W_m,W_{m+1}}(I_a)\subset B(\xi,r).
$$
Assume 
$
F_{W_0,\dots,W_m}(I_a)\subset B(\xi,r)
$ 
without loss of generality (otherwise replace $\kappa$ below by $\kappa^2$). 
Proposition~\ref{PropositionSizeGapsAndIntervals} and the definition of $n$ give 
\begin{align*}
|F_{W_0,\dots,W_m}(I_a)|
&
\geq 
\frac{|I_a|}{C_3}\big|[W_0,\dots,W_{m}]_\EE\big|
\geq
\kappa\cdot\big|[W_0,\dots,W_{m}]_\EE\big|
\\
&
\geq
\frac{\kappa\cdot\big|[W_0,\dots,W_{n-1}]_\EE\big|}{C_3\cdot\theta^{n-m-1}}
\geq
\frac{\kappa\cdot r}{C_3\cdot\theta^{n-m-1}}.
\end{align*}
We have  
$
|B(\xi,r)\cap\partial\DD|\leq (\pi/3)\cdot r
$, 
so that the inequality above gives an upper bound for $|n-m|$. In other words there exists a constant $N=N(T)\in\NN$ such that for any $r$ and $n$ as above we have
$$
B(\xi,r)\cap\EE_T\subset[W_0,\dots,W_{n-N}]_\EE.
$$
The upper bound in Equation~\eqref{EquationDimensionMeasure} follows by a chain of inequalities similar to those proving the lower bound.
\end{proof}

\section{Transfer operator on the circle}
\label{SectionTransferOperatorOnTheCircle}

The main result of this section is 
Theorem~\ref{TheoremTransferOperatorCircle}. Consider parameters $s,T$ with $1/2<s<\infty$ and 
$0<T\leq\infty$. Equation~\eqref{EquationTransferOperatorMoebiusDisc} below defines a transfer operator  $L_{(s,T)}$, acting on the Banach spaces $\cB$ and $\cB_T$ described in the next 
\S~\ref{SectionBanachSpacesPiecewiseLipschitzFunctions}. 
As in \S~\ref{SectionEstimatesContractionDistortion} all constant are uniform, that is depend only on the geometry of $\Omega_\DD$ and not on $s,T$, unless explicitly stated. For simplicity of notation, for any $a\in\cA$ and any finite sequence $w_k=(W_1,\dots,W_k)$ of cuspidal words as in 
Equation~\eqref{EquationFiniteWordSubshitfAlphabeth}, write 
$$
[a]
\quad
\textrm{ instead of }
\quad 
[a]_\DD
\quad
\textrm{ and }
\quad 
[w_k]
\quad
\textrm{ instead of }
\quad  
[w_k]_\EE.
$$

\subsection{Banach space of piecewise Lipschitz functions}
\label{SectionBanachSpacesPiecewiseLipschitzFunctions}

Let $\lipschitz(X)$ be the space of Lipschitz functions $f:X\to\CC$ on a metric space $X$. 
Any $f\in\lipschitz(X)$ has Lipschitz constant 
\begin{equation}
\label{EquationLipschitzConstant}
\lipschitz(f):=\sup_{x\not=y}\frac{|f(y)-f(x)|}{|y-x|}.
\end{equation}
If $X$ is compact, then the space $C(X)$ of continuous functions $f:X\to\CC$ with norm 
$\|f\|_\infty:=\sup_{x\in X}|f(x)|$ is a Banach space. Moreover $\lipschitz(X)$ is also a Banach space  with norm $\|f\|_\ast:=\|f\|_\infty+\lipschitz(f)$. The unitary ball for the norm $\|\cdot\|_\ast$ is relatively compact for the topology of the norm $\|\cdot\|_\infty$, by the Theorem of Ascoli-Arzela.

\smallskip

Fix $a\in\cA$. Any $f\in\lipschitz\big([a]\big)$ has an unique Lipschitz extension to the closure 
$\overline{[a]}$, with the same Lipschitz constant. Hence in particular $f$ is continuous and bounded. Consider the space of functions
$$
\cB:=
\left\{
f:\partial\DD\to\CC:
f|_{[a]}\in\lipschitz\big([a]\big)\quad\forall\quad a\in\cA
\right\}.
$$
Consider the norm $\|\cdot\|_\ast:\cB\to\RR_+$ defined by 
$$
\|f\|_\ast:=\|f\|_\infty+\lipschitz(f)
\quad
\textrm{ where }
\quad
\lipschitz(f):=\max_{a\in\cA}\lipschitz(f|_{[a]}).
$$
The axioms of a normed vector space are easily verified. Moreover $\cB$ is a Banach space, because   
$\lipschitz\big(\overline{[a]}\big)$ is a Banach space for any $a\in\cA$. For the same reason, the unitary ball in $\cB$ for the norm $\|\cdot\|_\ast$ is relatively compact for the norm 
$\|\cdot\|_\infty$. 

\medskip
 
Fix $T$ with $0<T<+\infty$ and let $\cB_T$ be the space of functions on $\EE_T$ which are restrictions to $\EE_T$ of functions in $\cB$, that is 
$$
\cB_T:=\{g:\EE_T\to\CC:\exists f\in\cB:g=f|_{\EE_T}\}
$$

For $0<T<+\infty$ the operator $L_{(s,T)}$ in 
Equation~\eqref{EquationTransferOperatorMoebiusDisc} below acts both on $\cB$ and $\cB_T$, while for 
$T=\infty$ the operator $L_{(s,\infty)}$ acts only on $\cB$. We write $\cB_{\infty}:=\cB$ for convenience. The choice of such Banach spaces is motivated by 
Remark~\ref{RemarkBigImageProperty} below. 
  
\begin{remark}
\label{RemarkBigImageProperty}
Fix $a\in\cA$. Consider any pair of cuspidal words $W_0=(a_0,\dots,a_n)$ and $W_1=(b_0,\dots,b_k)$ with $a_0=b_0=a$, where $n,k\geq0$ and in general $n\not=k$. If $W=(c_0,\dots,c_m)$ is any other cuspidal word, then 
$$
M_{W,W_0}=M_{W,W_1},
$$
that is the concatenation $W\ast W_0$ is allowed by the transition matrix in 
Equation~\eqref{EquationDefinitionTransitionMatrix} if and only if the concatenation $W\ast W_1$ is allowed. In other words, for given $W,W_0\in\cW$ the admissibility of the concatenation $W\ast W_0$ depends only on $W$ and on the first letter of $W_0$. 
\end{remark}

According to Remark~\ref{RemarkBigImageProperty}, we have a well defined set of cuspidal words
$$
\cW(a):=
\big\{W\in\cW:
M_{W,W_0}=1\textrm{ for any }W_0=(a_0,\dots,a_n)\in\cW\textrm{ with }a_0=a
\big\}.
$$
For fixed $T>0$ set 
$$
\cW(a,T):=\{W\in\cW(a):|W|\leq T\}.
$$ 
The set $\cW(a,T)$ is finite for $0<T<+\infty$. For $T=+\infty$ we have 
$\cW(a,T)=\cW(a)$, which is infinite countable. For $k\in\NN$ let $\cW(k,a,T)$ be the set of finite blocks $w_k:=(W_1,\dots,W_k)$ as in 
Equation~\eqref{EquationFiniteWordSubshitfAlphabeth} with $W_k\in\cW(a,T)$. In other words we have 
$$
w_k\in\cW(k,a,T)
\quad
\Leftrightarrow
\quad
[a]\subset\domain(W_k).
$$

\subsection{Preliminary uniform estimates}
\label{SectionPreliminaryUniformEstimates}

Refer to the notation in 
Equation~\eqref{EquationCoefficientsSU(1,1)} and observe that if $F\in\sugroup(1,1)$ is parabolic fixing $\xi_0\in\partial\DD$ and conjugated to $z\mapsto z+\mu$, then its coefficients 
$\alpha=\alpha(F)$ and $\beta=\beta(F)$ are given by
\begin{equation}
\label{EquationCoefficientsAlphaBetaParabolic}
\alpha=1+i(\mu/2)
\quad\text{ and }\quad
\beta=\overline{\xi_0}\cdot i(\mu/2).
\end{equation}

\begin{lemma}
\label{LemmaQuadraticAsymptoticDerivative}
There is an uniform constant $C>1$ such that for any cuspidal word $W\in\cW$ with $|W|>0$ and any 
$\xi\in\domain(W)$ we have
$$
\frac{1}{C\cdot|W|^2}\leq |D_\xi F_W|\leq \frac{C}{|W|^2}.
$$
\end{lemma}

\begin{proof}
Lemma~\ref{LemmaCombinatorialPropertiesCuspidal} implies that any cuspidal word $W$ can be decomposed as 
$$
W=V\ast\underbrace{P\ast\dots\ast P}_{k\textrm{ times }}
\quad\text{ with }\quad
k\geq0,
$$
where $P\in\cW$ is such that $F_P$ is parabolic fixing a vertex of $\Omega_\DD$ and $V$ is a cuspidal word which does not contain $P$ as a factor. In particular $F_P$ is conjugated to $z\mapsto z+\mu$, for some $\mu>0$ with $|\mu|$ uniformly bounded. Equation~\eqref{EquationDefinitionGeometricLenght} implies that there is an uniform constant $C_1>1$ such that in the decomposition above we have 
$$
C_1^{-1}\cdot k\leq |W|\leq C_1\cdot k.
$$
We have $F_W=F_V\circ F_P^k$, as elements of $\Gamma_0$. Denote $\omega_k$ and $\omega_{k+1}$ the poles of $F_P^k$ and $F_P^{k+1}$ respectively. For any 
$\xi\in\domain(W)$ Equation~\eqref{EquationCoefficientsAlphaBetaParabolic} gives
$$
\frac{1}{|(k+1)(\mu/2)|^2|\xi-\omega_{k+1}|^2}
=
|D_\xi F_P^{k+1}|
\leq
|D_\xi F_W|
\leq 
|D_\xi F_P^k|
=
\frac{1}{|k(\mu/2)|^2|\xi-\omega_{k}|^2},
$$
where the inequalities follow from Lemma~\ref{LemmaWeakContractionAdmissibleWords} factorizing 
$|D_\xi F_W|$ with the chain rule. The distance from poles is bounded by 
Lemma~\ref{LemmaDistanceFromPoles}. The statement follows.
\end{proof}

For any $0<T\leq \infty$ and any $a\in\cA$ let 
$
\cV(a,T):=\cW(a)\setminus\cW(a,T)
$, 
that is the set of cuspidal words $W$ with $|W|>T$ and $[a]\subset\domain(W)$.

\begin{corollary}
\label{CorollaryUniformlySummableSeries}
Fix $s_0>1/2$. There exists an uniform constant $C>0$ such that for any $s,T$ with 
$s_0\leq s<\infty$ and $0<T\leq\infty$, for any $a\in\cA$ and any 
$\xi\in[a]$ we have
$$
\sum_{W\in\cW(a,T)}|D_\xi F_W|^s\leq C
\quad
\textrm{ and }
\quad
\sum_{W\in\cW(a,T)}\big|\ln|D_\xi F_W|\big|\cdot|D_\xi F_W|^s\leq C.
$$
Moreover we also have
$$
\sum_{W\in\cV(a,T)}|D_\xi F_W|^s\leq C\cdot\left(\frac{1}{T}\right)^{2s-1}.
$$
\end{corollary}

\begin{proof}
The statement follows directly from Lemma~\ref{LemmaQuadraticAsymptoticDerivative}, because there exist uniform constants $N>1$ and $C_1$ such that for any $T>0$ we have 
$$
1\leq \sharp\{W\in\cW:T\leq |W|\leq T+ N\}\leq C_1.
$$
\end{proof}

\subsection{Transfer operator on the circle and its spectrum}
\label{SectionTransferOperatorCircleAndSpectrum}

For $s,T$ with $s>1/2$ and $0<T\leq+\infty$ let $L_{(s,T)}:\cB\to\cB$ and 
$L_{(s,T)}:\cB_T\to\cB_T$ be the operators defined by
\begin{equation}
\label{EquationTransferOperatorMoebiusDisc}
L_{(s,T)}f(\xi)=
\sum_{W\in\cW(a,T)}|D_\xi F_W|^s\cdot f(F_W\cdot\xi)
\quad
\textrm{ for }
\quad
\xi\in[a].
\end{equation}
The definition makes sense also for $f\in\cB_T$ because for any $a\in\cA$ the maps $F_W$ with $W\in\cW(a,T)$ leave $\EE_T$ invariant. For $T=+\infty$ the sum above is over all $W$ in the infinite set $\cW(a)$, and the only invariant set is the entire circle $\partial\DD$. The expression of the $k$-th iterated of $L_{(s,T)}$ is 
$$
L_{(s,T)}^kf(\xi)=
\sum_{w_k\in\cW(k,a,T)}|D_\xi F_{w_k}|^s\cdot f(F_{w_k}\cdot\xi)
\quad
\textrm{ for }
\quad
\xi\in[a].
$$

\begin{lemma}
\label{LemmaTransferOperatorsAreUniformlyBounded}
The operator $L_{(s,T)}$ is a well defined bounded linear operator both on $\cB$ and on $\cB_T$. Moreover there exists $C>0$ such that for any $s,T$, both on $\cB$ and on $\cB_T$, we have 
$$
\|L_{(s,T)}\|_\ast\leq C.
$$
\end{lemma}

\begin{proof}
Fix $f\in\cB$. For any $a\in\cA$ and any $x,y\in[a]$ we have 
\begin{align*}
&
\big|L_{(s,T)}f(y)-L_{(s,T)}f(x)\big|\leq
\\
&
\sum_{W\in\cW(a,T)}
|D_xF_W|^s\cdot\big|f(F_Wy)-f(F_Wx)\big|
+
f(F_Wy)\cdot\big||D_yF_W|^s-|D_yF_W|^s\big|\leq
\\
&
\sum_{W\in\cW(a,T)}
|D_xF_W|^s\cdot\lipschitz(f)\cdot\big|F_W(y)-F_W(x)\big|
+
\|f\|_\infty\cdot \|DF_W\|^s_\infty\cdot C_1\cdot|y-x|\leq
\\
&
\bigg(\sum_{W\in\cW(a,T)}\|DF_W\|^s_\infty\bigg)
\cdot
\big(\theta\cdot\lipschitz(f)+\|f\|_\infty\cdot C_1\big)
\cdot|y-x|.
\end{align*} 
The second inequality follows from Part (2) of  
Corollary~\ref{CorollaryLipschitzConstantDerivative} applied to $k=1$ (here $C_1$ denotes the uniform constant in the Corollary). The third inequality follows from 
Corollary~\ref{CorollaryExponentiallyContractingDistance}, applied to $k=1$. 
Corollary~\ref{CorollaryUniformlySummableSeries} gives 
$$
\lipschitz\big(L_{(s,T)}f\big)
\leq
C_2\cdot(\theta\cdot\lipschitz(f)+\|f\|_\infty\cdot C_1)
\leq
C_2\cdot(\theta+C_1)\cdot\|f\|_\ast,
$$
where here $C_2$ denotes the uniform constant in Corollary~\ref{CorollaryUniformlySummableSeries}. The inequality above implies that $L_{(s,T)}f\in\cB$. Moreover for any 
$a\in\cA$ and any $\xi\in[a]$ we have 
$$
\big|L_{(s,T)}f(\xi)\big|=
\sum_{W\in\cW(a,T)}|D_\xi F_W|^s\cdot \big|f(F_W\cdot\xi)\big|
\leq
\bigg(\sum_{W\in\cW(a,T)}\|DF_W\|^s_\infty\bigg)\cdot\|f\|_\infty.
$$
Thus 
$
\|L_{(s,T)}f(\xi)\|_\infty\leq C_2\cdot\|f\|_\infty
$ 
again by Corollary~\ref{CorollaryUniformlySummableSeries}. All the involved constants are uniform, thus the statement follows for the operator $L_{(s,T)}:\cB\to\cB$. The same result holds for 
$L_{(s,T)}:\cB_T\to\cB_T$, indeed the inequality above hold for any $x,y$ in $\partial\DD$, and thus in particular restricting to $x,y$ in $\EE_T$. The Lemma is proved.
\end{proof}

Consider $(s,T)$ with $T<\infty$. Let 
$
\cL_{(s,T)}:\lipschitz(\Sigma,\theta)\to\lipschitz(\Sigma,\theta)
$ 
be the operator in Equation~\eqref{EquationTransferOperatorMoebiusShift}. Let $\lambda(s,T)$ and $h_{(s,T)}$ be the leading eigenvalue and eigenfunction of $\cL_{(s,T)}$, as in 
Point (1) of Theorem~\ref{TheoremRuellePerronFrobenius}. Consider the probability measure 
\begin{equation}
\label{EquationMeasureMu(s,T)}
\mu_{(s,T)}=\Pi_\ast(\widehat{m}_{(s_T,T)})
\end{equation}
where $\widehat{m}_{(s,T)}$ is as in Point (4) of 
Theorem~\ref{TheoremRuellePerronFrobenius} and $\Pi:\Sigma\to\EE_T$ is the map in  
\S~\ref{SectionSubshiftAndCantorSetBoundary}. According to Lemma~\ref{LemmaCodingLipschitz} the map $\Pi:\Sigma\to\EE_T$ is continuous with continuous inverse, thus we have a bounded invertible operator:
$$
H:\big(C(\EE_T),\|\cdot\|_\infty\big)\to \big(C(\Sigma),\|\cdot\|_\infty\big)
\quad;\quad
H(f):=f\circ \Pi.
$$
Fix $f\in C(\EE_T)$ and $w=(W_0,W_1,\dots)\in\Sigma$. Let $a\in\cA$ such that 
$\xi:=\Pi(w)\in[a]$, that is the first letter of $W_0$. For $W\in\cW$ set 
$
w':=(W,W_0,W_1,\dots)
$. 
Observe that, denoting $W_0(w')$ the first word in $w'$, we have $W_0(w')=W$. Observe also that 
$$
W\in\cW(a,T)
\quad
\Leftrightarrow
\quad
M_{W,W_0(\xi)}=1
\quad
\Leftrightarrow
\quad
\sigma(w')=w.
$$
Therefore we have 
\begin{align*}
(HL_{(s,T)}f)(w)
&
=
L_{(s,T)}f\big(\Pi(w)\big)
=
\sum_{W\in\cW(a,T)}|D_{\Pi(w)}F_{W}|^sf\big(F_W(\Pi w)\big)
\\
&
=
\sum_{\sigma(w')=w}|D_{\Pi(w)}F_{W_0(w')}|^sf\big(F_{W_0(w')}(\Pi w)\big)
\\
&
=
\sum_{\sigma(w')=w}|D_{\Pi(w)}F_{W_0(w')}|^sf\big(\Pi(w')\big)
=
\cL_{(s,T)}(f\circ\Pi)(w)=(\cL_{(s,T)}Hf)(w),
\end{align*}
where the third to last equality holds because 
$
F_{W_0(w')}^{-1}(\Pi w')=\cF(\Pi w')=\Pi(\sigma w')=\Pi(w)
$ 
and therefore $F_{W_0(w')}(\Pi w)=\Pi(w')$. It follows that 
\begin{equation}
\label{EquationCommutativeDiagramTransferOperators}
H\circ L_{(s,T)}=\cL_{(s,T)}\circ H.
\end{equation}

See \S~\ref{AppendixQuasiCompactOperators} for terminology on quasi-compact operators. 
Let $\Lambda\subset\cB$ be the set defined in next 
\S~\ref{SectionInvariantSetLipschitzFunctions}.

\begin{theorem}
\label{TheoremTransferOperatorCircle}
There exists a real and positive $\widehat{\lambda}(s,T)>0$ and a strictly positive function 
$g_{(s,T)}\in\Lambda$ with 
\begin{equation}
\label{Equation(1)TheoremTransferOperatorCircle}
L_{(s,T)}(g_{(s,T)})=\widehat{\lambda}(s,T)\cdot g_{(s,T)}.
\end{equation}
The same holds on $\cB_T$ for the restriction of $g_{(s,T)}$ to $\EE_T$. Moreover $L_{(s,T)}$ is quasi-compact both on $\cB$ and on $\cB_T$ with
$$
\rho_{ess}(L_{(s,T)})<\theta\cdot\widehat{\lambda}(s,T)
$$
and $\widehat{\lambda}(s,T)$ is simple, with $|\lambda'|<\widehat{\lambda}(s,T)$ for any other eigenvalue. For $T<\infty$ we have 
$$
\widehat{\lambda}(s,T)=\lambda(s,T)
\quad
\textrm{ and }
\quad
g_{(s,T)}|_{\EE_T}=h_{(s,T)}\circ \Pi^{-1},
$$
moreover Equation~\eqref{EquationMeasureMu(s,T)} defines a mesure $\mu_{(s,T)}$ such that for any 
$f\in\cB_T$ we have 
\begin{equation}
\label{Equation(2)TheoremTransferOperatorCircle}
\lambda(s,T)^{-n}\cdot L_{(s,T)}^n(f)
\to
\bigg(\int fd\mu_{(s,T)}\bigg)\cdot g_{(s,T)}
\quad
\textrm{ uniformly as }
\quad
n\to\infty.
\end{equation}
For $T=\infty$ Equation~\eqref{EquationDefinitionMeasureParameterInfinity} below defines a mesure 
$\mu_{(s,\infty)}$ such that for any $f\in\cB$ we have 
\begin{equation}
\label{Equation(3)TheoremTransferOperatorCircle}
\widehat{\lambda}(s,\infty)^{-n}\cdot L_{(s,\infty)}^n(f)
\to
\bigg(\int fd\mu_{(s,\infty)}\bigg)\cdot g_{(s,\infty)}
\quad
\textrm{ uniformly as }
\quad
n\to\infty.
\end{equation}
Finally we have $\widehat{\lambda}(1,\infty)=1$.
\end{theorem}

\subsection{An invariant set of Lipschitz functions}
\label{SectionInvariantSetLipschitzFunctions}

We follow the ideas at pages 22-23 in \cite{ParryPollicott}. Let $0<\theta<1$ be the uniform constant in 
Lemma~\ref{LemmaUniformContractionBranches}. Denote here $\kappa>0$ the uniform constant in 
Lemma~\ref{LemmaLipschitzConstantLogDerivative}. Let $C>0$ be an uniform constant such that 
$$
\kappa+\theta\cdot C\leq C.
$$
Let $\Lambda$ be the set of real positive functions $f:\partial\DD\to\RR_+$ such that for any 
$a\in\cA$ we have 
\begin{equation}
\label{EquationDefinitionInvariantSet(Variations)}
f(\xi')\leq \exp\big(C\cdot|\xi'-\xi|\big)\cdot f(\xi)
\quad
\textrm{ for any }
\quad
\xi',\xi\in [a]. 
\end{equation}

Let $\Lambda_T$ be the set of functions $g:\EE_T\to\RR_+$ such that there exists some $f\in\Lambda$ 
with $g=f|_{\EE_T}$.

\begin{lemma}
\label{LemmaInvariantSetEquicontinuousEquibounded}
We have the inclusions $\Lambda\subset\cB$ and $\Lambda_T\subset\cB_T$. More precisely, any $f$ either in $\Lambda$ or in $\Lambda_T$ is bounded, and in terms of the constant $C$ introduced above it satisfies 
$$
\|f\|_\ast\leq \|f\|_\infty\cdot(1+Ce^{2C}).
$$
\end{lemma}

\begin{proof}
Fix $f\in\Lambda$. Since $|x-y|\leq\diameter(\DD)=2$ for any $x,y$, then 
Equation~\eqref{EquationDefinitionInvariantSet(Variations)} implies $\|f\|_\infty\leq+\infty$, that is $f$ is bounded (but a priori continuity is not yet proved). Moreover $(e^R-1)\leq Re^R$ for any $R\geq0$. Therefore for any $a\in\cA$ and any $x,y$ in 
$[a]$ we have 
\begin{align*}
f(y)-f(x)
&
\leq 
f(x)\cdot(e^{C|y-x|}-1)
\leq
\|f\|_\infty\cdot(e^{C|y-x|}-1)
\\
&
\leq
\|f\|_\infty\cdot e^{C|y-x|}\cdot C|y-x|
\leq 
\|f\|_\infty\cdot Ce^{2C}\cdot |y-x|.
\end{align*}
The last inequality holds reversing the role of $x$ and $y$, thus $f\in\lipschitz\big([a]\big)$, and a posteriori $f$ is also continuous on $[a]$. 
In particular we get $\lipschitz(f)\leq \|f\|_\infty\cdot Ce^{2C}$. 
The inequalities above hold restricting to $x,y$ in $\EE_T$, thus the statement follows also for any 
$f\in\Lambda_T$.
\end{proof}

\begin{lemma}
\label{LemmaInvarianceLambdaUnderTransferOperator}
For any $s,T$ and any $f\in\Lambda$ we have $L_{(s,T)}(f)\in\Lambda$. For any $f\in\Lambda_T$ we have also $L_{(s,T)}(f)\in\Lambda_T$.
\end{lemma}

\begin{proof}
It is enough to prove the Lemma for $f\in\Lambda$. It is clear that $L_{(s,T)}(f)$ is real and positive if $f$ is. Fix $a\in\cA$ and consider $y,x$ in $[a]$.
Corollary~\ref{CorollaryExponentiallyContractingDistance} gives 
$$
f(F_Wy)\leq 
\exp\big(C\cdot |f(F_Wy)-f(F_Wx)|\big)\cdot f(F_Wx)\leq
\exp\big(\theta C\cdot |y-x|\big)\cdot f(F_Wx).
$$
Moreover Lemma~\ref{LemmaLipschitzConstantLogDerivative} gives
\begin{align*}
|D_yF_W|^s
=
\exp\big(s\log(|D_yF_W|)\big)
&
\leq
\exp\big(s\log(|D_xF_W|)+\kappa\cdot|y-x|\big)
\\
&
=
\exp\big(\kappa\cdot|y-x|\big)\cdot|D_xF_W|^s.
\end{align*}
Recall that $\kappa+\theta\cdot C\leq C$. The statement follows because the estimates above give
\begin{align*}
L_{(s,T)}f(y)
&
=
\sum_{W\in\cW(a,T)}|D_yF_W|^sf(F_Wy)
\\
&
\leq
\sum_{W\in\cW(a,T)}
e^{\kappa\cdot|y-x|}\cdot|D_xF_W|^s
\cdot
e^{\theta C\cdot |y-x|}f(F_Wx)
\\
&
\leq
e^{(\kappa+\theta C)\cdot|y-x|}
\sum_{W\in\cW(a,T)}
\cdot|D_xF_W|^s\cdot f(F_Wx)
\leq
e^{C\cdot|y-x|}\cdot L_{(s,T)}f(x).
\end{align*}
\end{proof}

\subsection{Maximal eigenfunction for the operator on the circle}

Proposition~\ref{PropositionEigenfunctionTransferOperatorCircle} below follows the ideas at pages 22-24 in \cite{ParryPollicott}.

\begin{proposition}
\label{PropositionEigenfunctionTransferOperatorCircle}
There exists a real simple $\widehat{\lambda}(s,T)>0$ and a strictly positive $g_{(s,T)}\in\Lambda$ satisfying 
Equation~\eqref{Equation(1)TheoremTransferOperatorCircle}. For $T<\infty$ the same is true for the operator $L_{(s,T)}$ on $\cB_T$ taking the restriction of $g_{(s,T)}$ to $\EE_T$.
\end{proposition}

\begin{proof}
For simplicity set $L:=L_{(s,T)}$. By Lemma~\ref{LemmaQuadraticAsymptoticDerivative}, we have 
$\|L1\|_\infty\geq c$ for some uniform constant $c>0$. Therefore  
$
\|L(f+n^{-1})\|_\infty>c/n
$ 
for any $f$ continuous and positive and any integer $n\geq1$. It follows that for any integer 
$n\geq1$ we have a well defined (non linear) operator $M_n$ acting on the set of positive functions 
$f\in\cB$ by  
$$
M_n(f):=\frac{1}{\|L(f+n^{-1})\|_\infty}\cdot L(f+n^{-1}).
$$
Lemma~\ref{LemmaTransferOperatorsAreUniformlyBounded} gives a (uniform) upper bound for $\|L\|_\infty$ thus the operator $M_n$ is continuous in the topology of the norm $\|\cdot\|_\infty$, indeed for any 
$f,g$ in $\Lambda$ we have 
\begin{align*}
&
\|M_n(f)-M_n(g)\|_\infty\leq
\\
&
\left\|
\frac{L(f+n^{-1})}{\|L(f+n^{-1})\|_\infty}
-
\frac{L(g+n^{-1})}{\|L(f+n^{-1})\|_\infty}
\right\|
+
\left\|
\frac{L(g+n^{-1})}{\|L(f+n^{-1})\|_\infty}
-
\frac{L(g+n^{-1})}{\|L(g+n^{-1})\|_\infty}
\right\|\leq
\\
&
\frac{1}{\|L(f+n^{-1})\|_\infty}
\bigg(
\|L\|_\infty\cdot\|f-g\|_\infty
+
\big|
\|L(g+n^{-1})\|_\infty-\|L(f+n^{-1})\|_\infty
\big|
\bigg)\leq
\\
&
\frac{1}{\|L(f+n^{-1})\|_\infty}
\bigg(
2\|L\|_\infty\cdot\|f-g\|_\infty
\bigg)
\leq
\frac{n\cdot 2\|L\|_\infty}{c}\cdot\|f-g\|_\infty.
\end{align*}
Let 
$
\Lambda^{(1)}:=\{f\in\Lambda:f(\xi)\leq1\quad\forall\xi\in\partial\DD\}
$, 
where $\Lambda$ is the set of functions in 
\S~\ref{SectionInvariantSetLipschitzFunctions}. 
Lemma~\ref{LemmaInvarianceLambdaUnderTransferOperator} and the definition of $M_n$ imply that 
$
M_n(\Lambda^{(1)})\subset \Lambda^{(1)}
$. 
Lemma~\ref{LemmaInvariantSetEquicontinuousEquibounded}, via the Theorem of Ascoli-Arzel\`a, implies that $\Lambda^{(1)}$ is compact in the topology of $\|\cdot\|_\infty$ ($\Lambda^{(1)}$ is closed for uniform convergence). Moreover 
$\Lambda^{(1)}$ is obviously convex. Therefore the Schauder-Tychonov fixed point Theorem implies that there exists $g_n\in\Lambda^{(1)}$ such that $M_n(g_n)=g_n$, that is 
$$
L(g_n+n^{-1})=\|L(g_n+n^{-1})\|_\infty\cdot g_n.
$$
Compactness implies that, modulo subsequences, there exists $g\in\Lambda^{(1)}$ and $\lambda\geq 0$ 
such that $\|(g_n+n^{-1})-g\|_\infty\to0$ and $\|L(g_n+n^{-1})\|_\infty\to\lambda$ for $n\to+\infty$, that is 
$$
L(g)=\lambda\cdot g.
$$
Since $\|M_n(g_n)\|_\infty=1$ by construction, then $\|g_n\|_\infty=1$ for any $n$. Thus 
$\|g\|_\infty=1$. There exists $a_0\in\cA$ with $g(\xi')>0$ for any $\xi'\in[a_0]$, because 
$g\in\Lambda$.  Thus $g(F_W\cdot\xi)>0$ for any $W$ starting with $a_0$ and any $\xi\in\domain(W)$. For such $\xi$ we have $Lg(\xi)>0$, so that 
$
\lambda=\|Lg\|_\infty>0
$ 
strictly. Fix $a\in\cA$ and $\xi\in[a]$. For any $k\in\NN$ we have 
\begin{equation}
\label{EquationPropositionEigenfunctionTransferOperatorCircle}
g(\xi)=
\lambda^{-k}\cdot
\sum_{w_k\in\cW(k,a,T)}|D_\xi F_{w_k}|^{s}\cdot g(F_{w_k}\cdot\xi).
\end{equation}
If $g(\xi)=0$ for some $\xi$, then 
Equation~\eqref{EquationPropositionEigenfunctionTransferOperatorCircle} implies 
$g(F_{w_k}\cdot\xi)=0$ for any $w_k\in\cW(k,a,T)$. But points $F_{w_k}\cdot\xi$ become dense for  
$k\to\infty$, because the square $M^2_{W,W'}$ of the transition matrix in 
Equation~\eqref{EquationDefinitionTransitionMatrix} is positive by
Proposition~\ref{PropositionAperiodicityTransitionMatrix}. Therefore $g(\xi)>0$ for any $\xi$. Since $g\in\Lambda$, then 
Equation~\eqref{EquationDefinitionInvariantSet(Variations)} implies that there exists $m(s,T)>0$ such that 
$$
m(s,T)\leq g(\xi)\leq 1
\quad
\textrm{ for any }
\quad
\xi\in\partial\DD.
$$
In order to prove simplicity of $\lambda$, let $f\in\cB$ real valued such that $Lf=\lambda\cdot f$ (in general recall that $L$ is real and consider $\re(f)$ and $\im(f)$). Set 
$
t:=\inf_{\xi\in\partial\DD}f(\xi)/g(\xi)
$. 
Observe that $t\in\RR$. Then $h(\xi):=f(\xi)-tg(\xi)\geq0$ for any $\xi$. There exists $a\in\cA$,  
$
\xi_\infty\in\overline{[a]}
$ 
and a sequence $\xi_n\in[a]$ with $f(\xi_n)/g(\xi_n)\to t$ as $\xi_n\to\xi_\infty$. Then we have also $h(\xi_n)\to 0=h(\xi_\infty)$, because both $g$ and $f$ have Lipschitz extension to $\overline{[a]}$. Fix $k\in\NN$. 
Equation~\eqref{EquationPropositionEigenfunctionTransferOperatorCircle} implies 
\begin{align*}
0=
\lim_{n\to\infty}h(\xi_n)
&
=
\lambda^{-k}\cdot
\sum_{w_k\in\cW(k,a,T)}
\lim_{\xi_n\to\xi_\infty}|D_{\xi_n} F_{w_k}|^{s}\cdot 
\lim_{\xi_n\to\xi_\infty}h(F_{w_k}\xi_n)
\\
&
=
\lambda^{-k}\cdot
\sum_{w_k\in\cW(k,a,T)}|D_{\xi_\infty} F_{w_k}|^{s}h(F_{w_k}\xi_\infty).
\end{align*}
The last condition implies $h(F_{w_k}\xi_0)=0$ for any  $w_k\in\cW(k,a,T)$ and any $k\geq1$, thus $h(\xi)=0$ for any $\xi\in\partial\DD$ by continuity and we conclude $f=t\cdot g$. Observe that the sum above is finite for $0<T<\infty$, while for $T=\infty$ the sum is infinite and one applies the Dominated Convergence Theorem to the functions 
$$
\cH_n:\cW(k,a,T)\to\RR_+
\quad;\quad
\cH_n(w_k):=\lambda^{-k}\cdot|D_{\xi_n} F_{w_k}|^{s}\cdot h(F_{w_k}\xi_n).
$$
Setting  
$
\widehat{\cH}(w_k):=
\lambda^{-k}\cdot\big(\sup_{n\in\NN}|D_{\xi_n} F_{w_k}|^{s})\cdot\|h\|_\infty
$ 
we have
$
\cH_n(w_k)\leq \widehat{\cH}(w_k)
$ 
for any $n\in\NN$ and any $w_k\in\cW(k,a,T)$. Moreover, fix any $\xi\in[a]$ and observe that for any $w_k$ and any $n\in\NN$ Part (1) of 
Corollary~\ref{CorollaryLipschitzConstantDerivative} gives 
$
|D_{\xi_n}F_{w_k}|^s\leq C|D_\xi F_{w_k}|^s
$, 
so that the same holds for the supremum over $n\in\NN$. This implies that $\widehat{\cH}$ is summable, indeed we have
$$
\sum_{w_k\in\cW(k,a,T)}\widehat{\cH}(w_k)
\leq 
C\cdot\lambda^{-k}\cdot
\sum_{w_k\in\cW(k,a,T)}|D_\xi F_{w_k}|^s\cdot\|h\|_\infty
\leq
C\cdot\lambda^{-k}\cdot\|L^k(1)\|_\infty\cdot\|h\|_\infty.
$$
The arguments above hold for any point $\xi\in\partial\DD$, and thus in particular restricting to $\xi\in\EE_T$. Thus the statement for $L_{(s,T)}:\cB_T\to\cB_T$ holds too. 
\end{proof}

\subsection{Normalized operator and Lasota-Yorke inequality}


Fix any $s,T$. Let $g_{(s,T)}\in\Lambda$ and $\widehat{\lambda}(s,T)>0$ be as in 
Proposition~\ref{PropositionEigenfunctionTransferOperatorCircle}. In order to simplify the notation, set $$
g:=g_{(s,T)}
\quad\text{ ; }\quad
\widehat{\lambda}=\widehat{\lambda}(s,T)
\quad\text{ ; }\quad
L:=L_{(s,T)}.
$$
Define operators $\cN:\cB\to\cB$ and $\cN:\cB_T\to\cB_T$ by $\cN f(x):=g(s)\cdot f(x)$. It is easy to see that for any $f\in\cB$ and any $f\in\cB_T$ we have 
$$
\|\cN f\|_\infty\leq \|g\|_\infty\cdot\|f\|_\infty
\quad
\textrm{ and }
\quad
\lipschitz(\cN f)\leq 
\|g\|_\infty\cdot\lipschitz(f)
+
\|f\|_\infty\cdot\lipschitz(g).
$$
Thus $\|\cN\|_\ast\leq \|g\|_\ast$. Moreover $\cN$ is also invertible with 
$\|\cN^{-1}\|_\ast\leq \|g^{-1}\|_\ast$, where we observe that $g^{-1}\in\Lambda$ for $g\in\Lambda$ strictly positive. Define the \emph{normalized operators} 
$\widehat{L}:\cB\to\cB$ and $\widehat{L}:\cB_T\to\cB_T$ by 
$
\widehat{L}=\widehat{\lambda}^{-1}\cN^{-1}L\cN
$. 
For $a\in\cA$ and $W\in\cW(a)$ it is practical to introduce 
$$
N_W:\domain(W)\to\RR_+
\quad;\quad
N_W(x):=\widehat{\lambda}^{-1}\cdot \frac{|D_xF_W|^sg(F_Wx)}{g(x)},
$$
so that for any $f\in\cB$ (of in $\cB_T$), any $a\in\cA$ and any $x\in[a]$ we have 
$$
\widehat{L}f(x)=\sum_{W\in\cW(a,T)}N_W(x)f(F_Wx).
$$
Since $Lg=\widehat{\lambda}g$ then $\widehat{L}1=1$, and this corresponds to 
\begin{equation}
\label{EquationSumWeightsNormalizedTransferOperator}
\sum_{W\in\cW(a,T)}N_W(x)=1
\quad
\textrm{ for any }
\quad
a\in\cA, x\in[a].
\end{equation}
In particular $\|\widehat{L}\|_\infty=1$. We have 
$
\widehat{L}^kf(x)=\sum_{w_k\in\cW(k,a,T)}N_{w_k}(x)f(F_{w_k}x)
$, 
where for any $k\in\NN$ and any $w_k\in\cW(k,a,T)$ we consider function
$$
N_{w_k}(x):=\widehat{\lambda}^{-k}\cdot\frac{|D_xF_{w_k}|^s\cdot g(F_{w_k}x)}{g(x)}.
$$ 
Equation~\eqref{EquationSumWeightsNormalizedTransferOperator} implies  
$
\sum_{w_k\in\cW(k,a,T)}N_{w_k}(x)=1
$ 
for any $a\in\cA$ and any $x\in[a]$. 

\begin{proposition}
\label{PropositionLasotaYorke(AllParameters)}
For any $s,T$ there exists a constant $C=C(s,T)>0$ such that for any $k\in\NN$ we have 
$$
\lipschitz(\widehat{L}_{(s,T)}^kf)
\leq
\theta^k\cdot\lipschitz(f)+C\cdot \|f\|_\infty.
$$
\end{proposition}

\begin{proof}
Since $N_W(\cdot)$ is the product of 3 positive factors, write 
$\big|N_W(x)-N_W(y)\big|$ as telescopic sum of 6 terms. 
Denoting $C_1>0$ the uniform constant in Lemma~\ref{LemmaLipschitzConstantLogDerivative} we have  
\begin{align*}
|D_yF_W|^s-|D_xF_W|^s
\leq 
|D_xF_W|^s\cdot\big(e^{C_1|y-x|}-1\big)
&
\leq
|D_xF_W|^s\cdot e^{C_1|y-x|}C_1|y-x|
\\
&
\leq
\big(|D_xF_W|^s+|D_yF_W|^s\big)\cdot e^{C_1|y-x|}C_1|y-x|.
\end{align*}
The right hand side of the last inequality is symmetric in $x,y$, thus we have 
\begin{align*}
&
\bigg|\frac{|D_yF_W|^sg(F_Wx)}{g(x)}-\frac{|D_xF_W|^sg(F_Wx)}{g(x)}\bigg|
\leq
\frac{g(F_Wx)}{g(x)}
\cdot
\big(|D_xF_W|^s+|D_yF_W|^s\big)\cdot e^{C_1|y-x|}C_1|y-x|
\\
&
\leq
\widehat{\lambda}\cdot
\bigg(N_W(x)+
\frac{g(F_Wx)g(y)}{g(x)g(F_Wy)}\cdot N_W(y)
\bigg)
\cdot C_2|y-x|
\leq 
\widehat{\lambda}\cdot C_3\big(N_W(x)+ N_W(y)\big)\cdot|y-x|,
\end{align*}
where $C_2$ is an uniform constant and $C_3=C_3(s,T)>0$, depends on $s,T$ via $g$. Modulo increasing $C_3$, Lemma~\ref{LemmaInvariantSetEquicontinuousEquibounded} and 
Corollary~\ref{CorollaryExponentiallyContractingDistance} give
$$
\bigg|\frac{|D_yF_W|^sg(F_Wx)}{g(x)}-\frac{|D_yF_W|^sg(F_Wy)}{g(x)}\bigg|
\leq
\widehat{\lambda}\cdot C_3\big(N_W(x)+ N_W(y)\big)\cdot|y-x|.
$$
The same bound holds for 
$
|D_yF_W|^sg(F_Wy)\cdot|g^{-1}(x)-g^{-1}(y)|
$. 
Modulo increasing $C_3$ we get 
$$
\big|N_W(y)-N_W(x)\big|
\leq
C_3\big(N_W(x)+ N_W(y)\big)\cdot|y-x|.
$$
Equation~\eqref{EquationSumWeightsNormalizedTransferOperator} and 
Corollary~\ref{CorollaryExponentiallyContractingDistance} give 
\begin{align*}
\widehat{L}f(y)-\widehat{L}f(x)
&
\leq
\sum_{W\in\cW(a)}N_W(y)\big|f(F_Wy)-f(F_Wx)\big|+f(F_Wx)\big|N_W(y)-N_W(x)\big|
\\
&
\leq
\theta\cdot\lipschitz(f)+2C_3\cdot\|f\|_\infty.
\end{align*}
The statement follows for $k=1$. An easy inductive argument gives the statement for any $k$, where 
$
C=2C_3\sum_{k=0}^\infty\theta^k
$.
See the proof of Proposition~2.1 in \cite{ParryPollicott}.
\end{proof}

\subsection{Specific construction for $T=\infty$}

Fos $s>1/2$ set $g:=g_{(s,\infty)}$ and 
$
\widehat{\lambda}=\widehat{\lambda}(s,\infty)
$, 
where  $g_{(s,\infty)}$ and $\widehat{\lambda}(s,\infty)$ are as in 
Theorem~\ref{TheoremTransferOperatorCircle}. Set also $L:=L_{(s,\infty)}$.

\medskip

\begin{lemma}
\label{LemmaInvariantMeasureNormalizedOperatorInfinityCase}
There exists a Borel probability measure $\nu$ on $\partial\DD$ such that for any $f\in\cB$ we have 
$$
\int \widehat{L}fd\nu=\int fd\nu.
$$
\end{lemma}

\begin{proof}
Consider the disjoint union $\EE_\infty:=\bigsqcup_{a\in\cA}\overline{[a]}$, that is the set of pairs $(a,x)$ with $a\in\cA$ and $x\in\overline{[a]}$. A set $E\subset\EE_\infty$ is open if and only if 
$\iota_a^{-1}(E)$ is open in 
$\overline{[a]}$ for any $a\in\cA$, where we consider the maps $\iota_a(x):=(a,x)$, and with this topology $\EE_\infty$ is a compact set homeomorphic to $[0,1]\times\cA$. A function $f:\EE_\infty\to\CC$ is continuous if and only if 
$
f\circ\iota_a:\overline{[a]}\to\CC
$ 
is continuous for any $a\in\cA$. Let $\cC$ be the Banach space of continuos functions 
$f:\EE_\infty\to\CC$ with the sup norm $\|\cdot\|_\infty$. For any $a\in\cA$ and any $W\in\cW(a)$ both $F_W$ and $|DF_W|$ admit a continuous extension to $\overline{[a]}$. The same is true for the functions  $g$ and $g^{-1}$. Therefore the operator $\widehat{L}$ acts on $\cC$. The condition $\widehat{L}1=1$ implies that the dual operator 
$
\widehat{L}^\ast:\cC^\ast\to\cC^\ast
$ 
preserves the convex weakly compact subset $\cM(\EE_\infty)$ of Borel probability measures over 
$\EE_\infty$. Thus there exists 
$m\in\cM(\EE_\infty)$ such that $\widehat{L}^\ast m=m$, according to the Schauder-Tychonov fixed point theorem. Given the inclusion maps $p_a:\overline{[a]}\to\partial\DD$, let 
$
p:\EE_\infty\to\partial\DD
$ 
be the unique continuous map such that $p\circ\iota_a=p_a$ for any $a\in\cA$. Then consider the Borel probability measure 
$\nu:=p^\ast(m)$ over $\partial\DD$. For any $f\in\cB$ we have 
$$
\int fd\nu=
\int f\circ pdm=
\int f\circ pd(\widehat{L}^\ast m)=
\int \widehat{L}(f\circ p)dm=
\int (\widehat{L}f)\circ pdm=
\int \widehat{L}fd\nu.
$$
\end{proof}

\begin{lemma}
\label{LemmaUniformConvergenceNormalizedOperatorInfinityCase}
For any $f\in\cB$ we have uniform convergence
$$
\widehat{L}_{(s,\infty)}^kf\to\int fd\nu_{(s,\infty)}.
$$
\end{lemma}

\begin{proof}
Since $\widehat{L}$ is real we can assume that $f$ is real valued (in the general case just consider $\re(f)$ and $\im(f)$). Set $f^{(k)}:=\widehat{L}^kf$. Recalling that $\|\widehat{L}\|_\infty=1$, 
Proposition~\ref{PropositionLasotaYorke(AllParameters)} and the Theorem of 
Ascoli-Arzel\`a imply that there exists $f_\infty\in\cB$ with $f^{(k_n)}\to f_\infty$ uniformly along some subsequence $k_n$ with $n\to\infty$.
Equation~\eqref{EquationSumWeightsNormalizedTransferOperator} implies 
$
\sup f^{(k+1)}\leq \sup f^{(k)}
$ 
for any $k\in\NN$. Moreover $f^{(k)}(x)\geq \inf f$ for any $k\in\NN$ and any $x$, where 
$\inf f>-\infty$ strictly. Fix $N\in\NN$. The last two conditions imply 
$
\sup\widehat{L}^Nf_\infty=\sup f_\infty
$. 
Indeed assume there exists $\epsilon>0$ with 
$
\sup\widehat{L}^Nf_\infty\leq\sup f_\infty-3\epsilon
$. 
For any $n$ big enough we have 
$$
\|f^{(k_n)}-f_\infty\|_\infty<\epsilon
\quad
\textrm{ ; }
\quad
\|\widehat{L}^Nf^{(k_n)}-\widehat{L}^Nf_\infty\|_\infty<\epsilon
\quad
\textrm{ ; }
\quad
|\sup f^{(k_n)}-\sup f^{(k_n+N)}|<\epsilon.
$$ 
Therefore we get an absurd because
$$
\sup \widehat{L}^Nf_\infty>
\sup \widehat{L}^Nf^{(k_n)}-\epsilon=
\sup f^{(k_n+N)}-\epsilon>
\sup f^{(k_n)}-2\epsilon>
\sup f_\infty-3\epsilon.
$$
Consider $x_0$ and $x_N$ with  
$
f_\infty(x_0)=\sup f_\infty=\sup\widehat{L}^Nf_\infty=\widehat{L}^Nf_\infty(x_N)
$, 
where $N$ is the integer fixed above. Such points exist modulo replacing $f_\infty$ and $\widehat{L}^Nf_\infty$ by $f_\infty\circ p$ and $\widehat{L}^Nf_\infty\circ p$, where $p:\EE_\infty\to\partial\DD$ is the continuous map in the proof of 
Lemma~\ref{LemmaInvariantMeasureNormalizedOperatorInfinityCase}. We have 
$$
\widehat{L}^Nf_\infty(x_N)=
\sum_{w_k\in\cW(k,a)}N_{w_k}(x_N)f_\infty(F_{w_k}x_N)=f_\infty(x_0).
$$
Since 
$
\sum_{w_k\in\cW(k,a)}N_{w_k}(x_N)=1
$ 
and any term in the sum is positive, then for any $w_n\in\cW(k,a)$ we have 
$
f_\infty(F_{w_k}x_N)=f_\infty(x_0)
$. 
Thus continuity implies 
$
f_\infty(x)=f_\infty(x_0)
$ 
for any $x$. Finally, setting $\nu^{(k)}:=(\widehat{L}^k)^\ast\nu$ we get 
$$
f_\infty(x_0)=
\int f_\infty d\nu=
\lim_{n\to\infty}\int f^{(k_n)} d\nu=
\lim_{n\to\infty}\int f d\nu^{(k_n)}=
\int f d\nu.
$$
The argument can be repeated for any subsequence of $f^{(k)}$, extracting a sub-subsequence converging to the integral. Thus the Lemma follows. 
\end{proof}

Let $\nu_{(s,\infty)}:=\nu$ be the probability measure as in 
Lemma~\ref{LemmaInvariantMeasureNormalizedOperatorInfinityCase} and 
Lemma~\ref{LemmaUniformConvergenceNormalizedOperatorInfinityCase}. Let $\mu_{(s,\infty)}$ be the probability measure on $\partial\DD$ defined by  
\begin{equation}
\label{EquationDefinitionMeasureParameterInfinity}
\mu_{(s,\infty)}(f):=\int f(x)g_{(s,\infty)}^{-1}(x)d\nu_{(s,\infty)}(x)
\quad
\textrm{ for any }
\quad
f\in C(\partial\DD).
\end{equation}

\subsection{End of the proof of Theorem~\ref{TheoremTransferOperatorCircle}}
\label{SectionEndProofTheoremTransferOperatorCircle}

Equation~\eqref{Equation(1)TheoremTransferOperatorCircle} and simplicity of 
$
\widehat{\lambda}(s,T)
$ 
follow from 
Proposition~\ref{PropositionEigenfunctionTransferOperatorCircle}. 
Proposition~\ref{PropositionLasotaYorke(AllParameters)} and Corollary~1 in \cite{Hennion} imply that the normalized operator satisfies   
$
\rho_{ess}(\widehat{L}_{(s,T)})\leq\theta
$. 
Since $\cN$ is an isomorphism of Banach spaces, then the spectrum of 
$L_{(s,T)}$ is equal to the spectrum of $\widehat{L}_{(s,T)}$ multiplied by 
$
\widehat{\lambda}(s,T)
$, 
and we have 
$$
\rho_{ess}(L_{(s,T)})\leq \theta\cdot\widehat{\lambda}(s,T).
$$ 
Since $0<\theta<1$ and $\widehat{\lambda}(s,T)$ is eigenvalue, then 
$L_{(s,T)}:\cB\to\cB$ is quasi-compact. The same holds for 
$L_{(s,T)}:\cB_T\to\cB_T$. It remains to prove maximality of $\widehat{\lambda}(s,T)$.

Consider $T<\infty$. Point (4) of 
Theorem~\ref{TheoremRuellePerronFrobenius} and 
Equation~\eqref{EquationCommutativeDiagramTransferOperators} imply that for any $f\in C(\EE_T)$ and for 
$n\to\infty$ we have uniform convergence 
\begin{align*}
\lambda(s,T)^{-n}\cdot L_{(s,T)}^n f
=
\lambda(s,T)^{-n}\cdot
H^{-1 }\cL_{(s,T)}^n (Hf)
&
\to
\bigg(\int (Hf)d\widehat{m}_{(s,T)}\bigg)H^{-1}h_{(s,T)}
\\
&
=
\bigg(\int f d\mu_{(s,T)}\bigg)\big(h_{(s,T)}\circ\Pi^{-1}).
\end{align*}
Applying the last result to the function $f=g_{(s,T)}\in\Lambda_T$ in   
Proposition~\ref{PropositionEigenfunctionTransferOperatorCircle} we obtain 
$$
\frac{\widehat{\lambda}(s,T)^n}{\lambda(s,T)^n}\cdot g_{(s,T)}
\to
\bigg(\int g_{(s,T)} d\mu_{(s,T)}\bigg)\big(h_{(s,T)}\circ\Pi^{-1})
\quad
\textrm{ as }
\quad
n\to\infty.
$$
Since both $g_{(s,T)}$ and $h_{(s,T)}\circ\Pi^{-1}$ are positive, then 
$
\widehat{\lambda}(s,T)=\lambda(s,T)
$. 
The normalization 
$
\int g_{(s,T)} d\mu_{(s,T)}=1
$ 
gives also 
$
g_{(s,T)}=h_{(s,T)}\circ\Pi^{-1}
$. 
Equation~\eqref{Equation(2)TheoremTransferOperatorCircle} follows, and it implies maximality of 
$
\lambda(s,T)=\widehat{\lambda}(s,T)
$ 
for $L_{(s,T)}:\cB_T\to\cB_T$. For $L_{(s,T)}:\cB\to\cB$, consider $\lambda'\in\CC$ and $f\in\cB$ with $L_{(s,T)}f=\lambda'\cdot f$. Then set 
$f_T:=f|_{\EE_T}$ and $h_T:=f_T\circ\Pi=H(f_T)$. We have 
$
L_{(s,T)}f_T=\lambda'\cdot f_T
$ 
because $\EE_T$ is invariant under the maps $F_W$ defining $L_{(s,T)}$. 
Moreover $h_T\in\lipschitz(\Sigma,\theta)$, because  
$\Pi:\Sigma\to\EE_T$ is Lipschitz by Lemma~\ref{LemmaCodingLipschitz}. 
Equation~\eqref{EquationCommutativeDiagramTransferOperators} implies  
$$
\lambda'\cdot h_T=
H(\lambda'\cdot f_T)=
HL_{(s_T,T)}f_T=
\cL_{(s_T,T)}Hf_T=
\cL_{(s_T,T)}h_T.
$$
Resuming, any eigenvalue $\lambda'$ of $L_{(s,T)}:\cB\to\cB$ is also eigenvalue of $\cL_{(s,T)}$. Thus 
$\lambda(s,T)$ is maximal for $L_{(s,T)}:\cB\to\cB$ by Point (2) of 
Theorem~\ref{TheoremRuellePerronFrobenius}.

Consider $T=\infty$. 
Equation~\eqref{Equation(2)TheoremTransferOperatorCircle} and maximality of 
$\widehat{\lambda}(s,\infty)$ follow because for any $f\in\cB$ 
Lemma~\ref{LemmaUniformConvergenceNormalizedOperatorInfinityCase} implies 
\begin{align*}
\widehat{\lambda}(s,\infty)^{-k}\cdot L_{(s,\infty)}f
=
\cN\widehat{L}^k\cN^{-1}f
\to
\bigg(\int f(x)g_{(s,\infty)}^{-1}(x)d\nu_{(s,\infty)}(x)\bigg)g_{(s,\infty)}.
\end{align*}

It only remains to prove 
$
\widehat{\lambda}(1,\infty)=1  
$. 
For simplicity set $\lambda:=\widehat{\lambda}(1,\infty)$ and $g:=g_{(1,\infty)}$. Fix any $a\in\cA$ and set also $\cW(k):=\cW(k,a,T=\infty)$. The mean value Theorem and 
Corollary~\ref{CorollaryLipschitzConstantDerivative} imply that there exists a constant $C>0$ such that for any $w_k\in\cW(k,a,\infty)$ the Lebesgue measure $\big|[w_k]\big|$ of $[w_k]$ satisfies 
$$
C^{-1}\cdot\big|[w_k]\big|
\leq 
|D_\xi F_{w_k}|
\leq
C\cdot \big|[w_k]\big|.
$$
We have  
$
\sum_{w_k\in\cW(k,a,\infty)}\big|[w_k]\big|=1
$ 
because $\{[w_k]:w_k\in\cW(k,a,\infty)\}$ is a partition of $\partial\DD$ for any $k\geq2$. 
Equation~\eqref{EquationPropositionEigenfunctionTransferOperatorCircle} implies that for any $k\geq2$ and any $\xi\in[a]$ we have 
\begin{align*}
g(\xi)
&
=
\lambda^{-k}\cdot
\sum_{w_k\in\cW(k,a,\infty)}|D_\xi F_{w_k}|\cdot g(F_{w_k}\cdot\xi)
\geq
\lambda^{-k}\cdot
m\cdot\sum_{w_k\in\cW(k,a,\infty)}|D_\xi F_{w_k}|
\\
&
\geq
\frac{m}{\lambda^{k}\cdot C}\cdot\sum_{w_k\in\cW(k,a,\infty)}\big|[w_k]\big|
=
\frac{m}{\lambda^{k}\cdot C},
\end{align*}
where $m>0$ is such that $m\leq g(\xi)\leq 1$ for any $\xi$. We have also
\begin{align*}
g(\xi)
&
=
\lambda^{-k}\cdot
\sum_{w_k\in\cW(k,a,\infty)}|D_\xi F_{w_k}|\cdot g(F_{w_k}\cdot\xi)
\leq
\lambda^{-k}\cdot
\sum_{w_k\in\cW(k,a,\infty)}|D_\xi F_{w_k}|
\\
&
\leq
\frac{C}{\lambda^{k}}\cdot\sum_{w_k\in\cW(k,a,\infty)}\big|[w_k]\big|
=
\frac{C}{\lambda^{k}}.
\end{align*}
Resuming, we have $m\leq g(\xi)\leq1$ and 
$
m\cdot C^{-1}\leq \lambda^k g(\xi)\leq C
$ 
for any $k\geq 2$. Hence we must have $\lambda=1$. 
Theorem~\ref{TheoremTransferOperatorCircle} is proved. \qed

\section{Perturbative estimate of maximal eigenvalue: proof of Theorem~\ref{TheoremDimensionShiftSpace}}
\label{SectionPerturbativeEstimateMaximalEigenvalue}

In this section we consider only the transfer operators $L_{(s,T)}$ on the Banach space $\cB$, for parameters $s,T$ with $9/10\leq s\leq 1$ and $0<T\leq+\infty$. For such $T$, let $s_T$ be the solution of Equation~\eqref{EquationBowenFormulaDimension}. As in 
\S~\ref{SectionEstimatesContractionDistortion} all constants are uniform, unless explicitly stated. 

\subsection{Expansion of transfer operator in parameter $s$}

Observe that for any $D,s$ with $0<D<1$ and $9/10<s\leq 1$ and any $h\in\RR$ such that 
$9/10<s+h\leq 1$ we have 
\begin{equation}
\label{EquationOrderTwoError(First)}
D^{s+h}=
D^s+\ln(D)\cdot\int_s^{s+h}D^tdt=
D^s+h\cdot D^s\ln(D)+\ln(D)\cdot\int_s^{s+h}(D^t-D^s)dt,
\end{equation}
where we recall that $d/dt(D^t)=(\ln D)\cdot D^t$. The mean value Theorem gives 
\begin{equation}
\label{EquationOrderTwoError(Second)}
\big|D^{s+h}-\big(D^s+h\cdot D^s\ln(D)\big)\big|
\leq 
|h|^2\cdot|\ln(D)|^2\cdot D^{4/5}.
\end{equation}

For $0<T\leq+\infty$ consider the operator $A_T:=\cB\to\cB$ defined by
\begin{equation}
\label{EquationDefinitionDerivativeOperator(A)}
A_Tf(x):=\sum_{W\in\cW(a,T)}
\big(\ln|D_xF_W|\big)|D_xF_W|^{s_T}f(F_Wx)
\quad
\textrm{ if }
\quad
x\in[a].
\end{equation}

For any $W\in\cW$ set 
$$
\|\ln DF_W\|_\infty:=
\sup_{\xi\in\domain(W)}\big|\ln|D_\xi F_W| \big|.
$$

Since $s_T\to1$ as $T\to\infty$, let $T_0$ be such that $9/10<s_T\leq1$ for $T_0\leq T\leq\infty$.

\begin{lemma}
\label{LemmaNormDerivativeOperator(A)}
There exists an uniform constant $C>0$ such that for $T_0\leq T\leq\infty$ we have
$$
\|A_T\|_\ast\leq C.
$$
\end{lemma}

\begin{proof}
Fix $T$ with $T_0\leq T\leq\infty$. 
Corollary~\ref{CorollaryUniformlySummableSeries} gives directly an uniform upper bound for 
$\|A_T\|_\infty$. Fix $W\in\cW$ and consider the function 
$$
\Phi_W:\domain(W)\to\RR_+
\quad;\quad
\Phi_W(x):=\big(\ln|D_xF_W|\big)|D_xF_W|^{s_T}.
$$
Denote $C_1>0$ and $C_2>0$ be the constants in 
Lemma~\ref{LemmaLipschitzConstantLogDerivative} and 
Corollary~\ref{CorollaryLipschitzConstantDerivative} respectively. For any $x,y$ in $\domain(W)$ we have 
\begin{align*}
&
|\Phi_W(y)-\Phi_W(x)|
\\
&
\leq
\big|\ln|D_yF_W|\big|\big(|D_yF_W|^{s_T}-|D_xF_W|^{s_T}\big)
+
|D_xF_W|^{s_T}\big(\ln|D_yF_W|-\ln|D_xF_W|\big)
\\
&
\leq
\big(
C_2\cdot\|\ln DF_W\|_\infty\cdot\|DF_W\|_\infty^{s_T}
+
C_1\cdot\|DF_W\|_\infty^{s_T}
\big)\cdot|y-x|.
\end{align*}
Therefore for any $f\in\cB$, any $a\in\cA$ and any $x,y$ in $[a]$ we have 
\begin{align*}
&
\big|A_Tf(y)-A_Tf(x)\big|\leq
\\
&
\sum_{W\in\cW(a,T)}
\big|
\Phi_W(y)\big(f(F_Wy)-f(F_Wx)\big)
+
f(F_Wx)\cdot\big(\Phi_W(y)-\Phi_W(x)\big)
\big|\leq
\\
&
\sum_{W\in\cW(a,T)}
\big|\Phi_W(y)\big|\cdot\theta\cdot\lipschitz(f)\cdot|y-x|
+
\|f\|_\infty\cdot\big|\Phi_W(y)-\Phi_W(x)\big|
\leq
\\
&
C_3\cdot\big(\theta\cdot\lipschitz(f)+\|f\|_\infty\big)\cdot|y-x|,
\end{align*}
where the second inequality follows from the estimate above and from 
Corollary~\ref{CorollaryExponentiallyContractingDistance}, and where the last inequality follows from 
Corollary~\ref{CorollaryUniformlySummableSeries}, in terms of some uniform constant $C_3$. The uniform bound for $\|A_T\|_\ast$ follows combining the estimates above.
\end{proof}

\begin{lemma}
\label{LemmaTaylorEspansionOperator(A)}
There exists an uniform constant $C>0$ such that for any $T$ with $T_0\leq T\leq \infty$ and any $s$ with 
$9/10\leq s\leq 1$ we have
$$
\|L_{(s,T)}-L_{(s_T,T)}\|_\ast\leq C\cdot|s-s_T|
\quad
\textrm{ and }
\quad
\|L_{(s,T)}-L_{(s_T,T)}-(s-s_T)A_T\|_\ast\leq C\cdot|s-s_T|^2.
$$
\end{lemma}

\begin{proof}
The first bound obviously follows from the second and from 
Lemma~\ref{LemmaNormDerivativeOperator(A)}, thus we only prove that 
$
\|R\|_\ast\leq C\cdot|s-s_T|^2
$, 
where for convenience of notation we set
$$
R:=L_{(s,T)}-L_{(s_T,T)}-(s-s_T)A_T.
$$
Fix $W\in\cW$. In the notation of the proof of Lemma~\ref{LemmaNormDerivativeOperator(A)}, consider the function  
$$
\Psi_W:\domain(W)\to\RR_+
\quad;\quad
\Psi_W(x):=|D_xF_W|^s-|D_xF_W|^{s_T}-(s-s_T)\Phi_W(x).
$$
According to Equation~\eqref{EquationOrderTwoError(Second)} for any $x$ as above we have
$$
|\Psi_W(x)|\leq |s-s_T|^2\cdot\big|\ln|D_xF_W|\big|^2\cdot|D_xF_W|^{4/5}.
$$
Moreover the same argument which gives Part (1) of 
Corollary~\ref{CorollaryUniformlySummableSeries} implies that there exists an uniform constant $C_1>0$ such that for any $a\in\cA$ and any $x\in[a]$ we have
$$
\sum_{W\in\cW(a,T)}\big|\ln|D_xF_W|\big|^2\cdot|D_xF_W|^{4/5}\leq C_1.
$$
It follows immediately that 
$
\|R\|_\infty\leq C_1\cdot|s-s_T|^2
$. 
Moreover for $W\in\cW$ and $x\in\domain(W)$ Equation~\eqref{EquationOrderTwoError(First)} gives
$$
\Psi_W(x)=\big(\ln|D_xF_W|\big)\cdot\int_{s_T}^s B_W(x,t)dt,
$$
where we set $B_W(x,t):=|D_xF_W|^t-|D_xF_W|^{s_T}$. For $s_T<t<s$ or $s<t<s_T$, the first equality in 
Equation~\eqref{EquationOrderTwoError(First)} gives
$$
\big|B_W(x,t)\big|
\leq 
|t-s_T|\cdot \big(\ln|D_xF_W|\big)|D_xF_W|^{4/5}
\leq
|s-s_T|\cdot \|\ln DF_W\|_\infty\cdot\|DF_W\|^{4/5}.
$$
Moreover for $x,y$ in $\domain(W)$ again  the first equality in 
Equation~\eqref{EquationOrderTwoError(First)} gives
\begin{align*}
&
\big|B_W(y,t)-B_W(x,t)\big|=
\\
&
\bigg|
\big(\ln|D_yF_W|\big)\int_{s_T}^t|D_yF_W|^rdr
-
\big(\ln|D_xF_W|\big)\int_{s_T}^t|D_xF_W|^rdr
\bigg|
\leq
\\
&
\bigg|\ln|D_yF_W|\bigg|\int_{s_T}^t\big||D_yF_W|^r-|D_xF_W|^r\big|dr
+
\bigg|\ln|D_yF_W|-\ln|D_xF_W|\bigg|\int_{s_T}^t|D_xF_W|^rdr
\leq
\\
&
\|\ln DF_W\|_\infty \cdot|t-s_T|\cdot C_2\cdot\|DF_W\|^{4/5}\cdot|y-x|
+
C_2\cdot|y-x|\cdot|t-s_T|\cdot \|DF_W\|^{4/5}
\leq
\\
&
C_2\cdot \|DF_W\|^{4/5}\cdot (1+\|\ln DF_W\|_\infty)\cdot|s-s_T|\cdot|y-x|,
\end{align*}
where $C_2$ is some uniform constant and 
$
\big||D_yF_W|^r-|D_xF_W|^r\big|
$ 
is bounded by Part (2) of Corollary~\ref{CorollaryLipschitzConstantDerivative}, while   
$
\big|\ln|D_yF_W|-\ln|D_xF_W|\big|
$ 
is bounded by Lemma~\ref{LemmaLipschitzConstantLogDerivative}. The bounds obtained above for 
$
\big|B_W(y,t)-B_W(x,t)\big|
$ 
and for $\big|B_W(x,t)\big|$ give 
\begin{align*}
&
\big|\Psi_W(y)-\Psi_W(x)\big|=
\\
&
\bigg|
\big(\ln|D_yF_W|\big)\int_{s_T}^sB(y,t)dt
-
\big(\ln|D_xF_W|\big)\int_{s_T}^sB(x,t)dt
\bigg|
=
\\
&
\bigg|\ln|D_yF_W|\bigg|\int_{s_T}^s\big|B(y,t)-B(x,t)\big|dt
+
\bigg|\ln|D_yF_W|-\ln|D_xF_W|\bigg|\int_{s_T}^sB(x,t)dt
\leq
\\
&
\|\ln DF_W\|_\infty\cdot C_2\cdot \|DF_W\|^{4/5}\cdot (1+\|\ln DF_W\|_\infty)\cdot|s-s_T|^2\cdot|y-x|
+
\\
&
C_3\cdot|y-x|\cdot|s-s_T|^2\cdot \|\ln DF_W\|_\infty\cdot\|DF_W\|^{4/5}
\leq
\\
&
C_4\cdot(1+\|\ln DF_W\|_\infty)^2\cdot\|DF_W\|^{4/5}\cdot|y-x|\cdot|s-s_T|^2,
\end{align*}
where $C_4$ is some uniform constant, and where $C_3$ is the constant in 
Lemma~\ref{LemmaLipschitzConstantLogDerivative}. The same argument which gives  
Corollary~\ref{CorollaryUniformlySummableSeries} implies that there is some uniform constant $C_5$ such that 
$$
\sum_{W\in\cW(a,T)}\big|\Psi_W(y)-\Psi_W(x)\big|
\leq
C_5\cdot |y-x|\cdot|s-s_T|^2.
$$
The estimates above and the bound $\|R\|_\infty\leq C_1\cdot|s-s_T|^2$ from the first part of the proof imply the bound on $\|R\|_\ast$, indeed $|Rf(y)-Rf(x)|$ is bounded by 
\begin{align*}
&
\sum_{W\in\cW(a,T)}
|\Psi_W(x)|\cdot\big|f(F_Wy)-f(F_Wx)\big|
+
|f(F_Wx)|\cdot|\Psi_W(y)-\Psi_W(x)|
\\
&
\sum_{W\in\cW(a,T)}
|\Psi_W(x)|\cdot\lipschitz(f)\cdot\theta\cdot|y-x|
+
\|f\|_\infty\cdot|\Psi_W(y)-\Psi_W(x)|.
\end{align*}
\end{proof}

\subsection{Expansion of transfer operator in parameter $T$}

Fix $a\in\cA$ and $0<T<\infty$ and set $\cV(a,T):=\cW(a)\setminus\cW(a,T)$, as in 
\S~\ref{SectionPreliminaryUniformEstimates}. Let $\Delta_{(s,T)}:\cB\to\cB$ be the operator defined by  
\begin{equation}
\label{EquationDefinitionDerivativeOperator(Delta)}
\Delta_{(s,T)}f(x):=
\sum_{W\in\cV(a,T)}|D_xF_W|^{s}f(F_Wx)
\quad
\textrm{ if }
\quad
x\in[a].
\end{equation}

\begin{lemma}
\label{LemmaNormDerivativeOperator(Delta)}
There exists an uniform constant $C>0$ such that for $s,T$ as above we have 
$$
\|\Delta_{(s,T)}\|_\ast\leq C\cdot T^{-(2s-1)}.
$$
\end{lemma}

\begin{proof}
Corollary~\ref{CorollaryUniformlySummableSeries} bounds $\|\Delta_{(s,T)}\|_\infty$. 
Moreover for any $a\in\cA$, any $x,y\in[a]$, any $f\in\cB$, the difference 
$
|\Delta_{(s,T)}f(y)-\Delta_{(s,T)}f(x)|
$ 
is bounded by 
$$
\sum_{W\in\cV(a,T)}
|D_yF_W|^s|f(F_Wy)-f(F_Wx)|+|f(F_Wx)|\cdot\big||D_yF_W|^s-|D_xF_W|^s\big|.
$$
Hence Corollary~\ref{CorollaryLipschitzConstantDerivative} and 
Corollary~\ref{CorollaryUniformlySummableSeries} and give the bound on 
$\|\Delta_{(s,T)}\|_\ast$.
\end{proof}

\begin{lemma}
\label{LemmaTaylorEspansionOperator(Delta)}
There exists an uniform constant $C>0$ such that for any $T>0$ we have  
$$
\|\Delta_{(1,T)}-\Delta_{(s,T)}\|_\ast
\leq 
\frac{C\cdot|s-1|}{T^{3/5}}.
$$
\end{lemma}

\begin{proof}
Recall the fist part of Equation~\eqref{EquationOrderTwoError(First)}, and for any $W\in\cW$ and $x\in\domain(W)$ set 
$$
B_W(x):=
|D_xF_W|-|D_xF_W|^s=
\big(\ln|D_xF_W|\big)\cdot\int_s^1|D_xF_W|^tdt.
$$ 
We have 
$
|B_W(x)|\leq\|\ln DF_W\|_\infty\cdot\|DF_W\|^s\cdot|s-1|
$. 
Moreover $2s>1+3/5+1/5$ because $s>9/10$. Observing that $\ln|W|\leq |W|^{1/5}$ for $|W|$ big enough, 
Corollary~\ref{CorollaryUniformlySummableSeries} implies that there is some uniform constant $C>0$ such that
\begin{align*}
\|\Delta_{(1,T)}-\Delta_{(s,T)}\|_\infty
&
\leq
|s-1|\cdot\sum_{W\in\cV(a,T)}\|\ln DF_W\|_\infty\cdot\|DF_W\|^s
\leq
\frac{C\cdot|s-1|}{T^{3/5}}.
\end{align*}
Moreover for any $W\in\cW$ and any $x,y$ in $\domain(W)$ we have 
\begin{align*}
&
\big|B_W(y)-B_W(x)\big|
=
\bigg|
\big(\ln|D_yF_W|\big)\int_s^1|D_yF_W|^tdt
-
\big(\ln|D_xF_W|\big)\int_s^1|D_xF_W|^tdt
\bigg|
\leq
\\
&
\big|\ln|D_yF_W|\big|\int_s^1\big||D_yF_W|^t-|D_xF_W|^t\big|dt
+
\big|\ln|D_yF_W|-\ln|D_xF_W|\big|\int_s^1|D_xF_W|^tdt\leq
\\
&
C_1\cdot\bigg(
\|\ln DF_W\|_\infty\cdot\|DF_W\|^s+\|DF_W\|^s
\bigg)\cdot|s-1|\cdot|x-y|=
\\
&
C_1\cdot\|DF_W\|^s\cdot(1+\|\ln DF_W\|_\infty)\cdot|s-1|\cdot|x-y|,
\end{align*}
where $C_1$ is some uniform constant and 
$
\big||D_yF_W|^t-|D_xF_W|^t\big|
$ 
is bounded by Corollary~\ref{CorollaryLipschitzConstantDerivative}, while   
$
\big|\ln|D_yF_W|-\ln|D_xF_W|\big|
$ 
is bounded by Lemma~\ref{LemmaLipschitzConstantLogDerivative}. Therefore, for any $f\in\cB$, the bound on 
$
\lipschitz\big((\Delta_{(1,T)}-\Delta_{(s,T)})f\big)
$ 
follows as usual considering any $a\in\cA$ and any $x,y$ in $[a]$ and bounding the difference 
\begin{align*}
&
(\Delta_{(1,T)}-\Delta_{(s,T)})f(y)-
(\Delta_{(1,T)}-\Delta_{(s,T)})f(x)=
\\
&
\sum_{W\in\cV(a,T)}
B_W(y)\big(f(F_Wy)-f(F_Wx)\big)
+
f(F_Wx)\cdot\big(B_W(y)-B_W(x)\big).
\end{align*}
\end{proof}

\subsection{Spectral decomposition for parameters close to $(1,\infty)$}
\label{SectionSpectralPropertiesGeneralParameters}

For basic properties of spectral projectors see 
\S~\ref{AppendixDefinitionSpectralProjectors}. 
Recall from Theorem~\ref{TheoremTransferOperatorCircle} that $L_{(s,T)}$ is quasi-compact with isolated and simple maximal eigenvalue $\widehat{\lambda}(s,T)$, corresponding to the positive eigenfunction $g_{(s,T)}$. In particular 
$
\widehat{\lambda}(1,\infty)=1
$. 
Hence there exists $\epsilon_0>0$ such that the spectrum of $L_{(1,\infty)}$ can be decomposed as
$$
\spectrum(L_{(1,\infty)})=\Sigma\cup\{1\}
\quad
\textrm{ with }
\quad
\Sigma\subset B(0,1-2\epsilon_0).
$$
In terms of such $\epsilon_0$ let $\gamma$ and $\gamma_\ast$ be the loops in $\CC$ defined for $0\leq t<2\pi$ by 
\begin{equation}
\label{EquationLoopResolventSet}
\gamma(t):=1+\epsilon_0\cdot e^{it}
\quad
\textrm{ and }
\quad
\gamma_\ast(t):=(1-\epsilon_0)\cdot e^{it}.
\end{equation}
For parameters $(s,T)$ consider the expression 
\begin{equation}
\label{EquationSpectralProjectorMaximalEigenvalue}
P_{(s,T)}:=
\frac{-1}{2\pi i}\int_\gamma (L_{(s,T)}-\xi\cdot\id)^{-1}d\xi.
\end{equation}
According to \S~\ref{AppendixDefinitionSpectralProjectors}, the operator $P_{(s,T)}$ is defined if and only if $\gamma$ is contained in the resolvent set $\cR(L_{(s,T)})$. In this case $P_{(s,T)}$ is a projection commuting with $L_{(s,T)}$, that is $P_{(s,T)}^2=P_{(s,T)}$ and 
$
L_{(s,T)}P_{(s,T)}=P_{(s,T)}L_{(s,T)}
$, 
and we have an $L_{(s,T)}$-invariant spectral decomposition: 
$$
\cB=N_{(s,T)}\oplus V_{(s,T)}
\quad
\textrm{ where }
\quad
V_{(s,T)}:=\ker(P_{(s,T)}-\id)
\quad
\textrm{ and }
\quad
N_{(s,T)}:=\ker(P_{(s,T)}).
$$
According to Theorem~\ref{TheoremTransferOperatorCircle}, for those $(s,T)$ such that $P_{(s,T)}$ is defined, we have
\begin{equation}
\label{EquationExplicitFormProjection}
P_{(s,T)}(f)=\bigg(\int f(\xi)d\mu_{(s,T)}(\xi)\bigg)\cdot g_{(s,T)}
\quad\text{ for any }\quad f\in\cB.
\end{equation}
The discussion above implies that $P_{(1,\infty)}$ is defined, and the corresponding spectral decomposition 
$
\cB=N_{(1,\infty)}\oplus V_{(1,\infty)}
$ 
satisfies 
$$
V_{(1,\infty)}=\ker(L_{(1,\infty)}-\id)
\quad
\textrm{ and }
\quad
\rho(L_{(1,\infty)}|_{N_{(1,\infty)}})\leq 1-2\epsilon_0.
$$

\begin{lemma}
\label{LemmaSpectralPropertyAllParameters}
There exist $T_0>0$ and $0<s_0<1$ for any 
$
(s,T)\in[s_0,1]\times[T_0,+\infty]
$ 
the following holds. 
\begin{enumerate}
\item
We have $|\widehat{\lambda}(s,T)-1|<\epsilon_0$.
\item
The spectral radius of restriction $L_{(s,T)}:N_{(s,T)}\to N_{(s,T)}$ satisfies 
$$
\rho\big(L_{(s,T)}|_{N_{(s,T)}}\big)\leq1-\epsilon_0.
$$
\item
The projection given by Equation~\eqref{EquationSpectralProjectorMaximalEigenvalue} satisfies
$$
\|P_{(s,T)}-P_{(1,\infty)}\|_\ast
=
O\big(\|L_{(s,T)}-L_{(1,\infty)}\|_\ast\big).
$$
\end{enumerate}
\end{lemma}

\begin{proof}
Since $L_{(s,T)}=L_{(s,\infty)}-\Delta_{(s,T)}$ for any $(s,T)$, then 
Lemma~\ref{LemmaNormDerivativeOperator(Delta)} and 
Lemma~\ref{LemmaTaylorEspansionOperator(A)} imply that 
$
\|L_{(s,T)}-L_{(1,\infty)}\|_\ast
$ 
can be made arbitrarily small for $(s,T)$ close enough to $(1,\infty)$. 
The discussion in \S~\ref{AppendixStabilitySpectralDecomposition} implies that for any such $(s,T)$ both the loops $\gamma$ and $\gamma_\ast$ in Equation~\eqref{EquationLoopResolventSet} are contained in the resolvent set $\cR(L_{(s,T)})$. Thus $P_{(s,T)}$ is defined and Point (3) follows. 
By \S~\ref{AppendixDefinitionSpectralProjectors} we know that 
$
\spectrum\big(L_{(s,T)}|_{N_{(s,T)}}\big)
$ 
is contained in the interior of $\gamma_\ast$. For the same reason $\widehat{\lambda}(s,T)$ is the part of 
$
\spectrum(L_{(s,T)})
$ 
included in the interior of $\gamma$. Finally the spectrum of $L_{(s,T)}$ does not have other components, because $\cB=N_{(s,T)}\oplus V_{(s,T)}$. The Corollary is proved. 
\end{proof}

\subsection{Expansion of maximal eigenvalue in parameters $s$ and $T$}
\label{ExpansionMaximalEigenvalueInParameters}

We follow the approach in \S~7 in \cite{Hensley}. For $(s,T)$ close to $(1,\infty)$ call $\lambda(s,T)$ the eigenvalue in Lemma~\ref{LemmaSpectralPropertyAllParameters}, also for the value $T=\infty$. Consider the spectral decomposition 
$$
\cB=N_{(1,\infty)}\oplus V_{(1,\infty)},
$$ 
where 
$
N_{(1,\infty)}:=\ker(P_{(1,\infty)})
$ 
and 
$
V_{(1,\infty)}:=\ker(\id-P_{(1,\infty)})
$. 
Let $\delta\in\RR$ and $u_{(1,\infty)}\in N_{(1,\infty)}$ such that we have
$$
A_\infty(g_{(1,\infty)})=\delta\cdot g_{(1,\infty)}+u_{(1,\infty)}.
$$
We have $\delta<0$ strictly, indeed $|D_xF_W|\leq\theta<1$ for any $W\in\cW$ and any $x\in\domain(W)$, and the expression of $A_\infty$ in Equation~\eqref{EquationDefinitionDerivativeOperator(A)} gives 
\begin{align*}
\delta
&
=
\int A_\infty g_{(1,\infty)}d\mu_{(1,\infty)}
=
\sum_{a\in\cA,\,W\in\cW(a)}
\int_{[a]}
\big(\ln|D_xF_W|\big)\cdot|D_xF_W|\cdot g_{(1,\infty)}(F_Wx)d\mu_{(1,\infty)}(x).
\end{align*}

\begin{proposition}
\label{PropositionExpansionEigenvalueParameter(s)}
We have
$$
\lambda(s,\infty)=1+\delta(s-1)+O(|s-1|^2).
$$
\end{proposition}

\begin{proof}
According to Lemma~\ref{LemmaSpectralPropertyAllParameters}, the restriction 
$
L_{(1,\infty)}|_{N_{(1,\infty)}}:N_{(1,\infty)}\to N_{(1,\infty)}
$ 
has spectral radius 
$
\rho(L_{(1,\infty)}|_{N_{(1,\infty)}})\leq1-\epsilon_0
$, 
thus its resolvent $\cR(L_{(1,\infty)}|_{N_{(1,\infty)}})$ contains $z=1$. Consider 
$
u_{(1,\infty)}\in N_{(1,\infty)}
$ 
in the decomposition of $A_\infty(g_{(1,\infty)})$ and set 
$$
v_{(1,\infty)}:=
\big((\id-L_{(1,\infty)})|_{N_{(1,\infty)}}\big)^{-1}(u_{(1,\infty)})
=
-R(L_{(1,\infty)}|_{N_{(1,\infty)}},1)(u_{(1,\infty)}).
$$
Since $L_{(1,\infty)}(g_{(1,\infty)})=g_{(1,\infty)}$, the definition of $v_{(1,\infty)}$ and 
Lemma~\ref{LemmaTaylorEspansionOperator(A)} give  
\begin{align*}
&
L_{(s,\infty)}\big(g_{(1,\infty)}+(s-1)v_{(1,\infty)}\big)
=
\\
&
\big(L_{(1,\infty)}+(s-1)A_\infty+O(|s-1|^2)\big)
\big(g_{(1,\infty)}+(s-1)v_{(1,\infty)}\big)
=
\\
&
\big(1+\delta(s-1)\big)g_{(1,\infty)}
+(s-1)\big(u_{(1,\infty)}+
L_{(1,\infty)}(v_{(1,\infty)})\big)+
O(|s-1|^2)=
\\
&
\big(1+\delta(s-1)\big)\big(g_{(1,\infty)}+(s-1)v_{(1,\infty)}\big)+O(|s-1|^2).
\end{align*}
On the other hand we have
$
L_{(s,\infty)}\circ P_{(s,\infty)}=P_{(s,\infty)}\circ L_{(s,\infty)}
$, 
therefore 
\begin{align*}
&
\lambda(s,\infty)P_{(s,\infty)}\big(g_{(1,\infty)}+(s-1)v_{(1,\infty)}\big)=
\\
&
L_{(s,\infty)}P_{(s,\infty)}\big(g_{(1,\infty)}+(s-1)v_{(1,\infty)}\big)=
\\
&
P_{(s,\infty)}L_{(s,\infty)}\big(g_{(1,\infty)}+(s-1)v_{(1,\infty)}\big)=
\\
&
P_{(s,\infty)}
\bigg(
\big(1+\delta(s-1)\big)\big(g_{(1,\infty)}+(s-1)v_{(1,\infty)}\big)+O(|s-1|^2)
\bigg)=
\\
&
\big(1+\delta(s-1)\big)P_{(s,\infty)}\big(g_{(1,\infty)}+(s-1)v_{(1,\infty)}\big)+O(|s-1|^2),
\end{align*}
where the third equality follows from the expression for 
$
L_{(s,\infty)}\big(g_{(1,\infty)}+(s-1)v_{(1,\infty)}\big)
$ 
obtained above, the first two equalities are standard algebra, and the forth inequality follows because $P_{(s,\infty)}$ is close to $P_{(1,\infty)}$, and thus it has norm close to 
$\|P_{(1,\infty)}\|_\ast$. Therefore
$$
\bigg(
\lambda(s,\infty)-\big(1+\delta(s-1)\big)
\bigg)\cdot P_{(s,\infty)}\big(g_{(1,\infty)}+(s-1)v_{(1,\infty)}\big)
=
O(|s-1|^2).
$$
The statement follows because 
$
\|g_{(1,\infty)}+(s-1)v_{(1,\infty)}\|_\ast
\geq
\|g_{(1,\infty)}\|_\ast-|s-1|\cdot\|v_{(1,\infty)}\|_\ast
$ 
and because $P_{(s,\infty)}$ is close to $P_{(1,\infty)}$, which acts on $g_{(1,\infty)}$ as the identity.
\end{proof}

For $0<T\leq\infty$ consider the spectral decomposition  
$$
\cB=N_{(s_T,\infty)}\oplus V_{(s_T,\infty)},
$$ 
where 
$
N_{(s_T,\infty)}:=\ker(P_{(s_T,\infty)})
$ 
and 
$
V_{(s_T,\infty)}:=\ker(\id-P_{(s_T,\infty)})
$. 
Consider $\beta_T\in\RR$ and $u_{(s_T,\infty)}\in N_{(s_T,\infty)}$ such that, in the direct sum above, we have 
$$
\Delta_{(s_T,T)}g_{(s_T,\infty)}
=
\beta_T\cdot g_{(s_T,\infty)}+u_{(s_T,\infty)}.
$$

\begin{proposition}
\label{PropositionExpansionEigenvalueParameter(T)}
We have
$$
\lambda(s_T,T)=\lambda(s_T,\infty)-\beta_T
+
O(T^{-(2s_T-1)})\cdot\big(O(|s_T-1|)+O(T^{-(2s_T-1)})\big).
$$
\end{proposition}

\begin{proof}
According to Lemma~\ref{LemmaSpectralPropertyAllParameters}, the restriction of 
$L_{(s_T,\infty)}$ to $N_{(s_T,\infty)}$ has spectral radius 
$
\rho(L_{(s_T,\infty)}|_{N_{(s_T,\infty)}})\leq1-\epsilon_0
$, 
thus its resolvent set 
$
\cR(L_{(s_T,\infty)}|_{N_{(s_T,\infty)}})
$ 
contains $z=1$. Set $Q_{(s,T)}:=\id-P_{(s,T)}$. For any $u\in N_{(s_T,\infty)}$ we have 
\begin{align*}
&
(L_{(s_T,\infty)}-\id)u=
(L_{(s_T,\infty)}-\id)Q_{(s_T,\infty)}u=
\\
&
\big((L_{(1,\infty)}-\id)+(L_{(s_T,\infty)}-L_{(1,\infty)})\big)
\big(Q_{(1,\infty)}+(Q_{(s_T,\infty)}-Q_{(1,\infty)})\big)u=
\\
&
(L_{(1,\infty)}-\id)Q_{(1,\infty)}u+O(\|L_{(s_T,\infty)}-L_{(1,\infty)}\|_\ast)u,
\end{align*}
where the last inequality follows from 
Equation~\eqref{EquationDistanceSpectralProjectors}. 
Consider the invertible bounded operator $G:\cB\to\cB$ with 
$
GP_{(1,\infty)}G^{-1}=P_{(s_T,\infty)}
$, 
introduced in \S~\ref{AppendixStabilitySpectralDecomposition}. We have $u=G(v)$ for some 
$v\in N_{(1,\infty)}$, and 
Equation~\eqref{EquationDistanceBaseChanges} gives 
\begin{align*}
&
u=G(v)=v+(G-\id)v=
\\
&
v+O(\|P_{(1,\infty)}\|_\ast\cdot\|P_{(s_T,\infty)}-P_{(1,\infty)}\|_\ast)v
=
v+O(\|L_{(s_T,\infty)}-L_{(1,\infty)}\|_\ast)v.
\end{align*}
We get 
$$
\|(L_{(s_T,\infty)}-\id)u\|_\ast
\geq
\|(L_{(1,\infty)}-\id)v\|_\ast-\epsilon\|v\|_\ast
\geq 
c_0\|v\|_\ast\geq c_0'\|u\|_\ast
$$
where $\epsilon>0$ above can me made arbitrarily small for 
$
\|L_{(s_T,\infty)}-L_{(1,\infty)}\|_\ast
$ 
small enough, and $c_0>0$ and $c_0'>0$ are uniform constant not depending on $T$. In conclusion we have the uniform bound (i.e. not depending on $T$):
$$
\|R(L_{(s_T,\infty)}|_{N_{(s_T,\infty)}},1)\|_\ast\leq (c_0')^{-1}.
$$
Consider $u_{(s_T,\infty)}$ in the decomposition of 
$
\Delta_{(s_T,T)}g_{(s_T,\infty)}
$ 
and set 
$$
v_{(s_T,\infty)}:=
\big((\id-L_{(s_T,\infty)})|_{N_{(s_T,\infty)}}\big)^{-1}(u_{(s_T,\infty)})
=
-R(L_{(s_T,\infty)}|_{N_{(s_T,\infty)}},1)(u_{(s_T,\infty)}).
$$
It is practical to introduce the quantity
$$
\textrm{Err}(T):=
\big(\lambda(s_T,\infty)-1\big)\cdot v_{(s_T,\infty)}
-\beta_T\cdot v_{(s_T,\infty)}+\Delta_{(s_T,T)}(v_{(s_T,\infty)}).
$$
Recall that 
$
\Delta_{(s_T,T)}=O(T^{-(2s_T-1)})
$, 
according to Lemma~\ref{LemmaNormDerivativeOperator(Delta)}. Since $P_{(s_T,\infty)}$ is close to 
$P_{(1,\infty)}$, then we have 
$$
\beta_Tg_{(s_T,\infty)}=
P_{(s_T,\infty)}\Delta_{(s_T,T)}g_{(s_T,T)}=O(T^{-(2s_T-1)}).
$$
Hence    
$
\beta_T=O(T^{-(2s_T-1)})
$, 
because $g_{(s_T,\infty)}$ is close to $g_{(1,\infty)}$ for $T$ big enough. The uniform bound on 
$
\|R(L_{(s_T,\infty)}|_{N_{(s_T,\infty)}},1)\|_\ast
$ 
implies  
$
v_{(s_T,\infty)}=O(T^{-(2s_T-1)})
$. 
Proposition~\ref{PropositionExpansionEigenvalueParameter(s)} gives 
$$
\textrm{Err}(T)
=
O(T^{-(2s_T-1)})\cdot\big(O(|s_T-1|)+O(T^{-(2s_T-1)})\big).
$$
Since $L_{(s_T,T)}(g_{(s_T,T)})=g_{(s_T,T)}$, then the definition of $v_{(s_T,\infty)}$ gives 
\begin{align*}
&
L_{(s_T,T)}(g_{(s_T,+\infty)}-v_{(s_T,\infty)})
=
\\
&
(L_{(s_T,\infty)}-\Delta_{(s_T,T)})(g_{(s_T,\infty)}-v_{(s_T,\infty)})=
\\
&
\big(\lambda(s_T,\infty)-\beta_T\big)\cdot g_{(s_T,\infty)}
-u_{(s_T,\infty)}
-L_{(s_T,\infty)}(v_{(s_T,\infty)})
+\Delta_{(s_T,T)}(v_{(s_T,\infty)})=
\\
&
\big(\lambda(s_T,\infty)-\beta_T\big)\cdot \big(g_{(s_T,\infty)}- v_{(s_T,\infty)}\big)
+
\textrm{Err}(T).
\end{align*}

On the other hand we have
$
L_{(s_T,T)}\circ P_{(s_T,T)}=P_{(s_T,T)}\circ L_{(s_T,T)}
$, 
therefore standard algebra gives
\begin{align*}
&
\lambda(s_T,T)P_{(s_T,T)}\big(g_{(s_T,+\infty)}-v_{(s_T,\infty)}\big)=
\\
&
L_{(s_T,T)}P_{(s_T,T)}\big(g_{(s_T,+\infty)}-v_{(s_T,\infty)}\big)=
\\
&
P_{(s_T,T)}L_{(s_T,T)}\big(g_{(s_T,+\infty)}-v_{(s_T,\infty)}\big)=
\\
&
P_{(s_T,T)}
\bigg(
\big(\lambda(s_T,\infty)-\beta_T\big)\cdot \big(g_{(s_T,\infty)}- v_{(s_T,T)}\big)
+
\textrm{Err}(T)
\bigg)=
\\
&
\big(\lambda(s_T,\infty)-\beta_T\big)
\cdot
P_{(s_T,T)}\big(g_{(s_T,\infty)}- v_{(s_T,T)}\big)
+
P_{(s_T,T)}\big(\textrm{Err}(T)\big).
\end{align*}
Since $P_{(S_T,T)}$ is close to $P_{(1,\infty)}$, then its norm admits an uniform bound which does not depend on $T$. It follows that  
$
P_{(s_T,T)}\big(\textrm{Err}(T)\big)=O\big(\textrm{Err}(T)\big)
$. 
Therefore
\begin{align*}
&
\bigg(
\lambda(s_T,T)-\big(\lambda(s_T,\infty)-\beta_T\big)
\bigg)
\cdot 
P_{(s_T,T)}\big(g_{(s_T,\infty)}- v_{(s_T,\infty)}\big)=
\\
&
O\big(\textrm{Err}(T)\big)
=
O(T^{-(2s_T-1)})\cdot\big(O(|s_T-1|)+O(T^{-(2s_T-1)})\big).
\end{align*}
The statement follows because 
$
P_{(s_T,T)}(v_{(s_T,\infty)})=O(T^{-(2s_T-1)})
$, 
while on the other hand $P_{(s_T,T)}(g_{(s_T,\infty)})$ has norm bounded from below, because it is close to $g_{(1,\infty)}$.
\end{proof}

\subsection{Asymptotic of $\beta_T$}
\label{SectionAsymptoticBeta}

\begin{lemma}
\label{LemmaExistenceLimit}
There exists the limit
$$
\beta:=
\lim_{T\to+\infty} T\cdot\int\Delta_{(1,T)}g_{(1,\infty)}d\mu_{(1,\infty)}.
$$
Moreover $\beta>0$ strictly.
\end{lemma}

\begin{proof}
For simplicity, write $g:=g_{(1,\infty)}$ and $\mu:=\mu_{(1,\infty)}$. 
Fix $a\in\cA$. According to Lemma~\ref{LemmaCombinatorialPropertiesCuspidal}, for any $W\in\cV(a,T)$ there exist $k\in\NN$, a word $P\in\cW$ such that $F_P$ is parabolic fixing a vertex of $\Omega_\DD$ and a cuspidal word $V\in\cW(a)$ which is a final factor of $P$, such that 
\begin{equation}
\label{EquationSectionLemmaExistenceLimit}
W=P^{(k)}\ast V
\quad
\textrm{ where }
\quad
P^{(k)}:=\underbrace{P\ast\dots\ast P}_{k\textrm{ times }}.
\end{equation}
The set $\cP(a)$ of words $P$ as above is finite (with cardinality bounded by twice the number of vertices of $\Omega_\DD$). For $P\in\cP(a)$ let $\xi_P$ be the vertex of $\Omega_\DD$ which is the parabolic fixed point of $F_P$. Let $\cI(P)$ be the set of initial factors of $P$, which is a finite set with cardinality the number of letters of $P=(a_1,\dots,a_p)$ (the empty word counts as a factor). 

Fix $P\in\cP(a)$. For $W=P^{(k)}\ast V$ the arc $[W]_\EE$ shrinks to $F_V^{-1}(\xi_P)$ as $k\to\infty$. Moreover we have $F_W(\xi)\in[W]_\EE$ for any $\xi\in[a]\subset\domain(W)$. Since $g\in\cB$, then there exist the limits of $g(\xi)$ for $\xi\to\xi_P$ either from the left or from the right. Denote these two limits as 
$
g\big(F_V^{-1}(\xi_P),\varepsilon(P)\big)
$, 
in terms of the symbol 
$
\varepsilon(P)\in\{L,R\}
$. 
The discussion above implies that the limit below exists:
$$
\lim_{k\to\infty}g\big(F_{P^{(k)}\ast V}\cdot\xi\big)
=
g\big(F_V^{-1}(\xi_P),\varepsilon(P)\big).
$$
Let $\mu>0$ be such that $F_P$ is conjugated to $z\mapsto 2\mu$. 
Equation~\eqref{EquationCoefficientsAlphaBetaParabolic} implies that for any $k\in\NN$ and any 
$\xi\in\domain(W)$ the $k$-th power has derivative 
$$
|D_\xi F_P^k|
=
\frac{1}{k^2\cdot|\mu|^2}
\cdot
\frac{1}{\big|\xi-\xi_P\big(1+i(k\mu)^{-1}\big)\big|^2}.
$$

For $N\in\NN^\ast$ we have the identity 
$
1=N\cdot\sum_{k=N}^\infty \big(k(k+1)\big)^{-1}
$. 
This implies that if $a_k$ is a sequence in $\CC$ with $a_k\to\lambda$ for $k\to\infty$, then 
$
N\cdot\sum_{k=N}^\infty a_k\cdot k^{-2}\to\lambda
$ 
as $N\to\infty$. The discussions above, and this last remark, imply that for any $P\in\cP(a)$, any $V\in\cI(P)$ and any 
$
\xi\in\domain(W=P^{(k)}\ast V)
$ 
we have 
$$
\lim_{N\to\infty}
N\cdot
\sum_{k=N}^\infty|D_{F_V(\xi)} F_P^k|\cdot g(F_{P^{(k)}V}\xi)
=
\frac
{g\big(F_V^{-1}(\xi_P),\varepsilon(P)\big)}
{|\mu|^2\cdot\big|F_V(\xi)-\xi_P\big|^2}.
$$

The limit in the statement is equal to  
$
\sum_{a\in\cA}\big(\lim_{T\to\infty}\int_{[a]}f_{a,T}d\mu\big)
$, 
where for any fixed $a\in\cA$ we consider the function 
$f_{a,T}:[a]\to\RR_+$ defined by
$$
f_{a,T}(\xi):=
T\cdot\sum_{W\in\cV(a,T)}|D_\xi F_W|\cdot g(F_W\xi).
$$

Fix $a\in\cA$. For any $W\in\cV(a,T)$ consider its decomposition $W=P^{(k)}\ast V$ as in 
Equation~\eqref{EquationSectionLemmaExistenceLimit}. According to their definition in 
Equation~\eqref{EquationDefinitionGeometricLenght}, the geometric lengths $|P|$ and $|W|$ satisfy 
$
\big||W|-k\cdot|P|\big|\leq C
$, 
where $C>0$ is some uniform constant. Let $N(P,T)$ be the integer part of $T\cdot|P|^{-1}$. the discussion above and Corollary~\ref{CorollaryUniformlySummableSeries} give 
\begin{align*}
f_{a,T}(\xi)
&
=
\sum_{P\in\cP(a)}\sum_{V\in\cI(P)}|P|\cdot N(P,T)\cdot\sum_{k=N(P,T)}^\infty
|D_{F_V(\xi)} F_P^k|\cdot |D_\xi F_V|\cdot g(F_{P^{(k)}V}\xi)+O(T^{-1})
\\
&
=
\sum_{P\in\cP(a)}\sum_{V\in\cI(P)}
\frac
{|P|\cdot |D_\xi F_V|\cdot g\big(F_V^{-1}(\xi_P),\varepsilon(P)\big)}
{|\mu(P)|^2\cdot\big|F_V(\xi)-\xi_P\big|^2}+o(1)
\quad
\text{where }o(1)\to0\textrm{ as }T\to\infty.
\end{align*}
Any $V\in\cI(P)$ is a final factor of $P$, thus\footnote{In order to see the implication, set $W=(b_0,\dots,b_m)$ and assume $\varepsilon(W)=R$ without loss of generality. We have 
$
F_V^{-1}(\xi_P)=[\widehat{b_m}]\cap[b^\ast]
$ 
for $b^\ast\in\cA$ with 
$
o(b^\ast)=o(\widehat{b_m})-1
$. 
On the other hand $a\not=\widehat{b_m},b^\ast$.} $F_V^{-1}(\xi_P)$ is not an endpoint of $[a]$. Therefore 
$\big|F_V(\xi)-\xi_P\big|$ is bounded from below for any $\xi\in[a]$. 
The existence of the limit $\beta$ follows because the expression above is bounded and thus integrable. Moreover 
$\beta>0$ strictly because $g$ is bounded from below by a positive constant
(see proof of Proposition~\ref{PropositionEigenfunctionTransferOperatorCircle}).
\end{proof}

\begin{lemma}
\label{LemmaExpansionBeta}
We have 
$$
\beta_T=\beta\cdot T^{-1}+O\big(|s_T-1|\cdot T^{-3/5}\big)+o(T^{-1}).
$$
\end{lemma}

\begin{proof}
Recall that 
$
g_{(s_T,\infty)}-g_{(1,\infty)}
=
O\big(\|P_{(s_T,\infty)}-P_{(1,\infty)}\|_\ast\big)
$ 
and that 
$$
P_{(s_T,\infty)}-P_{(1,\infty)}
=
O\big(\|L_{(s_T,\infty)}-L_{(1,\infty)}\|_\ast\big)
=
O(|s_T-1|).
$$
Recall from the proof of 
Lemma~\ref{PropositionExpansionEigenvalueParameter(T)} that $\beta_T$ and $\Delta_{(s_T,T)}$ are both $O(T^{-(2s_T-1)})$. For $T$ big enough we have $s_T>9/10$ and thus 
$$
\beta_T\cdot g_{(s_T,\infty)}=\beta_T\cdot g_{(1,\infty)}+O(|s_T-1|\cdot T^{-4/5}).
$$ 
On the other hand 
$
\beta_T\cdot g_{(s_T,\infty)}=P_{(s_T,\infty)}\Delta_{(s_T,T)}(g_{(s_T,\infty)})
$, 
moreover the three terms 
$$
\Delta_{(s_T,T)}(g_{(s_T,\infty)}-g_{(1,\infty)})
\quad;\quad
(\Delta_{(1,T)}-\Delta_{(s_T,T)})
\quad;\quad
(P_{(s_T,\infty)}-P_{(1,\infty)})\big)\Delta_{(1,T)}
$$
are 
$
O\big(|s_T-1|\cdot T^{-3/5}\big)
$, 
where in particular
$\|\Delta_{(1,T)}-\Delta_{(s_T,T)}\|_\ast$  is bounded by 
Lemma~\ref{LemmaTaylorEspansionOperator(Delta)} and $\Delta_{(1,T)}$ by 
Lemma~\ref{LemmaNormDerivativeOperator(Delta)}. 
The expression for $P_{(1,\infty)}$ in 
Equation~\eqref{EquationExplicitFormProjection} and 
Lemma~\ref{LemmaExistenceLimit} imply 
\begin{align*}
\beta_T\cdot g_{(s_T,\infty)}
&
=
P_{(s_T,\infty)}\Delta_{(s_T,T)}(g_{(s_T,\infty)})
=
P_{(s_T,\infty)}\Delta_{(s_T,T)}(g_{(1,\infty)})+O\big(|s_T-1|\cdot T^{-3/5}\big)
\\
&
=
P_{(s_T,\infty)}\Delta_{(1,T)}(g_{(1,\infty)})+O\big(|s_T-1|\cdot T^{-3/5}\big)
\\
&
=
P_{(1,\infty)}\Delta_{(1,T)}(g_{(1,\infty)})+O\big(|s_T-1|\cdot T^{-3/5}\big)
\\
&
=
(\beta\cdot T^{-1})\cdot g_{(1,\infty)}+o(T^{-1})+O\big(|s_T-1|\cdot T^{-3/5}\big).
\end{align*}
The Lemma follows comparing the two developments for $\beta_T\cdot g_{(s_T,\infty)}$. 
\end{proof}

\subsection{End of the proof of Theorem~\ref{TheoremDimensionShiftSpace}}

Independent arguments give a fist row upper bound on 
$
\big|1-\dim_H\big(\bad(\Gamma,\epsilon)\big)\big|
$.

\begin{lemma}
\label{LemmaRowErrorDimension}
There exists some $\lambda>0$ with 
$
\big|\dim_H(\EE_T)-1\big|=O(T^{-\lambda})
$.
\end{lemma}

\begin{proof}
Consider $\Gamma$ as a subgroup of $\sltwor$. Fix an integer $N\geq2$.  
For $k\in\NN$ let $\cP_k$ be the set of parabolic fixed points of $\Gamma$ with 
$N^{k-1}<D(\zeta)\leq N^k$. 
For any $\zeta\in\cP_k$ let $I(\zeta)$ be the interval centered at $\zeta$ with length 
$|I(\zeta)|=2/(N^{k+2}D(\zeta))$. For $k\geq1$ let $\cJ_k$ be the set of intervals $J\subset\RR$ with 
$|J|=N^{-2k}$ and endpoints in $N^{-2k}\cdot\ZZ$. Let $C_0$ be the collection (non-empty for $N$ big enough) of $J\in\cJ_1$ with 
$
J\cap\cup_{\zeta\in\cP_0}I(\zeta)=\emptyset
$. 
For $k\geq1$ let $C_k$ be the collection of intervals $J\in\cJ_{k+1}$ with $J\subset J'$ for some 
$J'\in C_{k-1}$ and $J\cap I(\zeta)=\emptyset$ for any $\zeta\in\cP_k$. Let $F_k$ be the union of intervals of $C_k$. Equation~\eqref{EquationDistanceDisjointHoroballs} implies that for any $J\in\cJ_k$ we have 
$
\big|J\cap\cup_{\zeta\in\cP_k}I(\zeta)\big|\leq \text{Const}\cdot N^{-1}|J|
$. 
Since $|J'|\leq |I(\zeta)|$ for any $\zeta\in\cP_k$ and $J'\in\cJ_{k+1}$, then
$
|J\cap F_k|\geq |J|\cdot(1-2\text{Const}\cdot N^{-1})
$. 
We have
$
\cap_{k\in\NN}F_k\subset\bad(\Gamma,\epsilon)
$ 
with $\epsilon=N^{-3}$, moreover Proposition~2.2 in~\cite{McMullenDimension} gives 
$$
\big|1-\dim_H\big(\cap_{k\in\NN}F_k\big)\big|
\leq 
\log(1-2\text{Const}\cdot N^{-1})/\log N^2
=
O((N\log N)^{-1}).
$$ 
The statement follows from Lemma~\ref{LemmaBadAndBoundedContinuedFraction} for any $\lambda<1/3$. 
\end{proof}

Combining 
Proposition~\ref{PropositionExpansionEigenvalueParameter(s)} and 
Proposition~\ref{PropositionExpansionEigenvalueParameter(T)} we get
\begin{align*}
&
1=\lambda(s_T,T)
=
\lambda(s_T,\infty)-\beta_T
+
O(T^{-(2s_T-1)})\cdot\big(O(|s_T-1|)+O(T^{-(2s_T-1)})\big)=
\\
&
1+\delta(s_T-1)+O(|s_T-1|^2)-\beta_T
+
O(T^{-(2s_T-1)})\cdot\big(O(|s_T-1|)+O(T^{-(2s_T-1)})\big),
\end{align*}
that is
$$
s_T=1+\frac{\beta_T}{\delta}+
O(|s_T-1|^2)+O(T^{-(2s_T-1)})\cdot\big(O(|s_T-1|)+O(T^{-(2s_T-1)})\big).
$$
Recalling that $s_T>9/10$ for $T$ big enough, Lemma~\ref{LemmaExpansionBeta} implies 
\begin{align}
s_T
&
=
1+\frac{\beta}{\delta\cdot T}+
\left\{
\begin{array}{c}
O(|s_T-1|^2)+O(T^{-(2s_T-1)})\cdot\big(O(|s_T-1|)+O(T^{-(2s_T-1)})\big)\\
+O\big(|s_T-1|\cdot T^{-3/5}\big)+o(T^{-1}).
\end{array}
\right.
\\
&
\label{EquationRecursiveBoundDifferenceDimensionToOne}
=1+(\beta/\delta)\cdot T^{-1}+O(|s_T-1|^2)+O\big(|s_T-1|\cdot T^{-3/5}\big)+o(T^{-1}).
\end{align}

Let $\lambda>0$ as in Lemma~\ref{LemmaRowErrorDimension}. 
Equation~\eqref{EquationRecursiveBoundDifferenceDimensionToOne} implies  
$$
s_T-1=
(\beta/\delta)\cdot T^{-1}+O(T^{-2\lambda})+O\big(T^{-(\lambda+3/5)}\big)+o(T^{-1}).
$$
For any initial value of $\lambda$, applying recursively the last estimate, we get in finite time 
$$
s_T-1=(\beta/\delta)\cdot T^{-1}+o(T^{-1}).
$$ 
The explicit form of the constant $\Theta>0$ is 
\begin{equation}
\label{EquationFormulaConstantTheta}
\Theta
=
\frac{\beta}{-\delta}
=
\frac
{\lim_{T\to\infty} T\cdot\int\Delta_{(1,T)}g_{(1,\infty)}d\mu_{(1,\infty)}}
{-\int A_\infty g_{(1,\infty)}d\mu_{(1,\infty)}}.
\end{equation}
The proof of Theorem~\ref{TheoremDimensionShiftSpace} is complete. This completes the proof of 
Theorem~\ref{TheoremMainTheorem} too. \qed

\appendix

\section{A class of groups with ideal Ford domain}
\label{AppendixExample}

Let $P\subset\DD$ be an ideal polygon, that is a polygon with finitely many sides and all vertices in $\partial\DD$. Assume that a side $L$ of $P$ is a diameter of $\DD$, that is it contains the origin. Let $H$ be the group generated by the reflections in the sides of $P$ and let $\Gamma$ be the index two subgroup of $H$ of orientation preserving elements. Finally let $\sigma_L$ be the reflection with respect to the diameter $L$ of $P$ and set
$$
\Omega:=P\cup\sigma_L(P).
$$ 

\begin{lemma}
$\Gamma$ is a free non uniform Fuchsian lattice. Moreover $\Omega$ is an ideal Ford domain for $\Gamma$. The quotient $\DD\backslash\Gamma$ is a punctured sphere.
\end{lemma}

\begin{proof}
All angles at vertices of $P$ are equal to zero. Therefore $H$ is discrete, according to Theorem 7.1.3 in \cite{Ratcliffe}. Thus $\Gamma$ is discrete too. Label the other sides of $P$ as $s_a$ with $a\in\cA_0$, where $\cA_0$ is an alphabet with the same number of letters as the number of sides of $P$ minus one.  
The polygon $\sigma_L(P)$ is also ideal and shares with $P$ the diameter $L$ as a side. Then introduce an alphabet 
$\widehat{\cA_0}$ with the same number of letters as $\cA_0$ and label the sides of $\sigma_L(P)$ other than $L$ as $s_a$ with $a\in\widehat{\cA}_0$. Finally set 
$
\cA:=\cA_0\cup\widehat{\cA_0}
$ 
and consider the involution $\iota:\cA\to\cA$ such that, writing $\widehat{a}:=\iota(a)$, we have
$$
s_{\widehat{a}}=\sigma_L(s_a)
\quad\text{ for any }a\in\cA.
$$
The sides of $\Omega$ are labelled as $s_a$ with $a\in\cA$. For any such side let $\sigma_a$ be the reflection in $s_a$, then set 
\begin{equation}
\label{Equation(1)IdealFordDomain}
F_a:=\sigma_L\circ\sigma_a.
\end{equation}
According to Theorem 3.3.4 in~\cite{KatokFuchsian} any $F\in\sugroup(1,1)$ is the composition of the reflection in its isometric circle $I_F$ (see \S~\ref{SectionIsometricCircles}) and the reflection in the straight line $L_F$ through the center of $\DD$ and the midpoint between the poles $\omega_F$ and $\omega_{F^{-1}}$ of $F$ and $F^{-1}$ respectively. Therefore for any $a\in\cA$ the map $F_a$ in Equation~\eqref{Equation(1)IdealFordDomain} has isometric circle equal to $s_a$ (more precisely the circle extending $s_a$) and we have $F_a^{-1}=F_{\widehat{a}}$. 
By definition $\Gamma$ is generated by elements of the form $\sigma_1\circ\sigma_2$, with 
$$
\sigma_1,\sigma_2\in\{\sigma_L\}\cup\{\sigma_a;a\in\cA_0\}.
$$ 
For $a,b\in\cA_0$ we have 
$
\sigma_a\circ\sigma_b=(\sigma_L\circ\sigma_a)^{-1}\circ(\sigma_L\circ\sigma_b)
$, 
thus it is clear that $\Gamma$ is generated by the maps $F_a$ in Equation~\eqref{Equation(1)IdealFordDomain} with 
$a\in\cA$. The key observation is that the interiors of the isometric circles of these maps satisfy 
$U_a\cap U_b=\emptyset$ for $a\not=b$, where we denote $U_a$ the interior of the isometric circle of $F_a$, according to the notation in \S~\ref{SectionDistanceFromPoles}. Therefore consider $a_0,\dots,a_n$ so that the concatenation $F_{a_0}\circ\dots\circ F_{a_n}$ is possible, that is $a_{k+1}\not=\widehat{a_k}$ for any $k=1\dots,n-1$. Recalling that $U_{\widehat{a}}=F_a(\CC\setminus U_a)$ for any $a\in\cA$, Equation~\eqref{EquationInclusionsIsometricCirclesLetters} implies
$$
F_{a_0,\dots,a_{n}}(\CC\setminus U_{a_n})
=
F_{a_0,\dots,a_{n-1}}(U_{\widehat{a_n}})
\subset
U_{\widehat{a_0}}.
$$
Since $\Omega=\DD\setminus\bigcup_{a\in\cA}U_a$ then we get
\begin{equation}
\label{Equation(2)IdealFordDomain}
F_{a_0,\dots,a_{n}}(\Omega)\cap\Omega=\emptyset.
\end{equation}
It follows that the origin $0\in\DD$ is not the fixed point of any $F\in\Gamma$. Therefore, according to Theorem 3.3.5 in \cite{KatokFuchsian}, the Ford fundamental domain of $\Gamma$ is defined and given by
$$
\widetilde{\Omega}:=\overline{\DD\setminus\bigcup_{F\in\Gamma\setminus\{\id\}}U_F}.
$$
It is clear that $\widetilde{\Omega}\subset\Omega$. Since $\Omega$ has finite area, than $\widetilde{\Omega}$ has finite area too and $\Gamma$ is a (non uniform) lattice. Moreover Equation~\eqref{Equation(2)IdealFordDomain} implies that for any pair of points $z_1,z_2$ in the interior of $\Omega$ and any $F\in\Gamma\setminus\{\id\}$ we have $F(z_1)\not=z_2$. It follows that $\Omega=\widetilde{\Omega}$ is the Ford fundamental domain of $\Gamma$. We observe that $\Omega$ is an ideal polygon, and that its quotient by the maps $F_a$ in Equation~\eqref{EquationInclusionsIsometricCirclesLetters} is a punctured sphere. Finally for any $a\in\cA$ the element $F_a$ is not elliptic (because $U_a\cap U_{\widehat{a}}=\emptyset$) and we have 
$
F_a(\DD\setminus U_a)\subset U_{\widehat{a}}
$, 
therefore the \emph{Ping Pong Lemma} implies that $\Gamma$ is the free product of the infinite ciclic groups $\langle F_a\rangle$ with $a\in\cA_0$, thus $\Gamma$ is free. The Lemma is proved. 
\end{proof}

\section{Size and separation for horoballs at parabolic fixed points}
\label{AppendixSeparation}

Let $\Gamma$ be a Fuchsian group and $\cP_\Gamma$ be its set of parabolic fixed points. 
Let $P\in\Gamma$ be primitive parabolic fixing $\infty$.  By discreteness there exists a constant $c_\Gamma>0$ with 
\begin{equation}
\label{EquationAppendixEntryC}
|c(G)|\geq c_\Gamma
\quad
\textrm{ for any }
\quad
G\in\Gamma\setminus\langle P\rangle.
\end{equation}
For $T>0$ set $H_T:=\{z\in\HH:\im(z)>T\}$. 
Equation~\eqref{EquationAppendixEntryC} easily implies the next Lemma. 

\begin{lemma}
\label{LemmaAppendixDisjointHoroballs}
For all $T\geq c_\Gamma^{-1}$ and all $G_1,G_2\in\Gamma$ the following holds.
\begin{enumerate}
\item
If 
$
G_2^{-1}G_1\in\langle P\rangle
$ 
then $G_1(H_T)=G_2(H_T)$.
\item
Otherwise $G_1(H_T)\cap G_2(H_T)=\emptyset$.
\end{enumerate} 
\end{lemma}

For $T\geq c_\Gamma^{-1}$ the punctured disc 
$
\cU:=\langle P\rangle\backslash H_T
$ 
is isometrically embedded in $\Gamma\backslash\HH$ and gives a so-called \emph{Margulis neighbourhood} of the cusp $[\infty]$. Lemma~\ref{LemmaAppendixDisjointHoroballs} implies:

\begin{corollary}
\label{CorollaryParabolicPointsOnlyFixedByParabolicElement}
If $G\in\Gamma$ satisfies $G\cdot\zeta=\zeta$ for some $\zeta\in\cP_\Gamma$, then 
$G$ is parabolic.
\end{corollary}

Fix a non-uniform lattice $\Gamma$ and $\cS=\{A_1,\dots,A_p\}$ as in 
Equation~\eqref{EquationRepresentativesOfCusps}. For any $k$ set 
$
\Gamma_k:=A_k^{-1}\Gamma A_k
$. 
Observe that $G\in\Gamma$ fixes 
$z_k=A_k\cdot\infty$ if and only if 
$
A_k^{-1}GA_k\in\Gamma_k
$ 
fixes $\infty$. Let $P_k\in\Gamma_k$ be primitive parabolic fixing $\infty$ and $T_k>0$ be such that 
$
\langle P_k\rangle\backslash H_{T_k}
$ 
is a Margulis neighbourhood of the cusp $[\infty]$ in 
$
\Gamma_k\backslash\HH
$. 
The map $A_k:\HH\to\HH$, $z\mapsto A_k\cdot z$ induces a map 
$$
A_k:\Gamma_k\backslash\HH\to\Gamma\backslash\HH
\quad,\quad
\Gamma_kz\mapsto \Gamma A_k\cdot z,
$$
and a Margulis neighbourhood of the cusp $[z_k]$ in $\Gamma\backslash\HH$ is given by the image  
$$
\cU_k:=A_k\big(\langle P_k\rangle\backslash H_{T_k}\big)=
\langle A_kP_kA_k^{-1}\rangle\backslash (A_k H_{T_k}).
$$ 
Set 
$
\widetilde{\cU}_k:=\langle A_kP_kA_k^{-1}\rangle\backslash (A_k H_{S_0})
$ 
for $k=1,\dots,p$, where
$$
S_0\geq\max\{T_1,\dots T_p\}
$$
is a constant big enough so that the Margulis neighbourhoods 
$
\widetilde{\cU}_1,\dots,\widetilde{\cU}_p
$ 
are mutually disjoint in $\Gamma\backslash\HH$. 
By Lemma~\ref{LemmaAppendixDisjointHoroballs}, the balls in the family 
$\{GA_k(H_{S_0}):G\in\Gamma\}$ are mutually disjoint for any fixed $k$. Hence we get disjointness for all $k=1,\dots,p$, that is 
\begin{equation}
\label{EquationAppendixDisjointHoroballs}
GA_k(H_{S_0})\cap FA_j(H_{S_0})\not=\emptyset
\quad
\Rightarrow
\quad
G=F
\textrm{ and }
j=k.
\end{equation}

Let $c=c(G)$ and $d=d(G)$ be the entries of $G\in\sltwor$ as in 
Equation~\eqref{EquationCoefficientsSL(2,C)}. If $c(G)\not=0$ then $G(H_T)$ is a ball tangent to the real line at $x=G\cdot\infty$. For $z=\alpha+iT$ we have  
$$
\im(G\cdot z)=
\frac{\im(z)}{|cz+d|^2}=
\frac{T}{|ciT+c\alpha+d|^2}=
\frac{T}{|c|^2T^2+|c\alpha+d|^2}.
$$
The expression above is maximal for $\alpha=-d/c$, thus the diameter of $G(H_T)$ is given by 
\begin{equation}
\label{EquationDiameterHoroball}
\diameter\big(G(H_T)\big)=\frac{1}{T\cdot c^2(G)}.
\end{equation}

If $B_1,B_2$ are disjoint balls in $\HH$, tangent to $\RR$ at $x_1,x_2$ and with radii $R_1,R_2$ respectively, we have 
$
|x_2-x_1|^2\geq(R_1+R_2)^2-(R_1-R_2)^2=4R_1R_2
$, 
where the equality holds if $B_1,B_2$ are mutually tangent. 
Recall Equation~\eqref{EquationDefinitionDenominator} for the denominator. 
Equation~\eqref{EquationDiameterHoroball} and 
Equation~\eqref{EquationAppendixDisjointHoroballs} imply that if $G\cdot z_i\not=F\cdot z_j$ are two different parabolic fixed points of $\Gamma$, then we have 
\begin{equation}
\label{EquationDistanceDisjointHoroballs}
|G\cdot z_i-F\cdot z_j|
\geq
\frac{1}{S_0}\cdot\frac{1}{D(G\cdot z_i)D(F\cdot z_j)}.
\end{equation}

\section{Some basic facts on spectra and projectors}
\label{AppendixSpectralProjectors}

We recall some facts on spectral properties of bounded linear operator. We follow \S~III.6 in \cite{Kato}.

\subsection{Spectrum and spectral projectors}
\label{AppendixDefinitionSpectralProjectors}

Let $(B,\|\cdot\|)$ be a Banach space and $L:B\to B$ be a bounded linear operator. The \emph{resolvent set} of $L$ is the set $\cR(L)$ of complex numbers $z\in\CC$ such that 
$
L-z\id:B\to B
$ 
is an invertible operator with bounded inverse. The set  
$
\spectrum(L):=\CC\setminus\cR(L)
$ 
is called the \emph{spectrum} of $L$. The \emph{spectral radius} $\rho(L)$ of $L$ is defined by
$$
\rho(L):=\sup\{|z|:z\in\spectrum(L)\}.
$$ 

We have $\rho(L)\leq\|L\|$. The spectrum $\spectrum(L)$ of a bounded linear operator $L$ is always a compact subset of $\CC$. The resolvent set $\cR(L)$ is thus open and we have the \emph{resolvent map}
$$
R(L,\cdot):\cR(L)\to\cL(B,B)
\quad;\quad
R(L,z):=(L-z\id)^{-1},
$$ 
where $\cL(B,B)$ denotes the space of bounded linear maps from $B$ to $B$. For $z\in\cR(L)$ the bounded linear operator $R(L,z):B\to B$ is called the resolvent of $L$. The resolvent map is holomorphic (in the terminology of \cite{Kato}, page 37), meaning that for any $z\in\cR(L)$ and any $h\in\CC$ small enough we have 
$$
R(L,z+h)=R(L,z)+h\cdot R^2(L,z)+ho(h).
$$
This follows from the \emph{first resolvent equation}: for any  
$
z,z'\in\cR(L)
$ 
we have 
$$
R(L,z)-R(L,z')=(z-z')R(L,z)R(L,z').
$$

\medskip

Assume that $\Sigma,\Sigma'$ are two compact subsets of $\CC$ with 
$\Sigma\cap\Sigma'=\emptyset$ and $\spectrum(L)=\Sigma\cup\Sigma'$. Then assume also that we have a smooth oriented loop $\gamma$ in $\cR$ containing $\Sigma$ in its interior and having $\Sigma'$ in the exterior. Then the \emph{spectral projector} for $\Sigma$ is the linear operator let $P:B\to B$ defined by 
\begin{equation}
\label{EquationDefinitionSpectralProjector}
P:=\frac{-1}{2\pi i}\int_\gamma R(L,z)dz.
\end{equation}

Since $z\mapsto R(L,z)$ is holomorphic, the map $P:B\to B$ does not depend on the specific choice of 
$\gamma$. Moreover $P$ is a projection, that is $P^2=P$, we have 
$P\circ L=L\circ P$, and $P$ is bounded. More precisely the continuity of the resolvent map implies  
$
M:=\sup_{z\in\gamma}\|R(L,z)\|<+\infty
$, 
and thus $\|P\|\leq |\gamma|\cdot M$. Obviously $(\id-P):B\to B$ is also a projection which commutes with $L$, with norm satisfying
$\|\id-P\|\leq1+\|P\|$. Thus setting $N:=\ker(P)$ and $V:=\ker(\id-P)$ we have
$$
B=N\oplus V
\quad
\textrm{ with }
\quad
L(N)\subset N
\quad
\textrm{ and }
\quad
L(V)\subset V.
$$

For $z\in\cR(L)$ the map $R(L,z):\cB\to\cB$ is bijective continuous and satisfies the relation    
$P\circ R(L,z)=R(L,z)\circ P$. Hence the restrictions $R(L,z)|_N:N\to N$ and $R(L,z)|_V:V\to V$ are bijective continuous and it follows that 
$$
R(L|_N,z)=R(L,z)|_N
\quad
\textrm{ and }
\quad
R(L|_V,z)=R(L,z)|_V 
\quad
\textrm{ for any  }
\quad
z\in\cR(L).
$$ 
One can prove that $\spectrum(L|_N)=\Sigma'$ and $\spectrum(L|_V)=\Sigma$. 
Details are in \S~III.6.4 in \cite{Kato}.

\subsection{Stability of spectral decompositions}
\label{AppendixStabilitySpectralDecomposition}

Let $L_0,A:B\to B$ be bounded linear operators. Assume that $L_0$ is invertible and that 
$
\|A\|\cdot \|L_0^{-1}\|<1
$. 
Then $\|A\circ L_0^{-1}\|<1$ and thus $-1\in\cR(A\circ L_0^{-1})$, so that 
$(\id+A\circ L_0^{-1}):B\to B$ is bounded invertible. It follows that $L_0+A$ is bounded invertible, indeed 
$$
L_0+A=(\id+A\circ L_0^{-1})\circ L_0.
$$
Let $L:B\to B$ be a bounded linear operator. Apply the construction above to $A:=L-L_0$. In particular, if $z\in\cR(L_0)$ and $\|L-L_0\|\cdot\|R(L_0,z)\|<1$, then $z\in\cR(L)$ too. If 
$\gamma$ is a closed smooth path in $\cR(L_0)$, then continuity of $z\mapsto R(L_0,z)$ implies  
$$
M_0:=\sup_{z\in\gamma}\|R(L_0,z)\|<+\infty.
$$
Therefore $\gamma\subset\cR(L_0)\cap\cR(L)$ provided that $\|L-L_0\|\leq M_0^{-1}$. Furthermore, it is easy to check that for any 
$z\in\cR(L_0)\cap\cR(L)$ we have
$$
(L_0-z\id)(L-z\id)\big(R(L_0,z)-R(L,z)\big)
=
(L-L_0)+(L_0L-LL_0)R(L_0,z).
$$
From  
$
L_0L-LL_0=L_0L-L_0^2+L_0^2-LL_0
$ 
we get 
$
\|L_0L-LL_0\|\leq 2\|L_0\|\cdot\|L-L_0\|
$, 
so that 
$$
\|R(L_0,z)-R(L,z)\|
\leq
\|R(L,z)\|\cdot\|R(L_0,z)\|\cdot(1+2\|L_0\|\cdot\|R(L_0,z)\|)\cdot\|L-L_0\|.
$$
Moreover, if $\|L-L_0\|<(2M_0)^{-1}$, then for any $z\in\gamma$ and any $v\not=0$ we have 
$$
\|(L-z\id)(v)\|
\geq 
\|(L_0-z\id)(v)\|-\|(L-L_0)(v)\|
\geq
\frac{\|v\|}{\|R(L_0,z)\|}-\frac{\|v\|}{2M_0}
\geq 
\frac{\|v\|}{2M_0},
$$
that is 
$
\sup_{z\in\gamma}\|R(L,z)\|\leq 2M_0
$. 
Resuming, if $\|L-L_0\|<(2M_0)^{-1}$, then we have 
$$
\sup_{z\in\gamma}
\|R(L_0,z)-R(L,z)\|
\leq
M_1\cdot\|L-L_0\|,
\quad\text{ where }\quad
M_1:=2M_0^2\cdot(1+2M_0\cdot\|L_0\|).
$$
Therefore if $P_0$ and $P$ are the spectral projectors along $\gamma$ associated respectively to $L_0$ and to $L$ in Equation~\ref{EquationDefinitionSpectralProjector}, then we have 
\begin{equation}
\label{EquationDistanceSpectralProjectors}
\|P-P_0\|\leq |\gamma|\cdot M_1\cdot\|L-L_0\|.
\end{equation}
Finally, we have $\|P-P_0\|<1$ for $\|L-L_0\|$ small enough. In this case there exists an isomorphism $G:B\to B$ such that $P=GP_0G^{-1}$. See \cite{Kato}, pages 33-34. In particular $G$ sends the spectral decomposition $B=N_0\oplus V_0$ onto the spectral decomposition $B=N\oplus V$, thus $\dim N=\dim N_0$ and $\dim V=\dim V_0$. Moreover we have
\begin{equation}
\label{EquationDistanceBaseChanges}
G-\id=O(\|P_0\|\cdot\|P-P_0\|)
\quad\textrm{ and }\quad 
G^{-1}-\id=O(\|P_0\|\cdot\|P-P_0\|).
\end{equation}

\subsection{Quasi compact operators}
\label{AppendixQuasiCompactOperators}

Recall from \cite{Hennion} the notion of \emph{quasi compact operator}. The essential spectral radius 
$\rho_{ess}(L)$ of a bounded linear operator $L:B\to B$ is the infimum of those $r\geq0$ such that there exists subspaces $N,V$ of $B$ such that the following holds.
\begin{enumerate}
\item
We have $B=N\oplus V$ with $L(N)\subset N$ and $L(V)\subset V$.
\item
We have $1\leq \dim V<+\infty$ and the restriction $L|_V$ has only eigenvalues $\lambda$ of modulus 
$|\lambda|\geq r$.
\item
The subspace $N$ is closed and the restriction has spectral radius $\rho(L|_N)<r$.
\end{enumerate}
The operator $L:B\to B$ is said \emph{quasi compact} if $\rho_{ess}(L)<\rho(L)$ strictly.

\end{document}